\documentclass[12pt]{amsart}

\usepackage[left=2cm,top=2cm,right=2cm,bottom = 2cm]{geometry}

\usepackage[plainpages,urlcolor=red,linktocpage=true]{hyperref}
\usepackage{amscd}
\usepackage{amssymb}
\usepackage{mathrsfs}
\usepackage{amsmath}
\usepackage{mathabx}
\usepackage{amsthm}
\usepackage[all,cmtip]{xy}
\usepackage{enumerate}
\usepackage{enumitem}
\usepackage{ytableau}
\usepackage{tikz}
\usepackage{mathtools}
\usepackage{color}
\usepackage{wasysym}

\usepackage{charter}%font
\usepackage{comment}

\theoremstyle{plain}
\newtheorem{thm}{Theorem}[section]
\newtheorem{prop}[thm]{Propsition}
\newtheorem{ques}[thm]{Question}
\theoremstyle{definition}
\newtheorem{defn}[thm]{Definition}
\theoremstyle{remark}
\newtheorem{rmk}[thm]{Remark}

\theoremstyle{lemma}
\newtheorem{lem}[thm]{Lemma}
\theoremstyle{Corollary}
\newtheorem{coro}[thm]{Corollary}
\theoremstyle{example}

\newcommand{\thmref}[1]{Theorem~\ref{#1}}
\newcommand{\secref}[1]{Section~\ref{#1}}
\newcommand{\appref}[1]{Appendix~\ref{#1}}
\newcommand{\lemref}[1]{Lemma~\ref{#1}}
\newcommand{\propref}[1]{Proposition~\ref{#1}}
\newcommand{\corref}[1]{Corollary~\ref{#1}}
\newcommand{\defref}[1]{Definition~\ref{#1}}

\newcommand{\rmkref}[1]{Remark~\ref{#1}}

\newcommand{\be}{\begin{equation}}
\newcommand{\ee}{\end{equation}}

\newcommand{\mc}{\mathcal}

\newcommand{\C}{{\mathbb C}}
\newcommand{\R}{{\mathbb R}}
\newcommand{\Z}{{\mathbb Z}}

\newcommand{\CE}{{\mathcal E}}

\newcommand{\CG}{{\mathcal G}}
\newcommand{\CP}{{\mathcal P}}
\newcommand{\CO}{{\mathcal O}}
\newcommand{\CS}{{\mathcal S}}
\newcommand{\CT}{{\mathcal T}}
\newcommand{\CH}{{\mathcal H}}

\newcommand{\CK}{{\mathcal K}}
\newcommand{\CL}{{\mathcal L}}
\newcommand{\CX}{{\mathcal X}}

\newcommand{\Ker}{{\rm{Ker}}}
\newcommand{\LHS}{{\rm{LHS}}}

\newcommand{\id}{{\rm{id}}}
\newcommand{\sgn}{{\rm{sgn}}}

\newcommand{\mf}{\mathfrak}
\newcommand{\fg}{{\mf g}}
\newcommand{\fh}{{\mf h}}
\newcommand{\fb}{{\mf b}}
\newcommand{\fn}{{\mf n}}

\newcommand{\ft}{{\mathfrak t}}

\newcommand{\fM}{{\mf M}}

\newcommand{\fR}{{\mf R}}

\newcommand{\cR}{\mc R}

\newcommand{\cP}{\mc P}

\newcommand{\bm}{\mathbf m}
\newcommand{\bn}{\mathbf n}

\newcommand{\bI}{\mathbf I}

\newcommand{\La}{\Lambda}
\newcommand{\la}{\lambda}

\newcommand{\de}{\delta}

\newcommand{\cM}{\mc M}

\newcommand{\U}{{\rm{U}}}

\newcommand{\End}{{\rm{End}}}

\newcommand{\Sym}{{\rm{Sym}}}

\newcommand{\Ad}{{\rm{Ad}}}

\advance\headheight by 2pt

\newcommand{\ot}{\otimes}

\newcommand{\fsl}{{\mathfrak {sl}}}

\newcommand{\gl}{{\mathfrak {gl}}}
\newcommand{\fgl}{{\mathfrak {gl}}}
\newcommand{\osp}{{\mathfrak {osp}}}

\newcommand{\GL}{{\rm{GL}}}

\newcommand{\Span}{{\rm Span}}

\newcommand{\Uqg}{{\rm{U}}_q(\gl_{m|n})}
\newcommand{\Uq}{{\rm{U}}_q}

\numberwithin{equation}{section}
\makeatletter
\def\moverlay{\mathpalette\mov@rlay}
\def\mov@rlay#1#2{\leavevmode\vtop{%
   \baselineskip\z@skip \lineskiplimit-\maxdimen
   \ialign{\hfil$\m@th#1##$\hfil\cr#2\crcr}}}
\newcommand{\charfusion}[3][\mathord]{
    #1{\ifx#1\mathop\vphantom{#2}\fi
        \mathpalette\mov@rlay{#2\cr#3}
      }
    \ifx#1\mathop\expandafter\displaylimits\fi}
\makeatother

\newcommand{\overbar}[1]{\mkern1.5mu\overline{\mkern-1.5mu#1\mkern-1.5mu}\mkern 1.5mu}

\begin{document}

\title[Invariant theory for $\Uqg$]{The first and second fundamental theorems\\ of  invariant theory for \\
the quantum general linear supergroup}
\author[Yang Zhang]{Yang Zhang}
\dedicatory{Dedicated to Professor Gus Lehrer on the occasion of his 70th birthday}
\date{\today}
%\thanks{This research was supported by}
\address{School of Mathematical Sciences,
	University of Science and Technology of China, Hefei 230026, China}
\address{School of Mathematics and Statistics, University of Sydney, Sydney, NSW 2006, Australia}
\email{yang91@mail.ustc.edu.cn}

\begin{abstract}
We develop the non-commutative polynomial version of the invariant theory for the quantum general linear supergroup ${\rm{ U}}_q(\mathfrak{gl}_{m|n})$. A non-commutative ${\rm{ U}}_q(\mathfrak{gl}_{m|n})$-module superalgebra $\mathcal{P}^{k|l}_{\,r|s}$ is constructed, which is the quantum analogue of the supersymmetric algebra  over $\mathbb{C}^{k|l}\otimes \mathbb{C}^{m|n}\oplus \mathbb{C}^{r|s}\otimes (\mathbb{C}^{m|n})^{\ast}$.  We analyse the structure of the subalgebra of ${\rm{ U}}_q(\mathfrak{gl}_{m|n})$-invariants in $\mathcal{P}^{k|l}_{\,r|s}$ by using the quantum super analogue of Howe duality.

The subalgebra of ${\rm{ U}}_q(\mathfrak{gl}_{m|n})$-invariants in $\mathcal{P}^{k|l}_{\,r|s}$ is shown to be finitely generated.  We determine its generators and establish a surjective superalgebra homomorphism from a braided supersymmetric algebra onto it.  This establishes the first fundamental theorem of invariant theory for ${\rm{ U}}_q(\mathfrak{gl}_{m|n})$.  

We show that the above mentioned  superalgebra homomorphism is an isomorphism if and only if $m\geq \min\{k,r\}$  and $n\geq \min\{l,s\}$, and obtain a PBW basis for the subalgebra of invariants in this case.  
When the homomorphism is not injective, we give a  representation theoretical description of the generating elements of the kernel.  This way we obtain the relations obeyed by the generators of the subalgebra of invariants, producing
the second fundamental theorem of invariant theory for ${\rm{ U}}_q(\mathfrak{gl}_{m|n})$.  

We consider the special case with $n=0$ in greater detail, obtaining a complete treatment of the non-commutative polynomial version of the invariant theory for ${\rm{ U}}_q(\mathfrak{gl}_{m})$. In particular, the explicit SFT proved here is believed to be new.  We also recover the FFT and SFT of the invariant theory for the general linear superalgebra from the classical  limit (i.e., $q\to 1$) of our results.	
\end{abstract}


\maketitle
\tableofcontents
\section{Introduction}
\noindent 1.1.\ \
%\subsection{ }
Let $\C^{m|n}$ be the natural module for the general linear Lie superalgebra $\gl_{m|n}$, and let $(\C^{m|n})^{\ast}$ be its dual. Denote by $\CS^{k|l}_{\,r|s}$ the supersymmetric algebra over $\C^{k|l}\ot \C^{m|n}\oplus \C^{r|s}\ot (\C^{m|n})^{\ast}$, which is isomorphic to $S(\C^{k|l}\ot \C^{m|n})\ot S(\C^{r|s}\ot (\C^{m|n})^{\ast})$. 
One formulation of the invariant theory of $\gl_{m|n}$  seeks to describe the subalgebra of  $\gl_{m|n}$-invariants of $\CS^{k|l}_{\,r|s}$. 
The first fundamental theorem (FFT) provides a finite set of generators for the subalgebra of invariants \cite{DLZ, LZ1, S1}, and the second fundamental theorem (SFT) describes the relations among the generators \cite{LZ2, S2}.

The aim  of this paper is to develop quantum analogues of FFT and SFT of  invariant theory for the quantum general linear supergroup $\Uq(\gl_{m|n})$.  The invariant theory of $\Uqg$ in this non-commutative algebraic setting was poorly understood previously. In fact, even the SFT of invariant theory for $\Uq(\gl_m)$ (the case $n=0$) in this setting was unknown. The main issue hindering progress is that (super) commutative algebraic techniques used in classical invariant theory are no longer applicable to the quantum case, as all algebras involved now are not commutative.  

In this paper, we construct a quantum analogue $\CP^{k|l}_{\,r|s}$
of $\CS^{k|l}_{\,r|s}$,  which forms a module superalgebra (defined in \secref{secmodalg}) over the quantum general linear supergroup.  We investigate the subalgebra of $\Uqg$-invariants of $\CP^{k|l}_{\,r|s}$, which is again non-commutative. 
We construct a finite set of  generators for the subalgebra of invariants,  and determine the algebraic relations obeyed by the generators.  These results amount to an FFT and SFT, which are respectively given in \thmref{thmFFT} (also see \thmref{FFTref}) and \thmref{thmSFT}.

\vskip 0.4cm
\noindent 1.2.\ \
The first problem we need to address is the following, which is absent in the classical (i.e., non-quantum) case.
\begin{ques}\label{q:q-poly}
Find an appropriate quantum polynomial superalgebra $\CP^{k|l}_{\,r|s}$ which has $\CS^{k|l}_{\,r|s}$ as classical limit ($q\rightarrow 1$). 
\end{ques}
As discussed below, the problem in itself is highly nontrivial. 
We address it following ideas of reference \cite{BZ}, where 
braided symmetric algebras were taken as quantum polynomial algebras. 
We  generalise this notion to the super setting. 

Given a $\Z_2$-graded vector space $W$, the braided supersymmetric algebra $S_q(W)$ (defined in \secref{secbra}) is viewed as the $q$-deformation of supersymmetric algebra $S(W)$, i.e., has $S(W)$ as classical limit ($q\rightarrow 1$). Recall that $S(W)$ and $S_q(W)$ are both $\Z_{+}$-graded. We say that $S_q(W)$ is a flat deformation of $S(W)$ if $\dim S(W)_N= \dim S_q(W)_N$ for all integers $N\geq 0$. Furthermore, if $W$ is a module  over a quantum general linear supergroup and $S_q(W)$ is a flat deformation, then $W$ is called a flat module. However, a well known fact is that even in the quantum group case almost  all the braided symmetric algebras are not flat deformations (cf. \cite{BZ,LZZ}), that is, they are ``smaller'' than the corresponding polynomial algebras.  Little is known about braided supersymmetric algebras in the quantum super case. Here we will construct some flat deformations, which are useful for the study of invariant theory.

Let $V^{m|n}$ be the natural module for $\Uqg$, and similarly introduce $V^{k|l}$ and $V^{r|s}$. Our preceding discussion suggests that  $\CP^{k|l}_{\,r|s}$ may be defined as the tensor product of $S_q(V^{k|l}\ot V^{m|n})$ and $S_q(V^{r|s}\ot (V^{m|n})^{\ast})$. Indeed, with some effort we will show that $\CP^{k|l}_{\,r|s}$ defined in this way is a flat deformation of $\CS^{k|l}_{\,r|s}$. Yet the algebra structure of $\CP^{k|l}_{\,r|s}$ is more subtle and needs to be defined properly.

The notion of  module superalgebra is therefore brought into our picture, which is a generalisation of module algebra in the sense of \cite[\S 4.1]{M}. For instance, in our case $S_q(V^{k|l}\ot V^{m|n})$ is  an associative algebra carrying  $\Uq(\gl_{k|l})\ot \Uqg$-module structure, whose algebraic structure is preserved by the quantum supergroup action, and similarly for $S_q(V^{r|s}\ot (V^{m|n})^{\ast})$. A useful observation originating from Hopf algebra theory  (cf. \cite{S,LZZ}) is that the tensor product of two module superalgebras is again a module superalgebra, with the multiplication defined in \propref{proptensormodalg}. This implies that $\CP^{k|l}_{\,r|s}$ is a module superalgebra, on which the two actions of $\Uq(\gl_{k|l})\ot \Uq(\gl_{r|s})$ and $\Uqg$ graded-commute with each other.  The subspace of $\Uqg$-invariants in $\CP^{k|l}_{\,r|s}$ is a subalgebra. Hence the  ``generators'' of $\Uqg$-invariant subalgebra of $\CP^{k|l}_{\,r|s}$ (i.e. FFT)  makes sense in this context.

However, the multiplication in $\CP^{k|l}_{\,r|s}$ at the present stage is quite unwieldy to use,  and furthermore, the $\Uqg$-action on $\CP^{k|l}_{\,r|s}$ is non-semisimple. This makes it highly nontrivial to describe the invariants. To overcome these difficulties, we shall use known results  on coordinate superalgebras \cite{Z98} of quantum general linear supergroups  following ideas of \cite{LZZ}. This enables us to give a  new formulation of $\CP^{k|l}_{\,r|s}$ in terms  of generators and relations, and also characterise the  module structure on  $\CP^{k|l}_{\,r|s}$ by using quantum Howe duality of type $(\Uq(\gl_{k|l}),\Uq(\gl_{r|s}))$.

\vskip 0.4cm
\noindent 1.3.\ \
We now briefly describe this new construction of $\CP^{k|l}_{\,r|s}$. Let $\pi: \Uqg\rightarrow \End(V^{m|n})$ be the natural representation. We recall from \cite{Z98} the coordinate superalgebra $\cM_{m|n}$ of the finite dual of $\Uqg$, which is generated by the matrix elements $t_{ab}$ defined by \eqref{eqmatelmt} and encodes the  bi-superalgebra structure. By applying truncation procedure, we obtain the subalgebra $\cM^{k|l}_{r|s}$ of $\cM_{m|n}$ with $k, r\le m$ and $l, s\le n$, which  is a module superalgebra over $\Uq(\gl_{k|l})\ot \Uq(\gl_{r|s})$. Conversely, for any non-negative integers $k,l,r,s$ we can define module superalgebra $\cM^{k|l}_{r|s}$ analogously and embed it into the coordinate superalgebra $\cM_{K|L}$ with $K=\max\{k,r\}$ and $L=\max\{l,s\}$.
Direct calculation shows that $\cM^{k|l}_{r|s}$ is isomorphic to $S_q(V^{k|l}\ot V^{r|s})$ as module superalgebra. All  these arguments go through for $\overbar{\cM}_{m|n}$ and  $\overbar{\cM}^{k|l}_{r|s}$, which are both generated by some matrix elements of the dual module $(V^{m|n})^{\ast}$,  and the latter  is isomorphic to the braided supersymmetric algebra $S_q((V^{k|l})^{\ast}\ot (V^{r|s})^{\ast})$. Now we can define $\CP^{k|l}_{\,r|s}$ as the tensor product of $\cM^{\;k|l}_{m|n}$ and $\overbar{\cM}^{\;r|s}_{m|n}$ with a presentation shown in  \lemref{lemmultP}. 

To prove the flatness of $\CP^{k|l}_{\,r|s}$ and  analyse its module structure, we make extensive use of quantum Howe duality. It was noted in \cite{Z98} that both $\cM_{m|n}$ and $\overbar{\cM}_{m|n}$ admit multiplicity-free decompositions as $\Uqg\ot \Uqg$-modules by a partial analogue of quantum Peter-Weyl theorem. We apply truncation procedure to $\cM_{m|n}$ (resp. $\overbar{\cM}_{m|n}$), producing a  multiplicity-free decomposition of the subalgebra $\cM^{k|l}_{r|s}$  (resp. $\overbar{\cM}^{k|l}_{r|s}$) as  $\Uq(\gl_{k|l})\ot \Uq(\gl_{r|s})$-module with $k, r\le m$ and $l, s\le n$. This is called quantum Howe duality of type $(\Uq(\gl_{k|l}),\Uq(\gl_{r|s}))$  (see \thmref{thmHowe}, and also \cite[Theorem 2.2]{WZ}), where $k,l,r,s$ can actually be any non-negative integers. Using this, we obtain the following  results: 
\begin{enumerate}
\item $\CP^{k|l}_{\,r|s}$ is a flat deformation of $\CS^{k|l}_{\,r|s}$; 
\item  A PBW basis for $\cM^{k|l}_{r|s}$ is constructed;
\item The quantum Howe duality on $\cM^{k|l}_{r|0}$ implies the quantum Schur-Weyl duality between $\Uq(\gl_{k|l})$ and the Hecke algebra $\CH_q(r)$. 
\end{enumerate}
In particular, we show that there are three 
flat $\Uq(\gl_{k|l})\ot\Uq(\gl_{r|s})$-modules $V^{k|l}\ot V^{r|s}$, $(V^{k|l})^{\ast}\ot (V^{r|s})^{\ast}$ and $V^{k|l}\ot (V^{r|s})^{\ast}$.

\vskip 0.4cm
\noindent 1.4.\ \
Now the following problem naturally arises.
\begin{ques}\label{q:fts}
Describe the invariant subalgebra $\CX^{k|l}_{\;r|s}:=(\CP^{k|l}_{\,r|s})^{\Uqg}$ in terms of generators (FFT) and defining relations (SFT). In particular, determine whether $\CX^{k|l}_{\;r|s}$ is a quantum polynomial superalgebra or a quotient thereof.
\end{ques}
As we have mentioned already, the $\Uqg$-invariant subalgebra $\CX^{k|l}_{\;r|s}$ is non-commutative;  there exist no results which can be readily applied to show that it is a (quotient of a)  quantum polynomial superalgebra in the sense of Question \ref{q:q-poly}.

Motivated by the approach in  \cite{LZZ}, we show that the invariant subalgebra $\CX^{k|l}_{\;r|s}$ is finitely generated, and construct the invariants explicitly using bi-superalgebra structure of the coordinate superalgebra. This amounts to the FFT; see  \thmref{thmFFT}. We want to mention that, due to the non-commutative nature of quantum polynomial superalgebra $\CP^{k|l}_{\,r|s}$, techniques from classical invariant theory based on commutative algebra fail in our case, especially those techniques addressing finite generation.

To fully explore the algebraic structure of $\CX^{k|l}_{\;r|s}$, we give a reformulation of FFT. Some elementary quadratic relations among invariants from $\CX^{k|l}_{\;r|s}$ are obtained in \lemref{leminvrel}, by which we are inspired to introduce an auxiliary  quadratic superalgebra $\widetilde{\cM}^{k|l}_{r|s}$ with the same relevant quadratic relations (defined in \secref{secrefFFT}). This quadratic superalgebra $\widetilde{\cM}^{k|l}_{r|s}$ is shown to be isomorphic to the braided supersymmetric algebra  $S_q(V^{k|l}\ot (V^{r|s})^{\ast})$ as superalgebra, which similarly admits quantum Howe duality and is literally a flat deformation. The FFT is then reformulated in terms of the surjective superalgebra homomorphism $\Psi^{k|l}_{r|s}: \widetilde{\cM}^{k|l}_{r|s}\twoheadrightarrow \CX^{k|l}_{\;r|s}$ in \thmref{FFTref}. This implies that the invariant subalgebra $\CX^{k|l}_{\;r|s}$ is the quotient of the quantum polynomial superalgebra $\widetilde{\cM}^{k|l}_{r|s}$ by the two-sided ideal $\Ker\, \Psi^{k|l}_{r|s}$.

The SFT of invariant theory now seeks to describe the kernel of $\Psi^{k|l}_{r|s}$, as the images of nonzero elements in kernel give rise to non-elementary relations among invariants. This can be reduced to the following situation. Note that $\widetilde{\cM}^{k|l}_{r|s}$ can be embedded into a larger superalgebra $ \widetilde{\cM}_{K|L}:=\widetilde{\cM}^{K|L}_{K|L}$ with $K=\max\{k,r\}$ and $L=\max\{l,s\}$. We are led to consider the kernel of $\Psi_{K|L}: \widetilde{\cM}_{K|L}\twoheadrightarrow \CX^{K|L}_{\;K|L}$, since the restriction $\Ker\,\Psi_{K|L}\cap \widetilde{\cM}^{k|l}_{r|s}$ exactly coincides with $\Ker\, \Psi^{k|l}_{r|s}$.  Using quantum Howe duality, we show that $\Ker\, \Psi_{K|L}$ as a $\Uq(\gl_{K|L})\ot \Uq(\gl_{K|L})$-module admits multiplicity-free decomposition over all $(K|L)$-hook partitions containing  partition $\la_c=((n+1)^{m+1})$; this is a special case of \corref{coroXdec}. This module structure can be used to characterise  $\Ker\, \Psi_{K|L}$ as a two-sided ideal of $\widetilde{\cM}_{K|L}$, which we shall explain below.
 
\vskip 0.4cm
\noindent 1.5.\ \
Now our method relies essentially on some favourable properties of matrix elements of tensor representations of the quantum supergroup, by relating $\widetilde{\cM}_{K|L}$ to the coordinate superalgebra $\cM_{K|L}$. By a partial analogue of the Peter-Weyl theorem,  $\cM_{K|L}$  decomposes into a direct sum of subspaces $T_{\la}$, which are spanned by  matrix elements of the simple tensor modules  $L_{\la}^{K|L}$ for $\Uq(\gl_{K|L})$.  We prove in \thmref{thmidealdec} that the two-sided ideal in $\cM_{K|L}$ generated by a single subspace $T_{\la}$ admits a multiplicity-free decomposition into direct sum of $T_{\gamma}$ over all $(K,L)$-hook partitions $\gamma$ containing $\la$.  A key observation is that  this result can be translated to the quadratic superalgebra  $\widetilde{\cM}_{K|L}$, though it is generally not isomorphic to $\cM_{K|L}$ as superalgebra.  Consequently, the multiplicity-free decomposition  of $\Ker\, \Psi_{K|L}$ as a module indicates that $\Ker\, \Psi_{K|L}$ is indeed generated by the direct summand $\widetilde{T}_{\la_c}$ corresponding to  the minimal partition $\la_c=((n+1)^{m+1})$ as  a two-sided ideal of $\widetilde{\cM}_{K|L}$. Here $\widetilde{T}_{\la}\subset \widetilde{\cM}_{K|L}$ is an analogue of $T_{\la}$, which is isomorphic to $L_{\la}^{K|L}\ot L_{\la}^{K|L}$.

Our SFT of invariant theory asserts that $\Ker \, \Psi^{k|l}_{r|s}$ as a two-sided ideal of $\widetilde{\cM}^{k|l}_{r|s}$ is generated by the subspace $\widetilde{T}_{\la_c}\cap\widetilde{\cM}^{k|l}_{r|s}$ with $\la_c=((n+1)^{m+1})$. In particular, we show that $\Psi^{k|l}_{r|s}$ is  an isomorphism if and only if  $m\geq \min\{k,r\}$  and $n\geq \min\{l,s\}$,  with a PBW basis constructed for $\CX^{k|l}_{\;r|s}$.
This implies that the invariant subalgebra $\CX^{k|l}_{\;r|s}$ is a quantum polynomial superalgebra isomorphic to  $\widetilde{\cM}^{k|l}_{r|s}$ in this case. These results are given in  \thmref{thmSFT}

\vskip 0.4cm
\noindent 1.6.\ \
We consider two special cases of our FFT and SFT of invariant theory for $\Uqg$.

The quantum general linear group $\Uq(\gl_m)$ is a special case of $\Uqg$ with $n=0$. We immediately  obtain the generators of the subalgebra of invariants, recovering the FFT of invariant theory given in \cite[Theorem 6.10]{LZZ}. The kernel of the  surjective algebra homomorphism  mentioned above is now generated by the  subspace of $\widetilde{T}_{\la_c}$ with $\la_c=(1^{m+1})$, which is shown to be  spanned by $(m+1)\times (m+1)$ quantum minors.  Therefore,  \thmref{thmFFTglm} and  \thmref{SFTglm} together give a complete treatment for the non-commutative invariant theory for  quantum general linear group, especially the SFT appears to be new.

The universal enveloping algebra $\U(\gl_{m|n})$ is another special case, where  $q\rightarrow 1$. Using the language of matrix elements, we provide a new approach to the FFT and SFT of invariant theory for $\gl_{m|n}$, which was originally obtained  in  \cite{S1,S2}. This is interesting  in its own right.

\vspace{0.3cm}
\noindent 1.7. There have been some earlier works on the non-commutative polynomial version of  invariant theory for quantum groups.  

The works \cite{GL,GLR,Hai} investigate the coinvariant theory for the quantum general linear (super) group, where they only consider the comodule structure on the underlying non-commutative polynomial algebra. The set of  coinvarints in their setting is not a subalgebra and thus our  notion of ``generators'' (FFT) and ``relations'' (SFT) is different from  theirs. 

In \cite{LZZ} it gives a general method to construct the quantum analogues of polynomial rings,  by using module algebras and the braiding of quantum group arising from the universal $\fR$-matrix.   Then the paper gives a complete treatment of FFT for each quantum group associated with a classical Lie algebra. However, there is no complete treatment of SFTs for quantum groups. One of our results in this paper is the SFT for $\Uq(\gl_m)$, which may shed light on the SFTs for the quantum orthogonal and symplectic groups. We also note that the FFT and SFT of invariant theory for quantum symplectic group are obtained in \cite{S}, where the construction of  underlying non-commutative polynomial algebra is a little ambiguous.

It was shown in \cite{LZZ,WZ,Z03}  that 
(skew) Howe duality \cite{H89,H92} survives quantisation for the quantum general linear (super) group, and the resulting quantum Howe duality was applied to develop the $q$-deformed Segal-Shale-Weil representations. 
More recently, the quantum skew Howe dualities \cite{CKM, RT, QS} were 
used in the categorification of representations of $\Uq(\fsl_n)$ and $\Uq(\gl_{m|n})$. 
Here we extend the quantum Howe duality of type $(\Uq(\gl_{k|l}),\Uq(\gl_{r}))$ established in \cite{WZ} to type $(\Uq(\gl_{k|l}),\Uq(\gl_{r|s}))$, and simplify the original proof in \cite{WZ}.

Another  formulation of non-commutative invariant theory provides a description for the endomorphism algebra over quantum (super) groups. The paper \cite{LZZ16} establishes a full tensor functor from the category of ribbon graphs to the category of finite dimensional
representations of $\Uq(\fg)$ with $\fg=\gl_{m|n}, \osp(m|2n)$,   
giving the FFT of invariant theory  in this endomorphism algebra setting.  However, very little was known previously about the non-commutative polynomial version of invariant theory for quantum supergroups.

\section{Quantum general linear supergroup}\label{sect:Uq}
We shall work over the filed $\CK:=\C(q)$, where $q$ is an indeterminate.
For any vector superspace $A=A_{\bar0}\oplus A_{\bar1}$, we let $[\ ]: A_{\bar0}\cup A_{\bar1}\longrightarrow \Z_2:=\{\bar0, \bar1\}$ be the parity map, that is, $[a]=\bar{i}$ if $a\in A_{\bar{i}}$.  Tensor products of vector superspaces are again vector superspaces. We define the functorial permutation map 
\begin{eqnarray}\label{eq:permut}
P: A\otimes B \longrightarrow B\otimes A
\end{eqnarray}
such that  $a\otimes b \mapsto (-1)^{[a][b]}b\otimes a$ for homogeneous $a\in A$ and $b\in B$, and generalise to inhomogeneous elements linearly. 
If $A$ is an associative superalgebra,  
we define the super bracket  $[\ , \ ]: A\otimes A \longrightarrow A$ such that $[X,Y]:=XY-(-1)^{[X][Y]}YX$.

\subsection{The quantum general linear supergroup $\Uqg$} 
Denote $\Z_{+}=\{0,1,2,\dots\}$.   For any given $m,n\in \Z_{+}$, let $\bI_{m|n}=\{1,2,\dots,m+n\}$ and $\bI_{m|n}^{\prime}=\bI_{m|n}\backslash \{m+n \}$. 
Let $\{\epsilon_a\mid a\in \bI_{m|n}\}$ be the basis of a vector space with the non-degenerate symmetric bilinear form $(\epsilon_a,\epsilon_b)=(-1)^{[a]}\delta_{ab}$, 
where $[a]=\bar{0}$ if $a\le m$ and $[a]=\bar{1}$ otherwise. The roots of the general linear Lie superalgebra $\gl_{m|n}$ are $\epsilon_a-\epsilon_b$ for $a\ne b$ in $\bI_{m|n}$.
\begin{defn}\label{defnUglmn}{\rm (\cite{Z93})}
  The quantum general linear supergroup $\Uqg$ is the unital associative $\CK$-superalgebra generated by the
  \[
  \begin{aligned}
   &\text{even generators:}\   K_a, K_a^{-1},E_{b,b+1}, E_{b+1,b}, \ a,b\in \bI_{m|n},b\neq m,m+n,\\
   &\text{and odd generators:} \  E_{m,m+1},E_{m+1,m},
  \end{aligned} 
  \]
  subject to the relations
 \begin{enumerate}[itemsep=1.5mm]
 	\item [(R1)] $K_aK_a^{-1}=K_a^{-1}K_a=1, K_aK_b=K_bK_a$;
 	\item [(R2)] $K_aE_{b,b\pm 1}K_a^{-1}=q^{(\epsilon_a,\epsilon_b-\epsilon_{b\pm 1})}E_{b,b\pm 1}$;
 	\item [(R3)] $[E_{a,a+1},E_{b+1,b}]=\de_{ab}\dfrac{K_aK_{a+1}^{-1}-K_a^{-1} K_{a+1}}{q_a-q_a^{-1}}$ with $q_a=q^{(\epsilon_a,\epsilon_a)}=q^{(-1)^{[a]}}$;
 	\item [(R4)] $E_{a,a+1}E_{b,b+1}=E_{b,b+1}E_{a,a+1}, \quad E_{a+1,a}E_{b+1,b}=E_{b+1,b}E_{a+1,a},  \quad a-b\geq 2$;
 	\item [(R5)]  
 	$ (E_{a,a+1})^2E_{b,b+1}-(q+q^{-1})E_{a,a+1}E_{b,b+1}E_{a,a+1}+E_{b,b+1}(E_{a,a+1})^2=0, \quad |a-b|=1, a\neq m;$\\
 	$ (E_{a+1,a})^2E_{b+1,b}-(q+q^{-1})E_{a+1,a}E_{b+1,b}E_{a+1,a}+E_{b+1,b}(E_{a+1,a})^2=0, \quad |a-b|=1, a\neq m;$
% 	Let $\cS_q(x,y)=x^2y-(q+q^{-1})xyx+yx^2$, then
% 	\[\cS_q(E_{a,a+1}, E_{b,b+1})= \cS_q(E_{a+1,a}, E_{b+1,b})=0, |a-b|=1, a\neq m; \]
 	\item [(R6)] $(E_{m,m+1})^2=(E_{m+1,m})^2=[E_{m-1,m+2},E_{m,m+1}]=[E_{m+2,m-1},E_{m+1,m}]=0$, 
 	
 	\hspace{-.5cm} where  $E_{m-1, m+2}$ and $E_{m+2, m-1}$ are determined recursively  by 
 	\[
 	E_{a,b}=
 	\begin{cases}
 	E_{a,c}E_{c,b}-q_c^{-1}E_{c,b}E_{a,c},& a<c<b,\\
 	E_{a,c}E_{c,b}-q_cE_{c,b}E_{a,c}, & b<c<a.
 	\end{cases} 	
 	\]
 \end{enumerate}
\end{defn}
Note that  $[E_{a,b}]= [a]+[b]$ for any quantum root vector $E_{a,b}$. 
It is well known that $\Uqg$ has the structure of Hopf superalgebra 
with the following structural maps:

\noindent
co-multiplication 
\begin{equation*}
  \begin{aligned}[c]
  	\Delta(E_{a,a+1})&=E_{a,a+1} \otimes K_a K_{a+1}^{-1} + 1 \otimes E_{a,a+1},\\
  	\Delta(E_{a+1,a})&=E_{a+1,a}\otimes 1 + K_a^{-1} K_{a+1}\otimes E_{a+1,a},\\
  	\Delta(K_a^{\pm 1})&=K_a^{\pm 1}\otimes K_a^{\pm 1},
  \end{aligned}
 \end{equation*}
 co-unit 
\[
 \epsilon(E_{a, a+1})= \epsilon(E_{a+1, a})=0, \ \forall a\in\bI_{m|n}', \quad
 \epsilon(K_b^{\pm 1})=1,\  \forall b\in\bI_{m|n}, 
\]
 and  antipode 
 \begin{equation*}
  \begin{aligned}[c]
  S(E_{a,a+1})&=-E_{a,a+1} K_a^{-1}K_{a+1},\\
  S(E_{a+1,a})&=-K_a K_{a+1}^{-1}E_{a+1,a},\\
  S(K_a^{\pm 1})&=K_a^{\mp 1}.
  \end{aligned} 
\end{equation*}
It should be pointed out that the antipode is a $\Z_2$-graded algebra anti-automorphism, i.e., $S(xy)=(-1)^{[x][y]}S(y)S(x)$ for homogeneous $x,y\in \Uqg$.  We will use Sweedler's notation for the co-multiplication, i.e.,  $\Delta(x)=\sum_{(x)}x_{(1)}\ot x_{(2)}$ for all $x\in \Uqg$.
 
\subsection{Representations of $\Uqg$}\label{RepUqg}
Throughout this paper, we will consider only $\Uqg$-modules which are $\Z_2$-graded, locally finite and of type $\bf{1}$.  A module is locally finite if any vector $v$ satisfies $\dim(\Uqg v)<\infty$; a module is of type $\bf{1}$ if the $K_a$ act semi-simply  with eigenvalues $q^{k}$ ($k\in\Z$). 

The representation theory of $\Uqg$ at generic $q$ is quite similar to that of $\gl(m|n)$, 
which was treated systematically in \cite{Z93}. In \appref{appA}, we briefly discuss the representation theory for $\Uqg$.  Here we consider the tensor modules in detail following \cite{Z98}. 

Let $V^{m|n}$ be the natural $\Uqg$-module, which has the standard basis $\{v_{a}\mid a\in \bI_{m|n}\}$ such that $[v_a]=[a]$, and  
 \[K_av_b=q_a^{\de_{ab}}v_b,\quad E_{a,a\pm 1}v_b=\de_{b,a\pm 1}v_a.\]
Denote by  $\pi: \Uqg\rightarrow \End(V^{m|n}) $ the corresponding representation of $\Uqg$ relative to the standard basis of $V^{m|n}$. Then
 \begin{equation}\label{eqnatrep}
 \begin{aligned}
 &\pi(K_a^{\pm})=I+(q_a^{\pm 1}-1)e_{aa},\quad a\in\bI_{m|n}\\
 &\pi(E_{b,b+1})=e_{b,b+1},\quad \pi(E_{b+1,b})=e_{b+1,b},\quad b\in \bI_{m|n}^{\prime}
 \end{aligned}
 \end{equation}
 where $e_{ab}$ is the $(a,b)$-matrix unit of size $(m+n)\times(m+n)$.
 
Given $k\in \Z_{+}$, we consider the tensor power $(V^{m|n})^{\ot k}=V^{m|n}\otimes\dots\otimes V^{m|n}$ ($k$ factors), which acquires a $\Uqg$-module structure through the iterated co-multiplication $\Delta^{(k-1)}=(\id\times \Delta^{(k-2)})$ with $\Delta^{(1)}=\Delta$.  It is proved in \cite{Z98} that $(V^{m|n})^{\ot k}$ is semi-simple for all $k$. To  characterise their simple submodules, we define a subset $\La_{m|n}$ of $\Z_{+}^{\times (m+n)}$ by
\begin{equation}\label{defnPar}
  \La_{m|n}=\{ \la=(\la_1,\la_2,\dots,\la_{m+n})\in \Z_{+}^{\times (m+n)}\mid \la_{m+1}\leq n,\, \la_a\geq \la_{a+1},\; \forall a\in \bI_{m|n}^{\prime}  \}.
\end{equation}
Such $\la$  is referred to as \emph{$(m,n)$-hook partition}, and $|\la|=\sum_{a\in\bI_{m|n}}\la_a$ is called the \emph{size  of $\la$}. We associate each $\la\in\La_{m|n}$ with $\la^{\natural}$ defined by 
\begin{equation}\label{eqlanat}
\la^{\natural}=\sum_{a=1}^{m}\la_a\epsilon_a+\sum_{b=1}^{n}\langle \la_{b}^{\prime}-m\rangle \epsilon_{m+b} =(\la_1,\dots,\la_m;\langle \la_1^{\prime}-m\rangle,\dots,\langle \la_n^{\prime}-m\rangle),
\end{equation}
where $\la^{\prime}=(\la_1^{\prime},\la_2^{\prime},\dots)$ with $\la_i^{\prime}=\#\{\la_{j}\mid \la_{j}\geq i,\  j\in \bI_{m|n} \}$  is the transpose partition of $\la$, and $\langle u\rangle:=\max\{u,0\}$  for any integer $u$. Observe  that $\la_{m+1}\leq n$ if and only if $ \la_{n+1}^{\prime}\leq m$.
% we have the bijective correspondence between $\La_{m|n}$ and $\La_{n|m}$ sending each $\la\in\La_{m|n}$ to its transpose $\la^{\prime}\in\La_{n|m}$.   
Now introduce the set  $\La_{m|n}^{\natural}=\{\la^{\natural}\mid \la\in \La_{m|n} \}$,
and the subsets $\La_{m|n}^{\natural}(N)=\{\la^{\natural}\in\La_{m|n}^{\natural}\mid |\la|=N \}$ for any $N\in\Z_+$.  

We denote by $L_{\la}^{m|n}$ the irreducible $\Uqg$-module with highest weight $\la^{\natural}\in\La^{\natural}$, that is, there exists a nonzero $v\in L_{\la}^{m|n}$, the highest weight vector, such that 
\[
E_{a, a+1}v=0,  \quad K_b v=q^{(\la^{\natural},\epsilon_b)}v, \quad \forall a\in \bI'_{m|n}, \ \ b\in \bI_{m|n}.
\]

The following results are from \cite{Z98}.
\begin{prop}\label{proptendec} {\rm (\cite{Z98})}
	\begin{enumerate}
	\item Each $\Uqg$-module $(V^{m|n})^{\ot N}$ with $N\in \Z_{+}$ can be decomposed into a direct sum of simple modules with highest weights belonging to $\La_{m|n}^{\natural}(N)$.
	\item Every irreducible $\Uqg$-module with highest weight belonging to $\La_{m|n}^{\natural}(N)$ is contained in $(V^{m|n})^{\ot N}$ as a simple submodule.
\end{enumerate}
\end{prop}

\subsection{Module superalgebras} \label{secmodalg}
Recall that the Drinfeld version of $\Uqg$ (defined over the power series ring $\C[[t]]$ with $q=\exp(t)$) admits a universal $\fR$ matrix, which is an invertible even element in a completion of 
$\Uqg\otimes \Uqg$  satisfying the relations
\begin{align}
&\fR\Delta(x)=\Delta^{\prime}(x)\fR,\quad \forall x\in \Uqg,\label{eqRmat}\\
& (\Delta \ot\id)\fR=\fR_{13}\fR_{23},\quad (\id\ot \Delta)\fR=\fR_{13}\fR_{12},\label{eqDeR}
\end{align}
where $\Delta^{\prime}$ is the opposite co-multiplication.  It is well known that $\fR$ satisfies the celebrated Yang-Baxter equation
\[
\fR_{12}\fR_{13}\fR_{23} =\fR_{23}\fR_{13}\fR_{12}.
\] 

Let $V$ and $W$ be two $\Uqg$-modules. It follows from \eqref{eqRmat} that the map 
 \[ \widecheck{\fR}:=P\fR:\, V\ot W\rightarrow W\ot V, \]
 satisfies $\widecheck{\fR}\Delta(x)(v\ot w)=\Delta(x)\widecheck{\fR}(v\ot w)$, 
 thus gives rise to an isomorphism of $\Uqg$-modules.

For the Jimbo version of the quantum general linear supergroup $\Uqg$ (over $\CK$) under consideration here, there exists a canonical braiding in the category of locally finite $\Uqg$-modules of type ${\bf 1}$, which plays the role of the universal $\fR$ matrix. 

In what follows, we shall make extensive use of module superalgebras over a Hopf superalgebra, which is a super analogue of module algebras in the sense of \cite[\S 4.1]{M}, see also \cite{LZZ}. An associative superalgebra $(A,\mu)$ with multiplication $\mu$ and identity $1_{A}$ is a $\Uqg$-\emph{module superalgebra} if $A$ is $\Uqg$-module with
\[x.\mu(a\ot b)=\sum_{(x)}(-1)^{[x_{(2)}][a]}\mu(x_{(1)}.a\ot x_{(2)}.b),\quad x.1_{A}=\epsilon(x)1_{A},  \]   
for homogeneous elements $a,b\in A$ and $x\in\Uqg$.
 
\begin{prop}\label{proptensormodalg}
Let $(A,\mu_{A})$ and $(B,\mu_{B})$ be locally finite $\Uqg$-module superalgebras. Then $A\ot B$ acquires a $\Uqg$-module superalgebra structure with the multiplication
\[\mu_{A,B}=(\mu_A\ot\mu_B)(\id_A\ot \widecheck{\fR}\ot \id_B),  \]	
where $\widecheck{\fR}:=P\fR$ with $P(a\ot b)=(-1)^{[a][b]}b\ot a$ for any homogeneous $a\in A,b\in B$.
\end{prop}
\begin{proof}
Let $\fR=\sum_h \alpha_h\ot\beta_h$. We have $[\alpha_h]=[\beta_h]$ for all $h$ since $\fR$ is an even element in $\Uqg\ot\Uqg$. Then under the given multiplication, it follows that
\begin{equation*}
(a_1\ot b_1)(a_2\ot b_2)=\sum_h (-1)^{([\alpha_h]+[b_1])[a_2]+[\alpha_h][\beta_h]}a_1(\beta_h.a_2)\ot (\alpha_h.b_1)b_2
\end{equation*}
for any homogeneous $a_1,b_1,a_2,b_2\in \Uqg$.
It is straightforward to check that for all homogeneous $x\in \Uqg$, 
\[x.((a_1\ot b_1)(a_2\ot b_2))=\sum_{(x)}(-1)^{[x_{(2)}]([a_1]+[b_1])}(x_{(1)}.(a_1\ot b_1))(x_{(2)}.(a_2\ot b_2)), \]
where we have used the identity $\widecheck{\fR}\Delta(x)(a\ot b)=\Delta(x)\widecheck{\fR}(a\ot b)$.
\end{proof}

We write $A\ot_{\fR}\!B$  for the braided tensor product of module superalgebras $A$ and $B$ as defined in  \propref{proptensormodalg} to distinguish it from the usual tensor product $A\ot_{\CK}\!B$.

\begin{rmk}
	Observe that the multiplication defined in this way is associative. Therefore, the  $\Uqg$-module superalgebra structure on $A_1\ot_{\fR} \!A_2\ot_{\fR}\dots\ot_{\fR}\! A_k$ is well-defined for any given $\Uqg$-module superalgebras $A_i$, $i=1,2,\dots,k$.
\end{rmk}
 
Let $A^{\Uqg}=\{ a\in A\mid x.a=\epsilon(x)a,\; \forall x\in\Uqg\}$ be the subspace of $\Uqg$-invariants in the module superalgebra $A$. The following result is well known, see e.g. \cite[Lemma 2.2]{LZZ}.
\begin{lem}\label{leminvalg}
Let $A$ be a $\Uqg$-module superalgebra. Then $A^{\Uqg}$ is a submodule superalgebra of $A$.
\end{lem} 
\begin{proof}
	For any homogeneous $a,b\in A^{\Uqg}$ and $x\in \Uqg$,  
	\[ x.(ab)=\sum_{(x)} (-1)^{[x_{(2)}][a]} (x_{(1)}.a) (x_{(2)}.b)=\sum_{(x)} (-1)^{[x_{(2)}][a]} \epsilon(x_{(1)}) \epsilon(x_{(2)}) ab=\epsilon(x) ab, \]
    where the last equation follows from the identity 
    \[\sum_{(x)} (-1)^{[x_{(2)}][a]} x_{(1)}\epsilon(x_{(2)})=\sum_{(x)}x_{(1)}\epsilon(x_{(2)})=x.   \]
     As a consequence, $ab\in A^{\Uqg}$.
\end{proof} 
 
\subsection{Quantised function algebras on $\Uqg$ }\label{secfunalg}
\subsubsection{The bi-superalgebra $\cM_{m|n}$}
Let 
\[\Uq(\gl_{m|n})^{\circ}:=\{ f\in (\Uqg)^{\ast}\mid  \text{$\Ker f$ contains a cofinite ideal of $\Uqg$}  \}\]
be the finite dual \cite{M} of $\Uqg$, which is a  Hopf superalgebra with structure dualising that of $\Uqg$.  We have the matrix elements $t_{ab}\in \Uq(\gl_{m|n})^{\circ}$,
$a, b\in\bI_{m|n}$, defined by \cite{Z98}
\begin{equation}\label{eqmatelmt}
  \langle t_{ab},x \rangle =\pi ( x )_{ab}, \quad \forall x\in\Uqg,
\end{equation}
where $\pi$ is the natural representation defined as in \eqref{eqnatrep}.
We set the $\Z_{2}$-grading of $t_{ab}$ by $[t_{ab}]=[a]+[b]$.  Consider  the subalgebra $\cM_{m|n}$ of $\Uq(\gl_{m|n})^{\circ}$ generated by matrix elements  $t_{ab} \,(a,b\in \bI_{m|n})$. The  multiplication which $\cM_{m|n}$ inherits from  $\Uq(\gl_{m|n})^{\circ}$ is given by
\begin{equation}\label{eqmultM} 
	\langle tt',x\rangle = \sum_{(x)} \langle t\otimes t',
	x_{(1)}\otimes x_{(2)} \rangle 
	= \sum_{(x)} (-1)^{[t'][x_{(1)}]} \langle t,  x_{(1)} \rangle
	\langle t',x_{(2)} \rangle
\end{equation}	
for all $t, t'\in \cM_{m|n}$ and $x\in \Uqg$. 

 We proceed to give the relations of the generators $t_{ab}$. Applying $\pi\ot\pi$ to both sides of \eqref{eqRmat}, we have
\begin{equation}\label{eqpiR}
 R(\pi\ot\pi)\Delta(x)=(\pi\ot\pi)\Delta^{\prime}(x)R, 
\end{equation}
where $R:=(\pi\ot\pi)\fR$, which is of the form \cite{Z98}
\[
\begin{aligned}
R&=q^{\sum_{a\in \bI_{m|n}}(-1)^{[a]} e_{aa}\ot e_{aa}}+(q-q^{-1})\sum_{a<b}(-1)^{[b]}e_{ab}\ot e_{ba}\\
&=I\otimes I + \sum_{a\in \bI_{m|n}} (q_a - 1)e_{a a}\otimes e_{a a} + 
(q-q^{-1})\sum_{a<b}(-1)^{[b]}e_{ab}\ot e_{ba}, 
\end{aligned}
\]
where $q_a=q^{(-1)^{[a]}}$. It is easy to directly verify that the matrix $R$ satisfies the Yang-Baxter equation.

Write $R=\sum_{a,b,c,d\in \bI_{m|n}} e_{ab}\ot e_{cd}R^{a\,c}_{b\,d}$,  then the nonzero entries of the matrix $R$ are 
\[
R^{a\,b}_{a\,b}=1,\; a\neq b,\quad R^{a\,a}_{a\,a}=q_a,\quad R^{a\,b}_{b\,a}=q_b-q_b^{-1},\; a<b,
\]
where $q_b-q_b^{-1}=(-1)^{[b]}(q-q^{-1})$ for all $a,b\in \bI_{m|n}$. It follows from \eqref{eqpiR} that the  generators $t_{ab}$ obey relations
\begin{equation*}\label{eqRrelation}
\sum_{a^{\prime},b^{\prime}}(-1)^{([a^{\prime}]+[c])([b]+[d])} R^{a\;b}_{a^{\prime}b^{\prime}} t_{a^{\prime}c} t_{b^{\prime}d}=\sum_{a^{\prime},b^{\prime}}(-1)^{([a^{\prime}]+[c])([b]+[b^{\prime}])} t_{bb^{\prime}} t_{aa^{\prime}} R^{a^{\prime}b^{\prime}}_{c\;d},
\end{equation*}	
which can be written more explicitly as
\begin{equation}\label{eqRel1}
 \begin{aligned}
  (t_{ab})^2=&0, &\quad &[a]+[b]=\bar{1},\\
  t_{ac}t_{bc}=&(-1)^{([a]+[c])([b]+[c])}q_ct_{bc}t_{ac},&\quad &a>b,\\
  t_{ab}t_{ac}=&(-1)^{([a]+[c])([a]+[b])}q_at_{ac}t_{ab},&\quad &b>c, \\
  t_{ac}t_{bd}=&(-1)^{([a]+[c])([b]+[d])}t_{bd}t_{ac},&\quad &a>b,\,c<d,\\
  t_{ac}t_{bd}=&(-1)^{([a]+[c])([b]+[d])}t_{bd}t_{ac}\\
  		&+(-1)^{[a]([b]+[d])+[b][d]}(q-q^{-1})t_{bc}t_{ad},&\quad &a>b,\, c>d. 
  \end{aligned}
\end{equation}
It is worthy to note  that  $\cM_{m|n}$ has a bi-superalgebra 	structure with the co-multiplication and the co-unit given by 
\[\Delta(t_{ab})=\sum_{c\in \bI_{m|n}} (-1)^{([a]+[c])([c]+[b])}t_{ac}\ot t_{cb}, \quad \epsilon(t_{ab})=\de_{ab}.\] 

\begin{rmk}
	With the matrix notation $T=\sum_{a,b\in \bI_{m|n}} e_{ab}\ot t_{ab}$,  relations \eqref{eqRel1} can be simply written in the form
	\[ RT_1T_2=T_2T_1R,  \]
	where $T_1T_2=\sum_{a,b,c,d\in\bI_{m|n}} (-1)^{([a]+[b])([c]+[d])}e_{ab}\ot e_{cd}\ot t_{ab}t_{cd} $; see \cite{Z98} and \cite{RTF} for details. 
\end{rmk}

There exist two kinds of actions  on $\cM_{m|n}$,  corresponding respectively to  the left and right translations in the classical situation.  We define the $\cR$-action  by  
\begin{equation}\label{eqRact}
  \cR: \Uqg\ot \cM_{m|n}\rightarrow\cM_{m|n} \quad   x\ot f \mapsto \cR_{x}(f)=\sum_{(f)}(-1)^{([f]+[x])[x]}f_{(1)}\langle f_{(2)},x\rangle.
\end{equation}	
 Recall that  there is an anti-automorphism $w$  of $\Uqg$  given by 
\[w(E_{a,a+1})=E_{a+1,a},\quad w(E_{a+1,a})=E_{a,a+1},\quad w(K_a)=K_a. \]
This and the antipode $S$  respectively give rise to another two left actions $\CL$ and $\widetilde{\CL}$
\begin{equation}\label{eqLact}
  \begin{aligned}
    &\CL: \Uqg\ot \cM_{m|n}\rightarrow\cM_{m|n} \quad  x\ot f \mapsto \sum_{(f)} \langle f_{(1)},w(x)\rangle f_{(2)},\\
    &\widetilde{\CL}: \Uqg\ot \cM_{m|n}\rightarrow\cM_{m|n} \quad  x\ot f \mapsto \sum_{(f)}\langle f_{(1)},S(x)\rangle f_{(2)},
  \end{aligned}
\end{equation}	

The two actions $\CL$ (or $\widetilde{\CL}$) and $\cR$ graded-commute with each other, i.e.,
 \[ \CL_{x}(\cR_{y}(f))=(-1)^{[x][y]}\cR_{x}(\CL_{y}(f)). \] 
 In addition, they preserve the algebraic structure of $\cM_{m|n}$. Thus, $\cM_{m|n}$ is a module superalgebra over $\CL(\Uqg)\ot \cR(\Uqg)$. It admits the following  multiplicity-free decomposition, which is a partial analogue of quantum Peter-Weyl theorem.
\begin{thm}\label{thmPeterWeyl} {\rm (\cite[Proposition 4]{Z98})}
	As an $\CL(\Uqg)\ot \cR(\Uqg)$-module,
	\begin{equation}\label{eqPeterWeyl}
	\cM_{m|n}\cong \bigoplus_{\la\in \La_{m|n}} L^{m|n}_{\la}\ot L^{m|n}_{\la},
	\end{equation}
	where $\La_{m|n}$ is the set of $(m,n)$-hook partitions.
\end{thm}	
\begin{rmk} \cite{Z98}
	$\cM_{m|n}\cong \bigoplus_{\la\in \La_{m|n}} (L^{m|n}_{\la})^{\ast}\ot L^{m|n}_{\la}$ as $\widetilde{\CL}(\Uqg)\ot \cR(\Uqg)$-module. 
\end{rmk}	
	
\subsubsection{The Hopf superalgebra $\CK[\GL_q(m|n)]$}
Following convention in \cite{Z98}, we define  elements $\bar{t}_{ab}\in \Uq(\gl_{m|n})^{\circ}$ by
\[\langle \bar{t}_{ab},x\rangle:=\bar{\pi}(x)_{ab}= (-1)^{[b]([a]+[b])}\langle t_{ba},S(x)\rangle,\quad \forall x\in \Uqg,\, a,b\in \bI_{m|n}. \]
These are  matrix elements of the dual vector representation $\bar{\pi}$ of $\Uqg$ acting on $(V^{m|n})^{\ast}$. They generate a $\Z_2$-graded bi-algebra $\overbar{\cM}_{m|n}$ with the multiplication as in \eqref{eqmultM},  and co-multiplication $\Delta$ and co-unit $\epsilon$ given by 
\[\Delta(\bar{t}_{ab}) = \sum_{c\in \bI_{m|n}} (-1)^{([a]+[c])([c]+[b])}\bar{t}_{ac}\ot \bar{t}_{cb}\quad \epsilon(\bar{t}_{ab})=\delta_{ab}.\]
This is a $\Uqg$-module superalgebra with respect to left $\Uqg$-actions
$\CL, \widetilde{\CL},\cR: \Uqg\ot \overbar{\cM}_{m|n}\rightarrow \overbar{\cM}_{m|n}$ similarly defined as in \eqref{eqRact} and \eqref{eqLact}. Furthermore,  $\overbar{\cM}_{m|n}$ admits the following decomposition as in \thmref{thmPeterWeyl}.
\begin{thm}\label{thmPeterWeyl2}
  As  an $\widetilde{\CL}(\Uqg)\ot \cR(\Uqg)$-module,
  \begin{equation}
  	\overbar{\cM}_{m|n}\cong \bigoplus_{\la\in \La_{m|n}} L^{m|n}_{\la}\ot (L^{m|n}_{\la})^{\ast}.
  \end{equation}
\end{thm}

Now we denote by $\CK[\GL_q(m|n)]$ the 	subalgebra  of $\Uq(\gl_{m|n})^{\circ}$ generated by 
all matrix elements $t_{ab}$ and $\bar{t}_{ab}$. The relations,  besides \eqref{eqRel1}, can be derived similarly from \cite{Z98}
\begin{align}
(\bar{\pi}\ot \bar{\pi})(\fR \Delta(x))=(\bar{\pi}\ot \bar{\pi})(\Delta^{\prime}(x)\fR ),  \label{eqrelbar}\\
(\pi\ot \bar{\pi})(\fR \Delta(x))=(\pi\ot \bar{\pi})(\Delta^{\prime}(x)\fR ), 
\end{align}
where the second one enables us to factorise $\CK[\GL_q(m|n)]$ into
\begin{equation*}\label{eqFact}
 \CK[\GL_q(m|n)]=\cM_{m|n}\overbar{\cM}_{m|n},
\end{equation*}
which inherits a natural bi-algebra structure from $\cM_{m|n}$ and $\overbar{\cM}_{m|n}$. Straightforward calculations can verify the relation
\[ \sum_{c\in\bI_{m|n}}(-1)^{[a]([b]+[c])}t_{ac}\bar{t}_{bc}=\sum_{c\in\bI_{m|n}}(-1)^{([a]+[c])([a]+[b]+[c])}\bar{t}_{ca}t_{cb}=\de_{ab}, \]
which implies that $\CK[\GL_q(m|n)]$ is a Hopf superalgebra with the  antipode $S$  given by
\begin{equation}\label{eqantipode}
S(t_{ab})=(-1)^{[a]([a]+[b])}\bar{t}_{ba},\quad S(\bar{t}_{ab})=(-1)^{[b]([a]+[b])}q^{(2\rho,\epsilon_a-\epsilon_b)}t_{ba}.
\end{equation}
Here $2\rho=\sum_{a< b}(-1)^{[a]+[b]}(\epsilon_a-\epsilon_b)$ and  $\epsilon_a-\epsilon_b$ is a root for $\gl_{m|n}$.

\begin{rmk}\label{rmkcoord}
    It is proved in \cite[Theorem 3.5]{ZZ06} that $\CK[\GL_q(m|n)]\cong \CO_q(\GL_{m|n})$ as Hopf superalgebras, where $\CO_q(\GL_{m|n})$ is the  coordinate superalgebra of quantum general linear supergroup.
\end{rmk}

\begin{rmk}\label{rmkRinv}
 Let $\overbar{R}=(\bar{\pi}\ot\bar{\pi})\fR$, by the formula in \cite{Z98} we have
 \[ 
 \begin{aligned}
 \overbar{R}&=q^{\sum_{a\in \bI_{m|n}}(-1)^{[a]} e_{aa}\ot e_{aa}}+(q-q^{-1})\sum_{a>b}(-1)^{[b]}e_{ab}\ot e_{ba}\\
 &=I\otimes I + \sum_{a\in \bI_{m|n}} (q_a - 1)e_{a a}\otimes e_{a a}+(q-q^{-1})\sum_{a>b}(-1)^{[b]}e_{ab}\ot e_{ba}.
 \end{aligned}
 \]
 Thus, relations \eqref{eqrelbar} can be written more explicitly as follows,
\begin{equation}\label{eqRel2}
\begin{aligned}
 (\bar{t}_{ab})^2=&0, &\quad &[a]+[b]=\bar{1},\\
\bar{t}_{ac}\bar{t}_{bc}=&(-1)^{([a]+[c])([b]+[c])}q_c^{-1}\bar{t}_{bc}\bar{t}_{ac},&\quad &a>b,\\
\bar{t}_{ab}\bar{t}_{ac}=&(-1)^{([a]+[c])([a]+[b])}q_a^{-1}\bar{t}_{ac}\bar{t}_{ab},&\quad &b>c, \\
\bar{t}_{ac}\bar{t}_{bd}=&(-1)^{([a]+[c])([b]+[d])}\bar{t}_{bd}\bar{t}_{ac},&\quad &a>b,\,c<d,\\
\bar{t}_{ac}\bar{t}_{bd}=&(-1)^{([a]+c)([b]+[d])}\bar{t}_{bd}\bar{t}_{ac}\\
&-(-1)^{[a]([b]+[d])+[b][d]}(q-q^{-1})\bar{t}_{bc}\bar{t}_{ad},&\quad &a>b,\, c>d. 
\end{aligned}
\end{equation}	
\end{rmk}

\section{Quantum Howe duality of type $(\Uq(\gl_{k|l}),\Uq(\gl_{r|s}))$}

\subsection{Quantum super analogue of Howe duality}\label{secHowedual}
We adopt the following notation for convenience. If $\CG=\{g_1, \dots, g_k\}$ is a finite set of  elements in $\Uqg$, we denote by $\langle \CG\rangle$ the linear span of the elements $g_{i_1}g_{i_2}\dots g_{i_N}$ for all $N\in\Z_+$ and $1\le i_1, \dots, i_N\le k$, that is, $\langle \CG\rangle$ is subalgebra generated by $\CG$.   Given integers $1\leq s,t\leq m+n$, we introduce the following subalgebra of $\Uqg$
\[
 \begin{aligned}
 &\Upsilon_u=\langle K_a\mid 1\leq a \leq u \rangle, \quad 1\leq u\leq m+n,\\
 &\overline{\Upsilon}_v=\langle K_a\mid v+1\leq a \leq m+n \rangle,\quad 0\leq v\leq m+n-1.
\end{aligned}
\]
For any $\Uqg$-module $V$, we define the $\Upsilon_u$-invariant subspace of $V$ by
\[ V^{\Upsilon_u}:=\{v\in V\mid K_av=v, \  \forall 1\leq a\leq u\}, \] 
and call this a truncation of $V$. 
The  $\overline{\Upsilon}_v$-invariant subspace $V^{\overline{\Upsilon}_v}$ can be defined similarly. The following two lemmas will be crucial for our proof of quantum Howe duality.

\begin{lem}\label{lemTrunh}
	Let $\la\in \La_{m|n}$ be an $(m,n)$-hook partition. The  $\overline{\Upsilon}_v$-invariant subspace of $L_{\la}^{m|n}$ is
	\begin{equation*}
	(L_{\la}^{m|n})^{\overline{\Upsilon}_v}\cong
	\begin{cases}
	L_{\la}^{m|(v-m)}, & \text{if}\ v> m \ \text{and}\ \la_{m+1}\leq v-m, \\
	L_{\la}^{v}, & \text{if}\ v\leq m \ \text{and}\ \la_{v+1}=0,\\
	0,& \text{otherwise},
	\end{cases}
	\end{equation*}                         
	where $L_{\la}^{m|v-m}$  is the irreducible $\Uq(\gl_{m|v-m})$-module with highest weight $(\la_1,\dots,\la_m;\langle \la_1^{\prime}-m\rangle,\dots, \\ \langle \la_{v-m}^{\prime}-m\rangle )$, and $L_{\la}^{v}$ is the irreducible $\Uq(\gl_v)$-module with highest weight $(\la_1,\dots,\la_v)$.  
\end{lem}
\begin{proof}
	Note that $\Uq(\gl_{m|n-1})$ is canonically embedded in $\Uqg$ as the subalgebra generated by the elements $E_{a, a+1}, E_{a+1, a}$, $K^{\pm 1}_b$ with $1\le a<m+n-1$ and $1\le b\le m+n-1$. Thus we can consider the restriction of the $\Uqg$-module $L_{\la}^{m|n}$ to a module for  $\Uq(\gl_{m|n-1})$. Let $v_{\la}\in  L_{\la}^{m|n}$ be the highest weight vector for $\Uqg$ with the highest weight $\la^{\natural}$ (see \eqref{eqlanat}).
	Then $(L_{\la}^{m|n})_0=\Uq(\gl_{m|n-1})v_{\la}$ forms an irreducible  $\Uq(\gl_{m|n-1})$-module. 
Note that $K_{m+n}$, which commutes with $\Uq(\gl_{m|n-1})$, acts on $(L_{\la}^{m|n})_0$ by multiplication by the scalar $q^{\langle \la_n^{\prime}-m\rangle}$.	
	
	Denote by $\Uq(\fn_{-})_0$ the subalgebra of $\Uqg$ generated by the elements $E_{m+n,a}$ with $1\leq a\leq m+n-1$. Then $L_{\la}^{m|n}=\Uq(\fn_{-})_0(L_{\la}^{m|n})_0$. 
    Thus, all the weights $\mu=(\mu_1,\mu_2,\dots,\mu_{m+n})$ of $L_{\la}^{m|n}$ satisfy $\mu_{m+n}\geq  \langle \la_{n}^{\prime}-m\rangle$.  Therefore, $(L_{\la}^{m|n})^{K_{m+n}}=(L_{\la}^{m|n})^{\overline{\Upsilon}_{m+n-1}}=0$ unless $\langle \la_n^{\prime}-m\rangle=0$, i.e., $\la_n^{\prime}\leq m$, which is equivalent to the condition that $\la_{m+1}\leq n-1$ (recall that $\la_{k+1}\leq l \Leftrightarrow \la_{l+1}^{\prime}\leq k$ for any $k,l\in \Z_{+}$). In this case, we have $(L_{\la}^{m|n})^{\overline{\Upsilon}_{m+n-1}}=(L_{\la}^{m|n})_0$, which is an irreducible  $\Uq(\gl_{m|n-1})$-module with the highest weight $(\la_1,\dots,\la_m;\langle \la_1^{\prime}-m\rangle,\dots, \langle \la_{n-1}^{\prime}-m\rangle )$. 

Iterating this truncation procedure for $(L_{\la}^{m|n})^{\overline{\Upsilon}_{m+n-1}}$, we arrive at the  assertion.
\end{proof}

\begin{lem}\label{lemTrunl}
	Let $\la\in \La_{m|n}$ be an $(m,n)$-hook partition. The $\Upsilon_u$-invariant subspace of  $L_{\la}^{m|n}$ is
	\begin{equation*}
	(L_{\la}^{m|n})^{\Upsilon_u}\cong
	\begin{cases}
	L_{\la}^{(m-u)|n}, & \text{if}\ u< m \ \text{and}\ \la_{m-u+1}\leq n, \\
	L_{\la^{\prime}}^{m+n-u}, & \text{if}\ u\geq m \ \text{and}\ \la_{m+n-u+1}^{\prime}=0,\\
	0,& \text{otherwise},
	\end{cases}
	\end{equation*}                         
	where $L_{\la}^{m-u|n}$  is the irreducible $\Uq(\gl_{m-u|n})$-module with highest weight $(\la_1,\dots,\la_{m-u};\langle \la_1^{\prime}-(m-u)\rangle,\dots,\langle \la_{n}^{\prime}-(m-u)\rangle )$, and  
	$L_{\la^{\prime}}^{m+n-u}$ is the irreducible $\Uq(\gl_{m+n-u})$-module with highest weight
	$(\la_{1}^{\prime},\dots,\la_{m+n-u}^{\prime})$.	 	
\end{lem}
\begin{proof}
Let $\overbar{\la}^{\natural}$ be the lowest weight of $L_{\la}^{m|n}$.   By \propref{proplowwht} in \appref{appA}, we have
	\[\overbar{\la}^{\natural}= (\langle\la_m-n\rangle,\langle\la_{m-1}-n\rangle,\dots,\langle\la_1-n\rangle; \la_n^{\prime},\la_{n-1}^{\prime},\dots,\la_1^{\prime}). \]
For convenience, we write $\varGamma_{\overbar{\la}}^{m|n}$ for $L_{\la}^{m|n}$ and $v_{\overbar{\la}}$ for the lowest weight vector of $\varGamma_{\overbar{\la}}^{m|n}$. Then $(\varGamma_{\overbar{\la}}^{m|n})_0:=\Uq(\gl_{m-1|n})v_{\overbar{\la}}$ is an irreducible $\Uq(\gl_{m-1|n})$-module, where $\Uq(\gl_{m-1|n})$ is canonically embedded in $\Uqg$ as the subalgebra generated by the elements $E_{a, a+1}, E_{a+1, a}$, $K^{\pm 1}_b$ with $2\leq a\leq m+n-1$ and $2\le b\le m+n$. Now $K_{1}$, which commutes with $\Uq(\gl_{m-1|n})$, acts on $(L_{\la}^{m|n})_0$ by multiplication by the scalar $q^{\langle \la_m-n\rangle}$.

 Denote by $\Uq(\fn_{+})_0$ the subalgebra of $\Uqg$ generated by $E_{1,a}, 2\leq a\leq m+n$, it follows that $\varGamma_{\overbar{\la}}^{m|n}=\Uq(\fn_{+})_0(\varGamma_{\overbar{\la}}^{m|n})_0$. Thus, all the weights $\mu=(\mu_1,\mu_2,\dots,\mu_{m+n})$ of $\varGamma_{\overbar{\la}}^{m|n}$ satisfy $\mu_{1}\geq  \langle\la_m-n\rangle$.  Therefore, it is clear that $(\varGamma_{\overbar{\la}}^{m|n})^{K_1}=(\varGamma_{\overbar{\la}}^{m|n})^{\Upsilon_1}=0$ unless $\langle\la_m-n\rangle=0$, i.e., $\la_m\leq n$. In this case, we have $(\varGamma_{\overbar{\la}}^{m|n})^{\Upsilon_1}=(\varGamma_{\overbar{\la}}^{m|n})_0$, which is an irreducible $\Uq(\gl_{m-1|n})$-module with the lowest weight $(\langle\la_{m-1}-n\rangle,\dots,\langle\la_1-n\rangle; \la_n^{\prime},\la_{n-1}^{\prime},\dots,\la_1^{\prime})$. This by \propref{proplowwht} means that $(\varGamma_{\overbar{\la}}^{m|n})^{\Upsilon_1}=L_{\la}^{(m-1)|n}$ has the highest weight $(\la_1,\dots,\la_{m-1};\langle \la_1^{\prime}-(m-1)\rangle,\dots,\langle \la_{n}^{\prime}-(m-1)\rangle )$. Iterations of  this truncation procedure lead to the assertion.
\end{proof}

Given integers $k,l$ that  $0\leq k\leq m-1$ and $0\leq l\leq n-1$, we define a new subalgebra of $\Uqg$ by
\begin{equation}\label{eqUp}
  \Upsilon_{k|l}:=\Upsilon_{m-k} \overline{\Upsilon}_{m+l}=\langle\, K_a\mid a\in \{1,\dots m-k\}\cup\{m+l+1,\dots,m+n\} \,\rangle,
\end{equation}
and then define the following truncation of $\cM_{m|n}$:
\begin{equation}\label{eqTrunM}
  \cM_{r|s}^{k|l}:=(\cM_{m|n})^{\CL(\Upsilon_{k|l})\ot \cR(\Upsilon_{r|s})}=\{f\in\cM_{m|n}\mid \CL_{\Upsilon_{k|l}}(f)=\cR_{\Upsilon_{r|s}}(f)=f\}.
\end{equation}
Introduce the set $\hat{\bI}_{k|l}:=\{ a\mid m-k+1\leq a\leq m+l \}$,
and similarly the set $\hat{\bI}_{r|s}$. Then the elements $\{K_a, E_{b,b+1},E_{b+1,b} \mid a\in \hat{\bI}_{k|l}, b\in \hat{\bI}_{k|l}, b\neq m+l\}$ generate a Hopf subalgebra $\Uq(\gl_{k|l})$ of $\Uqg$. Clearly, $\Upsilon_{k|l}$ commutes with the subalgebra $\Uq(\gl_{k|l})$.

We obtain the following presentation for $\cM_{r|s}^{k|l}$.

\begin{lem}\label{lemtrunM}
	The subalgebra $\cM_{r|s}^{k|l}$ of $\cM_{m|n}$ is generated by
	$\{t_{ab}\mid a\in \hat{\bI}_{k|l},\ b\in \hat{\bI}_{r|s}\}$ subject to the relevant relations of \eqref{eqRel1}.
\end{lem}
\begin{proof}
	Let $\prod_{k|l}=\prod_{a=1}^{m-k}\prod_{b=m+l+1}^{m+n}K_aK_b$, and similarly introduce $\prod_{r|s}$. Observe that for any $t_{ab}\in \cM_{m|n}$
	\[
	\CL_{\prod_{k|l}}(t_{ab})=
	\begin{cases}
	t_{ab}, \ &a\in \hat{I}_{k|l},\\
	q_at_{ab},\ & \text{otherwise},
	\end{cases}
	\quad
	\cR_{\prod_{r|s}}(t_{ab})=
		\begin{cases}
		t_{ab}, \ &a\in \hat{I}_{r|s},\\
		q_bt_{ab},\ & \text{otherwise}.
		\end{cases}	
	\]
   Therefore, $t_{a_1,b_1}t_{a_2,b_2}\dots t_{a_p,a_p}\in \cM^{k|l}_{r|s}$ if and only if $a_i\in \hat{\bI}_{k|l}$ and $b_i\in \hat{\bI}_{r|s}$ for all $1\leq i\leq p$. This implies that $\cM^{k|l}_{r|s}$ is generated by $\{t_{ab}\mid a\in \hat{\bI}_{k|l},\ b\in \hat{\bI}_{r|s}\}$, while the relations follows directly from \eqref{eqRel1}.
\end{proof}

The following theorem is the quantum Howe duality of type $(\Uq(\gl_{k|l}),\Uq(\gl_{r|s}))$ applied to the subalgebra $\cM_{r|s}^{k|l}$ of $\cM_{m|n}$ with $k,r\leq m$ and $l,s\leq n$, which is a generalisation of \cite[Theorem 2.2]{WZ}.

\begin{thm}\label{thmHowe} {\rm ($(\Uq(\gl_{k|l}),\Uq(\gl_{r|s}))$-duality)}
 The superalgebra $\cM_{r|s}^{k|l}$ admits a multiplicity-free decomposition  as $\CL(\Uq(\gl_{k|l}))\ot \cR(\Uq(\gl_{r|s}))$-module
 \begin{equation}\label{eqHoweDec}
  \cM_{r|s}^{k|l}\cong \bigoplus_{\la\in \La_{k|l}\cap \La_{r|s}} L^{k|l}_{\la}\ot L^{r|s}_{\la},
 \end{equation}
 where $\La_{k|l}$ (resp. $\La_{r|s}$) is the set of $(k|l)$ (resp. $(r|s)$) hook partitions.
\end{thm}
\begin{proof}
	By \thmref{thmPeterWeyl}, we have
	\[
	\cM_{r|s}^{k|l}=(\cM_{m|n})^{\CL(\Upsilon_{k|l})\ot \cR(\Upsilon_{r|s})}\cong
	\bigoplus_{\la\in \La_{m|n}} (L^{m|n}_{\la})^{\Upsilon_{k|l}}\ot (L^{m|n}_{\la})^{\Upsilon_{r|s}}.	
	\]
	By \lemref{lemTrunh} and \lemref{lemTrunl},
	\[(L^{m|n}_{\la})^{\Upsilon_{k|l}}=((L^{m|n}_{\la})^{\Upsilon_{m-k}})^{\overline{\Upsilon}_{m+l}}\cong(L_{\la}^{k|n})^{\overline{\Upsilon}_{m+l}}\cong L_{\la}^{k|l}, \]
	where the second isomorphism requires $\la_{k+1}\leq n$ and the last one $\la_{k+1}\leq l$, yielding $\la\in \La_{k|l}$. The same argument applies to $(L^{m|n}_{\la})^{\Upsilon_{r|s}}$. Thus,  $(L^{m|n}_{\la})^{\Upsilon_{k|l}}\ot (L^{m|n}_{\la})^{\Upsilon_{r|s}}\cong L^{k|l}_{\la}\ot L^{r|s}_{\la}$ if  $\la\in \La_{k|l}\cap \La_{r|s}$, and  it will be 0 otherwise.
\end{proof}

\begin{rmk}\label{rmkHowedual}
	Similarly, we can show that as $\widetilde{\CL}(\Uq(\gl_{k|l}))\ot \cR(\Uq(\gl_{r|s}))$-module, 
	\[
	\overbar{\cM}_{r|s}^{k|l}\cong \bigoplus_{\la\in \La_{k|l}\cap \La_{r|s}} L^{k|l}_{\la}\ot (L^{r|s}_{\la})^{\ast}.\]
\end{rmk}

\noindent {\bf Notation}. 
%We shall adopt the following notation  in the remainder of this paper:
%\begin{enumerate}
%	\item Let $M$ be  $\Uq(\gl_{k|l})$-module (resp. $\Uq(\gl_{k|l})\ot \Uq(\gl_{r|s})$-module) for any $k,l,r,s\in \Z_{+}$.  We denote by $M|_{q=1}$ the corresponding  $\U(\gl_{k|l})$-module (resp. $\U(\gl_{k|l})\ot \U(\gl_{r|s})$-module).
 We shall regard $\cM^{k|l}_{r|s}$ as the superalgebra  generated by $t_{ab}$ with $a\in \bI_{k|l}$ and $b\in \bI_{r|s}$ subject to relations of the form \eqref{eqRel1}, as the two superalgebras are isomorphic.
%\end{enumerate}

\begin{rmk}\label{rmkHowecl} 
	We have following facts in the classical case ($q\rightarrow 1$): 
	\begin{enumerate}
		\item By specialising $q$ to 1,  relations $\eqref{eqRel1}$ reduce to the supercommutative relations $[t_{ab}, t_{cd}]=0$. In this case, we denote by $\cM^{k|l}_{r|s}|_{q=1}$ the superalgebra over $\C$ generated by $t_{ab}$.    This superalgebra is isomorphic to the supersymmetric algebra $S(\C^{k|l}\ot \C^{r|s})$, where $\C^{k|l}$ and $\C^{r|s}$ are respectively natural representations of $\gl_{k|l}$ and $\gl_{r|s}$.  Thus, \thmref{thmHowe} recovers classical Howe duality for $S(\C^{k|l}\ot \C^{r|s})$ (see \cite[Theorem 3.2]{CW01} and \cite[Theorem 1.3]{S2}), that is, the decomposition \eqref{eqHoweDec} holds for $\cM^{k|l}_{r|s}|_{q=1}$ in the case of classical limit $q\rightarrow 1$. In general, we will show that $\cM^{k|l}_{r|s}$ is isomorphic to a quantum analogue of supersymmetric algebra; see \secref{secbra}. 
		\item For any $k,l\in \Z_{+}$, we denote by $L_{\la}^{k|l}|_{q=1}$ the irreducible $\gl_{k|l}$-module corresponding to the irreducible $\Uq(\gl_{k|l})$-module $L_{\la}^{k|l}$. It was proved in \cite[Proposition 3]{Z93} that  $\dim_{\C} L_{\la}^{k|l}|_{q=1}= \dim_{\CK}L_{\la}^{k|l}$. This will be used frequently in what follows.
	\end{enumerate}
\end{rmk}

As a quick application of the quantum Howe duality, we obtain an explicit PBW basis for $\cM^{k|l}_{r|s}$, which is a special case of \cite[Theorem 1.14]{YM}.
Let $\bm=(m_{ab}), a\in \bI_{k|l}, b\in \bI_{r|s}$ be a  $(k+l)\times (r+s)$ supermatrix with parity assignment $[m_{ab}]=[a]+[b]$ such that $m_{ab}\in \Z_{+}$ whenever $[m_{ab}]=\bar{0}$ and $m_{ab}\in \{0,1\}$ whenever $[m_{ab}]=\bar{1}$. We denote by $\fM^{k|l}_{r|s}$ the  set of all such supermatrices $\bm$. Introduce a linear order $\succ$ for the pairs $(a,b)$ such that 
\begin{equation}\label{eqorder}
(a,b)\succ (c, b+k), \quad (a,b)\succ (a+k,b), \quad \forall k>0.
\end{equation}
Define monomial $t^{\bm}\in \cM^{k|l}_{r|s}$ by
\[ t^{\bm}=\prod_{(a,b)}^{\succ}t_{ab}^{m_{ab}}=t_{11}^{m_{11}}t_{21}^{m_{21}}\cdots t_{12}^{m_{12}}t_{22}^{m_{22}}\cdots, \quad \bm\in \fM^{k|l}_{r|s}, \] 
where the factors are arranged decreasingly in the order $\succ$.  Note that $\cM^{k|l}_{r|s}$ is $\Z_{+}$-graded with the gradation  $\deg t_{ab}=1$. We write $|\bm|:=\sum_{a,b}m_{ab}$ for the degree of monomial $t^{\bm}$, and denote by $(\cM^{k|l}_{r|s})_{N}$ the homogeneous subspace of degree $N$ in $\cM^{k|l}_{r|s}$. 
%Clearly, $t^{\bm}\in (\cM^{k|l}_{r|s})_{N}$ for any $\bm\in \fM^{k|l}_{r|s}$ with $|\bm|= N$

\begin{lem}\label{lemIden}
We have the following identities:
\[ 
\begin{aligned}
\!\!\!\!\sum_{\la\in \La_{k|l}\cap \La_{r|s},\; |\la|= N}\!\!\!\! \dim_{\CK} L^{k|l}_{\la}\dim_{\CK}  L^{r|s}_{\la}&=\!\!\!\!\sum_{\la\in \La_{k|l}\cap \La_{r|s},\; |\la|= N}\!\!\!\! \dim_{\C} L^{k|l}_{\la}|_{q=1}\dim_{\C}  L^{r|s}_{\la}|_{q=1} 
\\
&=\#\{ t^{\bm}\mid \bm\in \fM^{k|l}_{r|s}, |\bm|= N\},
\end{aligned} 
\]
where $\#S$ stands for the cardinality of set S.
\end{lem}

\begin{proof}
    The first identity follows from \rmkref{rmkHowecl}, so we only need to prove the second one.
   It is easy to  see that $t^{\bm},\bm\in  \fM^{k|l}_{r|s}$ form a basis for $\cM^{k|l}_{r|s}|_{q=1}$, which is isomorphic to the super polynomial algebra generated by $kr+ls$ even variables and $kl+rs$ Grassmannian variables (cf. \cite[Theorem 3.1]{SZ}).   Let $(\cM^{k|l}_{r|s}|_{q=1})_{N}$ be the homogeneous subspace of degree $N$ in $\cM^{k|l}_{r|s}|_{q=1}$. Then we obtain
   \[ \dim_{\C} (\cM^{k|l}_{r|s}|_{q=1})_{N}= \#\{ t^{\bm}\mid \bm\in \fM^{k|l}_{r|s}, |\bm|= N\}.  \]
   On the other hand, it follows from classical  Howe duality for $\cM^{k|l}_{r|s}|_{q=1}$ that
   \[ \dim_{\C} (\cM^{k|l}_{r|s}|_{q=1})_{N}=\!\!\!\!\sum_{\la\in \La_{k|l}\cap \La_{r|s},\; |\la|= N}\!\!\!\! \dim_{\C} L^{k|l}_{\la}|_{q=1}\dim_{\C}  L^{r|s}_{\la}|_{q=1}.  \]
   This proves the second identity.
\end{proof}

We obtain the following PBW basis for superalgebra  $\cM^{k|l}_{r|s}$.
\begin{prop}\label{propMbasis}
	The set of monomials $\{t^{\bm}\mid \bm\in \fM^{k|l}_{r|s}  \}$ constitutes a $\CK$-basis  for  $\cM^{k|l}_{r|s}$. 
\end{prop}
\begin{proof}
	We deduce from relations $\eqref{eqRel1}$ that $\cM^{k|l}_{r|s}$ is spanned by the set of given monomials, and hence it remains to show the linear independence.  By \thmref{thmHowe},  we have
	\begin{equation}\label{eqdimleqN}
	(\cM^{k|l}_{r|s})_{N} \cong \bigoplus_{\la\in \La_{k|l}\cap \La_{r|s},\; |\la|=N} L^{k|l}_{\la}\ot L^{r|s}_{\la}, 
	\end{equation}
	which implies that 
	\[ \dim_{\CK} (\cM^{k|l}_{r|s})_{N}=\!\!\!\!\sum_{\la\in \La_{k|l}\cap \La_{r|s},\; |\la|= N}\!\!\!\! \dim_{\CK} L^{k|l}_{\la}\dim_{\CK}  L^{r|s}_{\la}.   \]
    Combing this and \lemref{lemIden}, we obtain  that
    \[ \dim_{\CK} (\cM^{k|l}_{r|s})_{N}=\#\{ t^{\bm}\mid \bm\in \fM^{k|l}_{r|s}, |\bm|= N\}. \]	
	Since the monomials $t^{\bm}$ with $|\bm|= N$ span the degree $N$ homogeneous subspace $(\cM^{k|l}_{r|s})_{N}$, they must be linearly independent. This completes our proof.
\end{proof}

We now turn to another formulation of $\cM^{k|l}_{r|s}$.

\subsection{Braided supersymmetric algebra and flat modules} \label{secbra}
 Let $V$ be a finite dimensional module over $\Uqg$, and $R_{V,V}\in \GL(V\ot V)$ be the  associated $R$-matrix. Recall that $\widecheck{R}_{V,V}=PR_{V,V}$, where $P$ is the  graded permutation map \eqref{eq:permut}. Then $\widecheck{R}_{V,V}\in \End_{\Uqg}(V\ot V)$ with
 \begin{equation}\label{eqRbraidrel}
 (\widecheck{R}_{V,V}\ot \id_V)(\id_V\ot \widecheck{R}_{V,V})(\widecheck{R}_{V,V}\ot \id_V)=(\id_V\ot \widecheck{R}_{V,V})(\widecheck{R}_{V,V}\ot \id_V)(\id_V\ot \widecheck{R}_{V,V}).
 \end{equation}
 It is well-known that $\widecheck{R}_{V,V}$ acts on $V\ot V$ semi-simply, and its eigenvalues are of the form $\pm q^{r}$ with $r\in \frac{1}{2}\Z$. %; see similar case in \cite[\S 8.4.3]{KS}.   
Thus $\widecheck{R}_{V,V}$ obeys a minimal characteristic polynomial of the form
 \[
 \prod_{i=1}^{k^{+}}(\widecheck{R}_{V,V}-q^{\chi_i^{+}})  
 \prod_{i=1}^{k^{-}}(\widecheck{R}_{V,V}+q^{\chi_i^{-}})=0,
 \]
where $k^{\pm}\in\Z_{+}$ and $\chi_i^{+}, \chi_i^{-}\in \frac{1}{2}\Z$. 
 Define the following submodules of $V\ot V$ 
 \begin{equation*}
 \begin{aligned}
 S^{2}_{q}(V)&=\{  w\in V\ot V   \mid \prod_{i=1}^{k^{+}}(\widecheck{R}_{V,V}-q^{\chi_i^{+}})(w)=0\}, \\
 \La^{2}_{q}(V)&=\{  w\in V\ot V   \mid \prod_{i=1}^{k^{-}}(\widecheck{R}_{V,V}+q^{\chi_i^{-}})(w)=0\}.
 \end{aligned}
 \end{equation*}
 
 We generalise braided symmetric algebra in the sense of \cite{BZ} as follows.
  \begin{defn}\label{defnBraid} 
 	Let $V$ be  $\Z_2$-graded finite-dimensional $\Uqg$-module. We define the braided supersymmetric algebra and braided superexterior algebra respectively by
 	\[ S_q(V)=T(V)/\langle \La^{2}_q(V)\rangle,\quad \La_q(V)=T(V)/\langle S^{2}_q(V)\rangle, \]
 	where $T(V)$ is the tensor superalgebra of $V$ and $\langle I
 	\rangle$ is the 2-sided ideal  generated by subset $I\subset T(V)$.
 \end{defn}

Let $V|_{q=1}$ be the complex vector space spanned by the same basis elements as in $V$. The  supersymmetric algebra $S(V|_{q=1})$ over $\C$ is  $\Z_{+}$-graded,  so is $S_q(V)$. As a quotient module of $T(V)$, $S_q(V)$  encodes $\Uqg$-module structure  which preserves the algebraic structure, and hence it is a module superalgebra. 
 
 Following \cite{BZ}, we will call $V$ \emph{flat module}  if and only if $\dim_{\CK} S_q(V)_N=\dim_{\C} S(V|_{q=1})_N$ for all $N\in \Z_{+}$, and $S_q(V)$  is called the \emph{flat deformation} of $S(V|_{q=1})$. The flat deformation $S_q(V)$ acts as quantum analogue of supersymmetric algebra. However, other than natural modules, flat modules are extremely rare even in the quantum group case; see \cite[\S 2.2]{LZZ}.

 We now give some  concrete examples of flat modules. Let $V^{k|l}$ be the natural representation for $\Uq(\gl_{k|l})$ with the standard weight basis $\{v_i\}_{i\in\bI_{k|l}}$. Then $\widecheck{R}_{V^{k|l},V^{k|l}}$ gives an automorphism of $V^{k|l}\ot V^{k|l}$ with the action 
 \begin{equation*}
 \widecheck{R}_{V^{k|l},V^{k|l}}(v_i\ot v_j)=
 \begin{cases}
 (-1)^{[i][j]}v_j\ot v_i, & \text{if}\ i<j, \\
 (-1)^{[i]}q_iv_i\ot v_i, & \text{if}\ i=j,\\
 (-1)^{[i][j]}v_j\ot v_i+(q-q^{-1})v_i\ot v_j,& \text{if}\ i>j.
 \end{cases}
 \end{equation*}    
 The tensor product $V^{k|l}\ot V^{k|l}$ can be decomposed as direct sum of two irreducible $\Uq(\gl_{k|l})$-submodules,
 \[V^{k|l}\ot V^{k|l}=S^{2}_{q}(V^{k|l})\oplus \La^{2}_{q}(V^{k|l}).\]
 The basis for $ S^{2}_{q}(V^{k|l})$ is given by 
 \begin{equation}\label{eqbaS}
 v_i\ot v_i, \ 1\leq i\leq k,\quad
 v_i\ot v_j+(-1)^{[i][j]}qv_j\ot v_i, \ 1\leq i<j\leq k+l,
 \end{equation}
 while the basis for $\La^{2}_{q}(V^{k|l})$ is 
 \begin{equation} \label{eqbaLa}
 v_i\ot v_i, \  k+1\leq i\leq k+l, \quad
 v_i\ot v_j-(-1)^{[i][j]}q^{-1}v_j\ot v_i,  \  1\leq i<j\leq k+l. 
 \end{equation} 
 Let $P^{(k|l)}_{\;s}$ and $P^{(k|l)}_{\;a}$ be the idempotent projection onto the irreducible submodules $S^{2}_{q}(V^{k|l})$ and $\La^{2}_{q}(V^{k|l})$, respectively. Then  we have $\widecheck{R}_{V^{k|l},V^{k|l}}=qP^{k|l}_{\;s}-q^{-1} P^{k|l}_{\;a},$
 which leads to 
 \begin{equation}\label{eqRquad}
 (\widecheck{R}_{V^{k|l},V^{k|l}}-q)(\widecheck{R}_{V^{k|l},V^{k|l}}+q^{-1})=0.
 \end{equation}
  
 \begin{prop} \label{propQspace}
   Let $k,l\in \Z_{+}$, we have 
 	\begin{enumerate}
 		\item As a superalgebra, $\cM^{k|l}_{1|0}\cong S_q(V^{k|l})$, which is  generated by $x_1,x_2,\dots,x_{k+l}$ with parity assignments $[x_i]=[i]$ and defining relations
 		\[ 
 		\begin{aligned}
 		&x_i^2=0, \quad \text{if}\ [i]=\bar{1},\\
 		& x_jx_i=(-1)^{[i][j]}qx_ix_j\quad 1\leq i<j\leq k+l.
 		\end{aligned}	
 		\]	
 		\item The natural representation $V^{k|l}$ is flat, and $S_q(V^{k|l})=\bigoplus_{N\in \Z_{+}} S_q(V^{k|l})_N$, 
 		where  $S_q(V^{k|l})_N$ is the irreducible $\Uq(\gl_{k|l})$-module with the highest weight $N\epsilon_1$. 
 	\end{enumerate}
 \end{prop}
 \begin{proof}
 	 The relations for $S_q(V^{k|l})$ in Part (1) come  from \eqref{eqbaLa}, and the isomorphism is given by $t_{i1}\mapsto x_i$ for  $1\leq i\leq k+l$. By \thmref{thmHowe}, as $\Uq(\gl_{k|l})\ot \Uq(\gl_{1|0})\cong \Uq(\gl_{k|l})$ module, 
 	\[ S_q(V^{k|l})\cong \cM^{k|l}_{1|0}\cong \bigoplus_{\la\in \La_{k|l}\cap\La_{1|0}}L^{k|l}_{\la}\ot_{\CK} \CK, \]
 	where the sum is over all one-row Young diagrams. For any  one-row Young diagram $\la$ with $N$ boxes, we obtain $S_q(V^{k|l})_N \cong L^{k|l}_{\la}$ and
 	\[ 
 	\dim_{\CK}S_q(V^{k|l})_N=\dim_{\CK}L^{k|l}_{\la}=\dim_{\C}L^{k|l}_{\la}|_{q=1}=\dim_{\C} S(\C^{k|l})_N.
 	\] 
 	Thus, $V^{k|l}$ is a flat $\Uq(\gl_{k|l})$-module.
 \end{proof}

The following proposition will not be used later, but is interesting in its own right. 
Recall that  Manin \cite{YM}  introduced  two multiparameter quantum superspaces $A_q$ and $A_q^{\ast}$ of superdimensions $(k|l)$. In the one-parameter case $A_q\cong \cM^{k|l}_{1|0}$ as superalgebras, while $A_q^{\ast}\cong \cM^{k|l}_{0|1}$ as shown in the following proposition.

\begin{prop} \label{propextalg}
 	Let $k,l\in \Z_{+}$, we have
 	\begin{enumerate}
 		\item  As a superalgebra, $\cM^{k|l}_{0|1}\cong A_q^{\ast}$, which is  generated by $\xi_1,\xi_2,\dots,\xi_{k+l}$ with parity assignments $[\xi_i]=[i]+\bar{1}$ and defining relations
 		\[ 
 		\begin{aligned}
 		&\xi_i^2=0, \quad \text{if}\ [i]=\bar{0},\\
 		& \xi_j\xi_i=(-1)^{([i]+\bar{1})([j]+\bar{1})}q^{-1}\xi_i\xi_j\quad 1\leq i<j\leq k+l.
 		\end{aligned}	
 		\] 
 	    \item $A_q^{\ast}=\bigoplus_{N\in \Z_{+}} (A_{q}^{\ast})_N$, where $(A_{q}^{\ast})_N$ is the irreducible $\Uq(\gl_{k|l})$-module with the highest weight $\sum_{i=1}^{N}\epsilon_i$ if $N<k$, and $\sum_{i=1}^{k}\epsilon_i+(N-k)\epsilon_{k+1}$ otherwise.
 	\end{enumerate}
 \end{prop}
 \begin{proof}
 	This can be proved similarly as in \propref{propQspace}.
 \end{proof}

 The following proposition gives the second formulation of the module superalgebra $\cM^{k|l}_{r|s}$ via the braiding operator on $V^{k|l}\ot V^{r|s}$.
 
 \begin{prop}\label{propBraidedSym}
 	Let $V^{k|l}$ and $V^{r|s}$ be natural modules of $\Uq(\gl_{k|l})$ and $\Uq(\gl_{r|s})$, respectively.
 	\begin{enumerate}
 		\item 	As a superalgebra, 
 		$\cM^{k|l}_{r|s}\cong S_q(V^{k|l}\ot V^{r|s}).$
 	    \item  As a $\Uq(\gl_{k|l})\ot\Uq(\gl_{r|s})$-module, 
 	    \[ S_q(V^{k|l}\ot V^{r|s})\cong \bigoplus_{\la\in \La_{k|l}\cap \La_{r|s}} L^{k|l}_{\la}\ot L^{r|s}_{\la}.  \]
 		In particular, $V^{k|l}\ot V^{r|s}$ is a flat $\Uq(\gl_{k|l})\ot\Uq(\gl_{r|s})$-module.
 	\end{enumerate}
 \end{prop}
 \begin{proof}
 	Let $U=V^{k|l}\ot V^{r|s}$, then it has a basis $\{x_{ia}:=v_i\ot v_a\mid i\in\bI_{k|l}, a\in\bI_{r|s} \}$,
 	where $v_i$ and $v_a$ are basis elements for $V^{k|l}$ and $V^{r|s}$, respectively. Define the permutation map $P_{23}: U\ot U\rightarrow U\ot U$ of two middle factors  by
 	\[ P_{23}(v_i\ot v_j\ot v_a\ot v_b)=(-1)^{[j][a]}P_{23}(v_i\ot v_a\ot v_j \ot v_b).  \]
 	Therefore, the $R$-matrix of $\Uq(\gl_{k|l}\times \gl_{r|s})$ acting on $U\ot U$ is 
 	$ \widecheck{R}_{U,U}=P_{23} \circ(\widecheck{R}_{V^{k|l},V^{k|l}}\ot \widecheck{R}_{V^{r|s},V^{r|s} })\circ P_{23},$
 	which implies $$ \La_q^{2}(U)=P_{23}\bigg( \big(S^{2}_{q}(V^{k|l})\ot \La^{2}_{q}(V^{r|s}) \big) \bigoplus \big(\La^{2}_{q}(V^{k|l})\ot S^{2}_{q}(V^{r|s})\big)   \bigg).  $$
 Using bases for $S_q^2(V^{k|l})$ and $\La_q^2(V^{k|l})$ ($S_q^2(V^{r|s})$ and $\La_q^2(V^{r|s})$) given in \eqref{eqbaS} and \eqref{eqbaLa}, we immediately obtain the following quadratic relations for $S_q(U)=T(U)/\La_q^{2}(U)$:
 	\begin{equation*}\label{eqRelS}
 	\begin{aligned}
 	(x_{ia})^2=&0, &\quad &[i]+[a]=\bar{1},\\
 	x_{ja}x_{ia}=&(-1)^{([i]+[a])([j]+[a])}q_ax_{ia}x_{ja},&\quad &j>i,\\
 	x_{ib}x_{ia}=&(-1)^{([i]+[a])([i]+[b])}q_ix_{ia}x_{ib},&\quad &b>a, \\
 	x_{ja}x_{ib}=&(-1)^{([i]+[b])([j]+[a])}x_{ib}x_{ja},&\quad &j>i,\,a<b,\\
 	x_{jb}x_{ia}=&(-1)^{([i]+[a])([j]+[b])}x_{ia}x_{jb}\\
 				&+(-1)^{[i]([j]+[b])+[j][b]}(q-q^{-1})x_{ib}x_{ja},&\quad &j>i,\, b>a. 
 	\end{aligned}
 	\end{equation*}
 	It is straightforward to verify that the assignment $t_{ia}\mapsto (-1)^{[i][a]}x_{ia}$ preserves defining relations, which becomes the superalgebra isomorphism between $\cM^{k|l}_{r|s}$ and $S_q(V^{k|l}\ot V^{r|s})$. Thus we have
 	\[ \dim_{\CK}S_q(V^{k|l}\ot V^{r|s})_{N}=\dim_{\CK}(\cM^{k|l}_{r|s})_{N}=\sum_{\la\in \La_{k|l}\cap \La_{r|s},\ |\la|= N } \!\!\!\!\!\!\!\! \dim_\CK L_{\la}^{k|l}\ot L_{\la}^{r|s},   \]
 	while by \rmkref{rmkHowecl} in the classical case ($q\to1$)
 	\[ \dim_{\C} S(\C^{k|l}\ot \C^{r|s})_{N}=\sum_{\la\in \La_{k|l}\cap \La_{r|s},\ |\la|\leq N } \!\!\!\!\!\!\!\! \dim_\C L_{\la}^{k|l}|_{q=1}\ot L_{\la}^{r|s}|_{q=1}.   \]
 	Using $\dim_{\CK}L^{k|l}_{\la}=\dim_{\C}L^{k|l}_{\la}|_{q=1}$ for any $k,l\in \Z_{+}$, we have
 	$ \dim_{\CK}S_q(V^{k|l}\ot V^{r|s})_{N}=\dim_{\C} S(\C^{k|l}\ot \C^{r|s})_{N}$. This implies that  $V^{k|l}\ot V^{r|s}$ is a flat $\Uq(\gl_{k|l})\ot\Uq(\gl_{r|s})$-module. For part (2), it is easy to see that $S_q(V^{k|l}\ot V^{r|s})$ acquires a $\Uq(\gl_{k|l})\ot\Uq(\gl_{r|s})$-module structure through the isomorphism given in part (1). Thus  the decomposition in part (2) follows from \thmref{thmHowe}. 
% 	and hence it is an epimorphism. By the surjectivity of this homomorphism and \thmref{thmHowe}, 
% 	\[
% 	\begin{aligned}
% 	 \dim_{\CK}S_q(V^{k|l}\ot V^{r|s})_{N} &\leq \dim_{\CK}(\cM^{k|l}_{r|s})_{N}
% 	&=\sum_{\la\in \La_{k|l}\cap \La_{r|s},\ |\la|= N } \!\!\!\!\!\!\!\! \dim_\CK L_{\la}^{k|l}\ot L_{\la}^{r|s}
%% 	&=\sum_{\la\in \La_{k|l}\cap \La_{r|s},\ |\la|\leq N } \!\!\!\!\!\!\!\! \dim_\C L_{\la}^{k|l}|_{q=1}\ot L_{\la}^{r|s}|_{q=1}\\
%% 	&\geq 	
% 	\end{aligned}
% 	\]
% 	On the other hand,  we obtain 
% 	 \[ 
% 	 \begin{aligned}
% 	  \dim_{\CK} S_q(V^{k|l}\ot V^{r|s})_{\leq N}&\geq\ \dim_{\C} S(\C^{k|l}\ot \C^{r|s})_{\leq N}\\
% 	   &=\!\!\!\!\!\!\!\!\sum_{\la\in \La_{k|l}\cap \La_{r|s},\ |\la|\leq N } \!\!\!\!\!\!\!\! \dim_\C L_{\la}^{k|l}|_{q=1}\ot L_{\la}^{r|s}|_{q=1}.
% 	 \end{aligned}
% 	\]
% 	 Therefore, we have $\dim_{\CK}(\cM^{k|l}_{r|s})_{\leq N}=\dim_{\CK} S_q(V^{k|l}\ot V^{r|s})_{\leq N}=\dim_{\C} S(\C^{k|l}\ot \C^{r|s})_{\leq N}$ for any $N\in \Z_{+}$. This implies that the  homomorphism we have defined is an isomorphism and $V^{k|l}\ot V^{r|s}$ is a flat $\Uq(\gl_{k|l})\ot\Uq(\gl_{r|s})$-module.
 \end{proof}
 
\begin{rmk} \label{rmkdul}
       We immediately recover the following dualities:
	\begin{enumerate}
		\item (skew $(\Uq(\gl_k),\Uq(\gl_s))$-duality, \cite[Theorem 6.16]{LZZ}) As $\Uq(\gl_{k})\ot \Uq(\gl_{s})$-module, 
		\[ \La_{q}(V^{k|0}\ot V^{s|0})\cong S_q(V^{k|0}\ot V^{0|s})\cong \bigoplus_{\la \in \La_{k|0}\cap \La_{0|s}} L^{k|0}_{\la} \ot L^{s|0}_{\la^{\prime}}, \]
		where the first isomorphism follows from the defining relations of these two algebras (see \cite[Proposition 6.14]{LZZ}), and the second one has used the isomorphism $L^{0|s}_{\la}\cong L^{s|0}_{\la^{\prime}}$ as $\Uq(\gl_{s|0})$-modules (see part (1) of \propref{proplowwht}). Note $\Uq(\gl_{s}):=\Uq(\gl_{s|0})$.
		\item ($(\Uq(\gl_k),\Uq(\gl_r))$-duality, \cite[Theorem 6.4]{LZZ} and \cite{Z03})
		     As $\Uq(\gl_{k})\ot \Uq(\gl_{r})$-module,
			\[ S_{q}(V^{k|0}\ot V^{r|0})\cong \bigoplus_{\la\in \La_{k|0}\cap \La_{r|0}}  L^{k|0}_{\la}\ot L^{r|0}_{\la}. \]
	     \item  ($(\Uq(\gl_{k|l}),\Uq(\gl_r))$-duality,  \cite[Theorem 2.2]{WZ})
	     	As $\Uq(\gl_{k|l})\ot \Uq(\gl_{r})$-module,
	     	\[ S_q(V^{k|l}\ot V^{r|0})\cong \bigoplus_{\la\in \La_{k|l}\cap\La_{r|0}  } L^{k|l}_{\la}\ot L^{r|0}_{\la}.\]
	\end{enumerate}
\end{rmk} 
 
\begin{prop}\label{propBraidedSym2}
 As a superalgebra, $\overbar{\cM}^{k|l}_{r|s} \cong S_{q}((V^{k|l})^{\ast}\ot (V^{r|s})^{\ast} )$. Thus we have the following multiplicity-free decomposition as $\Uq(\gl_{k|l})\ot\Uq(\gl_{r|s})$-module
 \[ S_{q}((V^{k|l})^{\ast}\ot (V^{r|s})^{\ast} )\cong \bigoplus_{\la\in \La_{k|l}\cap \La_{r|s}} L_{\la}^{k|l}\ot (L_{\la}^{r|s})^{\ast}. \]
  In particular, $(V^{k|l})^{\ast}\ot (V^{r|s})^{\ast}$ is a flat $\Uq(\gl_{k|l})\ot\Uq(\gl_{r|s})$-module.
\end{prop}
 \begin{proof}
 	We only give a sketch of the proof, as it is similar to that of \propref{propBraidedSym}. Let $v_{i}^{\ast}, i\in \bI_{r|s},$ be standard weight basis for $(V^{k|l})^{\ast}$. Then $\widecheck{R}_{(V^{r|s})^{\ast},(V^{r|s})^{\ast}}$ acts on $(V^{r|s})^{\ast}\ot (V^{r|s})^{\ast}$ by
 	\begin{equation}\label{eqRmatact}
 	\widecheck{R}_{(V^{r|s})^{\ast},(V^{r|s})^{\ast}}(v_i^{\ast}\ot v_j^*)=
 	\begin{cases}
 	(-1)^{[i][j]}v_j^*\ot v_i^*, & \text{if}\ i>j \\
 	(-1)^{[i]}q_iv_i^*\ot v_i^*, & \text{if}\ i=j,\\
 	(-1)^{[i][j]}v_j^*\ot v_i^*+(q-q^{-1})v_i^*\ot v_j^*,& \text{if}\ i<j.
 	\end{cases}
 	\end{equation}  
 	Let $\overbar{U}=(V^{k|l})^{\ast}\ot (V^{r|s})^{\ast}$ and $P_{23}$ be graded the permutation of two middle factors in $\overbar{U}\ot \overbar{U}$. Then we have
 	\[ \La_q^{2}(\overbar{U})=P_{23}\bigg( \big(S^{2}_{q}((V^{k|l})^{\ast})\ot \La^{2}_{q}((V^{r|s})^{\ast}) \big) \bigoplus \big(\La^{2}_{q}((V^{k|l})^{\ast})\ot S^{2}_{q}((V^{r|s})^{\ast})\big)   \bigg).  \]
 	The bases for $S^{2}_{q}((V^{r|s})^{\ast})$ and $\La^{2}_{q}((V^{r|s})^{\ast})$ are given respectively by
 	$\{v_i^{\ast}\ot v_{i}^{\ast}, 1\leq i\leq r, \ v_i^{\ast}\ot v_{j}^{\ast}+(-1)^{[i][j]}q^{-1}v_{j}^{\ast}\ot v_i^{\ast}, 1\leq i< j\leq r+s\}$ and $\{v_i^{\ast}\ot v_{i}^{\ast}, r+1\leq i\leq r+s, \ v_i^{\ast}\ot v_{j}^{\ast}-(-1)^{[i][j]}q v_{j}^{\ast}\ot v_i^{\ast}, 1\leq i< j\leq r+s\}$. Now we can determine the quadratic relations for $S_q(\overbar{U})=T(\overbar{U})/\La_q^{2}(\overbar{U})$ as follows:
 		\begin{equation*}\label{eqRelS2}
 		\begin{aligned}
 		(x_{ia})^2=&0, &\quad &[i]+[a]=\bar{1},\\
 		x_{ja}x_{ia}=&(-1)^{([i]+[a])([j]+[a])}q_a^{-1}x_{ia}x_{ja},&\quad &j>i,\\
 		x_{ib}x_{ia}=&(-1)^{([i]+[a])([i]+[b])}q_i^{-1}x_{ia}x_{ib},&\quad &b>a, \\
 		x_{ja}x_{ib}=&(-1)^{([i]+[b])([j]+[a])}x_{ib}x_{ja},&\quad &j>i,\,a<b,\\
 		x_{jb}x_{ia}=&(-1)^{([i]+[a])([j]+[b])}x_{ia}x_{jb}\\
 		&-(-1)^{[i]([j]+[b])+[j][b]}(q-q^{-1})x_{ib}x_{ja},&\quad &j>i,\, b>a. 
 		\end{aligned}
 		\end{equation*}	
 Here $x_{ia}:=v_{i}^{\ast}\ot v_{a}^{\ast}$.	The isomorphism between  $\overbar{\cM}^{k|l}_{r|s}$ and $S_q(\overbar{U})$ is given by $\bar{t}_{ia}\mapsto (-1)^{[i][a]} x_{ia}$. The multiplicity-free decomposition is also clear from this isomorphism and \rmkref{rmkHowedual}.
 \end{proof}

\begin{rmk}
	Also we will  show in \propref{propBraidedSym3} that $V^{k|l}\ot (V^{r|s})^{\ast}$ is a flat $\Uq(\gl_{k|l})\ot\Uq(\gl_{r|s})$-module. Therefore, we have three  flat $\Uq(\gl_{k|l})\ot\Uq(\gl_{r|s})$-modules: $V^{k|l}\ot V^{r|s}$, $(V^{k|l})^{\ast}\ot (V^{r|s})^{\ast}$ and $V^{k|l}\ot (V^{r|s})^{\ast}$.
\end{rmk}
 
\subsection{Howe duality implies Schur-Weyl duality}
We derive the quantum super analogue of Schur-Weyl duality from the quantum super analogue of Howe duality following the methods of 
 \cite[Proposition 4.1]{Z03}, and also \cite[Proposition 2.4.5.3]{H92} for the classical case.

 Recall that  the Hecke algebra $\CH_q(r)$ of the symmetric group  is the unital $\CK$-algebra with generators $H_1,H_2,\dots,H_{r-1}$  subject to relations
\begin{equation}
\begin{aligned}
(H_i-q)(H_i+q^{-1})=0,\quad &1\leq i\leq r-1,\\
H_iH_j=H_jH_i,\quad &|i-j|\geq 2,\\
H_iH_{i+1}H_i=H_{i+1}H_iH_{i+1},\quad &1\leq i\leq r-2.     
\end{aligned}
\end{equation}
It is well known that the Hecke algebra $\CH_q(r)$  and the symmetric group algebra $\CK\Sym_r$ are isomorphic as associative algebras (cf. \cite{L}). Therefore, the irreducible representations of $\CH_q(r)$  are also indexed by partitions $\la\vdash r$. 

Let $V^{k|l}$ be the natural module for $\Uq(\gl_{k|l})$ with the standard basis $\{v_a\}_{a\in \bI_{k|l}}$. Then there is a natural action of $\CH_q(r)$ on $(V^{k|l})^{\ot r}$ defined by
\[  H_i.v_{a_1}\ot v_{a_2}\ot\dots\ot v_{a_r}:= v_{a_1}\ot\dots\ot \widecheck{R}_{V^{k|l},V^{k|l}}(v_{a_i}\ot v_{a_{i+1}})\ot\dots\ot v_{a_r},      \]
for all  $1\leq i\leq r$. 
This action is well defined  following \eqref{eqRbraidrel} and  the quadratic relation \eqref{eqRquad}, thus leads to the algebra  homomorphism
\begin{eqnarray}\label{eq:hecke-action}  
\nu_r: \CH_q(r) \longrightarrow \End_{\Uq(\gl_{k|l})} ((V^{k|l})^{\ot r}).  
\end{eqnarray}
The quantum super analogue of Schur-Weyl duality (see , e.g. \cite{Mit,DR}) states that this is actually an epimorphism. 

\begin{thm}\label{thmSchur}
	{\rm (Schur-Weyl duality)}
	Let $V^{k|l}$ be the natural representation of $\Uq(\gl_{k|l})$.  As  $\Uq(\gl_{k|l})\ot\CH_q(r)$-module, $(V^{k|l})^{\ot r}$ has the following multiplicity-free decomposition,  
	\[ (V^{k|l})^{\ot r}\cong\bigoplus_{\la\in\La_{k|l},\ |\la|=r} L_{\la}^{k|l}\ot D_{\la},  \]
    where $D_{\la}$ is the simple left $\CH_q(r)$-module associated to $\la$. This in particular implies that the algebra  homomorphism $\nu_r$ defined by \eqref{eq:hecke-action} is surjective.
\end{thm}
\begin{proof}
	It follows from \thmref{thmHowe} that 
	$ \cM^{k|l}_{r|0}\cong \bigoplus_{\la\in\La_{k|l}\cap \La_{r|0}} L_{\la}^{k|l}\ot L_{\la}^{r|0}$ as $\CL(\Uq(\gl_{k|l}))\ot \cR(\Uq(\gl_{r}))$-module,
	where $L_{\la}^{r|0}$ is the  irreducible $\Uq(\gl_{r})$-module. Denote by 
	\[ \cM^{k|l}_{r|0}[0]:=\{ f\in \cM^{k|l}_{r|0}\mid \cR_{K_a}(f)=qf,\quad \forall 1\leq a\leq r  \}  \]
    the zero weight space of $\cM^{k|l}_{r|0}$ with respect to the $\cR$-action of $\Uq(\gl_r)$. This is isomorphic to $(V^{k|l})^{\ot r}$ with a basis $\{t_{a_1,1}t_{a_2,2}\dots t_{a_r,r} \mid 1\leq a_1,a_2,\dots,a_r\leq k+l \}$ by \propref{propMbasis}.  Then we immediately obtain as left $\CL(\Uq(\gl_{k|l}))$ -modules
	\[ \cM^{k|l}_{r|0}[0] \cong (V^{k|l})^{\ot r}\cong \bigoplus_{\la\in\La_{k|l},\ |\la|=r} L_{\la}^{k|l}\ot L_{\la}^{r|0}[0],  \]
	where $L_{\la}^{r|0}[0]$ is the zero weight space of $L_{\la}^{r|0}$, and we have used the fact that $L_{\la}^{r|0}[0]$ is nonzero if and only if $|\la|=r$.
	Now it is well known that $L_{\la}^{r|0}[0]$ is actually the irreducible $\CH_q(r)$-module associated to partition $\la$ of size $r$, see \cite[Proposition 4.1]{Z03} (and
	\cite[Proposition 2.4.5.3]{H92} for the classical case). This proves the multiplicity-free decomposition. The second statement is clear. 
\end{proof}

\section{The FFT of invariant theory for $\Uqg$}
\subsection{Invariants for  $\Uqg$}
Fix non-negative integers $k,l, r, s$. In this section, we shall work with the following $\Uqg$-module superalgebra.
\begin{defn}\label{defnA}
	Let $\CP^{k|l}_{\,r|s}:=\cM^{\;k|l}_{m|n}\ot_{\fR} \overbar{\cM}^{\;r|s}_{m|n}$, the braided tensor algebra of $\cM^{\;k|l}_{m|n}$ and $\overbar{\cM}^{\;r|s}_{m|n}$. Then it is a $\Uqg$-module superalgebra with respect to the action $\cR(\Uqg)\ot \cR(\Uqg)$,  defined for any $f\ot g\in \CP^{k|l}_{\,r|s}$ and $x\in \Uqg$ by
	\[ x(f\ot g)=\sum_{(x)} (-1)^{[x_{(2)}][f]} \cR_{x_{(1)}}(f)\ot \cR_{x_{(2)}} (g).\]
\end{defn}

\begin{rmk}\label{rmkflatdef}
  By \propref{propBraidedSym} and \propref{propBraidedSym2}, $\cP^{k|l}_{\,r|s}$ is a  flat deformation of the  supersymmetric algebra  $S(\C^{k|l}\ot\C^{m|n}\oplus (\C^{r|s})^{\ast}\ot (\C^{m|n})^{\ast})$, which is isomorphic to $S(\C^{k|l}\ot\C^{m|n}\oplus \C^{r|s}\ot (\C^{m|n})^{\ast})$ as superalgebra.
\end{rmk}

The superalgebra algebraic structure of $\CP^{k|l}_{\,r|s}$ can be described as follows. 
\begin{lem}\label{lemmultP}
Let $T_{ai}=t_{ai}\ot 1$ and $\overline{T}_{bj}=1\ot \bar{t}_{bj}$ with $a\in \bI_{k|l}$, $b \in \bI_{r|s}$ and $i,j\in\bI_{m|n}$. Then $\CP^{k|l}_{\,r|s}$ is generated by the elements $T_{ai}$ and $\overline{T}_{bj}$ subject to the following relations 
for all  $a\in \bI_{k|l}$, $b \in \bI_{r|s}$ and $i,j\in\bI_{m|n}$:
\begin{equation}\label{eqRelModAlg}
\begin{aligned}
&T_{ai} \ \text{\  satisfy the relevant relations for $t_{ai}$ in \eqref {eqRel1}},\\
& \overline{T}_{bi} \ \text{\  satisfy the relevant relations for $\overline{t}_{bi}$  in \eqref{eqRel2}},\\
& \overline{T}_{bj}T_{ai}=(-1)^{([a]+[i])([b]+[j])}T_{ai}\overline{T}_{bj},\quad i\neq j\\
&\overline{T}_{bi}T_{ai}=(-1)^{([a]+[i])([b]+[i])}q_i^{-1}T_{ai}\overline{T}_{bi}-\sum_{j>i}(-1)^{[b]([a]+[i])+[a][j]}(q-q^{-1})T_{aj}\overline{T}_{bj}.
     \end{aligned} 
	\end{equation}
\end{lem}
\begin{proof} All the statements in the lemma are obvious except the last two relations, which require proof.  Note that they are equivalent to 
\begin{eqnarray}\label{eq:all-cases}
\overline{T}_{bj}T_{ai}=\sum_{a^{\prime},b^{\prime}\in\bI_{m|n}}(-1)^{([a]+[i])([b]+[b^{\prime}])+([i]+[j]+[b^{\prime}])([j]+[b^{\prime}])}(R^{-1})^{j\;a^{\prime}}_{b^{\prime}i} T_{aa^{\prime}}\overline{T}_{bb^{\prime}}.
\end{eqnarray}

Let $\fR=\sum_h \alpha_h\ot \beta_h\in\Uqg\ot \Uqg$  be the universal $\fR$ matrix. It follows from  \propref{proptensormodalg} that
\begin{equation*}
(1\ot\bar{t}_{bj})(t_{ai}\ot 1)=P\fR(\bar{t}_{bj}\ot t_{ai})
=\sum_h (-1)^{[\alpha_h][\beta_h]+([\alpha_h]+[j]+[b])([i]+[a])} \cR_{\beta_h}(t_{ai})\ot \cR_{\alpha_h} (\bar{t}_{bj}),
\end{equation*}
where we have the $\cR$-actions
\begin{align*}
\cR_{\beta_h}(t_{ai})&=\sum_{a^{\prime}\in \bI_{m|n}}(-1)^{([a]+[a^{\prime}])([i]+[a^{\prime}])+([\beta_h]+[i]+[a])[\beta_h]}t_{aa^{\prime}}\langle t_{a^{\prime}i},\beta_h\rangle,\\
\cR_{\alpha_h}(\bar{t}_{bj})&=\sum_{b^{\prime}\in\bI_{m|n}}(-1)^{([b]+[b^{\prime}])([b^{\prime}]+[j])+([\alpha_h]+[j]+[b])[\alpha_h]+[j]([j]+[b^{\prime}])}\bar{t}_{bb^{\prime}}\langle t_{jb^{\prime}},S(\alpha_h)\rangle.
\end{align*}
Recall  $\fR^{-1}=\sum_h S(\alpha_h)\ot \beta_h$ and  $R^{-1}=(\pi\ot \pi)\fR^{-1}=\sum_{a,b,c,d\in\bI_{m|n}}e_{ab}\ot e_{cd} (R^{-1})^{ac}_{bd}$. Note that 
$ \sum_{h} \langle t_{jb^{\prime}},S(\alpha_h)\rangle \langle t_{a^{\prime}i},\beta_h\rangle=(R^{-1})^{j\;a^{\prime}}_{b^{\prime}i}$
with the following nonzero entries of $R^{-1}$
\begin{equation}\label{eqRinvmat}
(R^{-1})^{a\,b}_{a\,b}=1,\; a\neq b,\quad (R^{-1})^{a\,a}_{a\,a}=q_a^{-1},\quad (R^{-1})^{a\,b}_{b\,a}=-(q_b-q_b^{-1}),\; a<b.
\end{equation}
Using above equations together, we obtain \eqref{eq:all-cases} as required.
\end{proof}

We want to study the subalgebra of $\Uqg$-invariants in $\CP^{k|l}_{\,r|s}$ (cf. \lemref{leminvalg}). Let
\begin{equation}\label{eqninv}
 \CX^{k|l}_{\;r|s}:=(\CP^{k|l}_{\,r|s})^{\Uqg}=\{ f\in \CP^{k|l}_{\,r|s} \mid  x(f)=\epsilon(x)f,\,\forall x\in \Uqg
 \}. 
\end{equation}
We define the following elements in $\CP^{k|l}_{\,r|s}$ by
\begin{eqnarray}\label{eq:inv-X}
X_{ab}:=\sum_{i\in \bI_{m|n}} (-1)^{[a]([b]+[i])} t_{ai}\ot \bar{t}_{bi}=\sum_{i\in \bI_{m|n}} (-1)^{[a]([b]+[i])}T_{ai}\overline{T}_{bi} 
\end{eqnarray}
for all $a\in \bI_{k|l}$ and $b\in \bI_{r|s}$. The $\Z_{2}$-gradation is given by $[X_{ab}]=[a]+[b]$.

\begin{lem}\label{lemXinv}
	The elements $X_{ab}$ belong to $\CX^{k|l}_{\;r|s}$ for all $a\in \bI_{k|l}$ and $b\in  \bI_{r|s}$.
\end{lem}
\begin{proof}
	For all $x\in \Uqg$, we have
\begin{equation}\label{eqxPhi}
  xX_{ab}  =  \sum_{(x)}\sum_{i\in \bI_{m|n}}(-1)^{[x_{(2)}]([a]+[i])+[a]([b]+[i])}\cR_{x_{(1)}}(t_{ai})\ot \cR_{x_{(2)}} (\bar{t}_{bi}).
\end{equation}
Note that 
\begin{equation}\label{eqRt}
  \begin{aligned}
   \cR_{x_{(1)}}(t_{ai})&=\sum_{c\in \bI_{m|n}} (-1)^{[x_{(1)}]([a]+[i]+[x_{(1)}])+([a]+[c])([c]+[i])}t_{ac}\langle t_{ci},x_{(1)}\rangle\\
   \cR_{x_{(2)}}(\bar{t}_{bi})&=\sum_{d\in \bI_{m|n}} (-1)^{[x_{(2)}]([b]+[i]+[x_{(2)}])+([b]+[d])([d]+[i])}\bar{t}_{bd}\langle \bar{t}_{di},x_{(2)}\rangle,
  \end{aligned}
\end{equation}
where $\langle t_{ci},x_{(1)}\rangle=\pi(x_{(1)})_{ci}$ and $\langle \bar{t}_{di},x_{(2)}\rangle=(-1)^{[i]([d]+[i])}\pi(S(x_{(2)}))_{id}$. Observe that 
\[
\pi(x_{(1)})_{ci}\neq 0 \Leftrightarrow [x_{(1)}]=[c]+[i],\quad \pi(S(x_{(2)}))_{id}\neq 0\Leftrightarrow [x_{(2)}]=[i]+[d].
\]
Combing \eqref{eqxPhi} and \eqref{eqRt}, we obtain
\[
\begin{aligned}
xX_{ab}&=\sum_{(x)}\sum_{i,c,d\in \bI_{m|n}}(-1)^{[a]([b]+[d])}\pi(x_{(1)})_{ci}\pi(S(x_{(2)}))_{id}T_{ac}\overline{T}_{bd}\\
&=\sum_{c,d\in\bI_{m|n}} (-1)^{[a]([b]+[d])}\pi(\epsilon(x))_{cd}T_{ac}\overline{T}_{bd}\\
&= \epsilon(x) X_{ab},
\end{aligned}
 \]
 where we have used the identity $\sum_{(x)}x_{(1)}S(x_{(2)})=\epsilon(x)1$.
\end{proof}

\begin{lem}\label{leminvrel} 
	\begin{enumerate}
		\item The following relations hold in $\CP^{k|l}_{\,r|s}$:
		\begin{align*}
		&X_{ac}T_{bi}=(-1)^{([b]+[i])([a]+[c])}T_{bi}X_{ac},\; &a>b,\\
		&X_{ac}T_{ai}=(-1)^{([a]+[i])([a]+[c])}q_{a}^{-1} T_{ai}X_{ac},\\
		T_{ai}X_{bc}-(-1)&^{([a]+[i])([b]+[c])}X_{bc}T_{ai}=(-1)^{[c]([a]+[b])+[a][b]}(q-q^{-1})T_{bi}X_{ac},\; &a>b,\\
		&X_{ab}\overline{T}_{bi}=(-1)^{([a]+[b])([b]+[i])}q_{b}\overline{T}_{bi}X_{ai},\\
		&X_{ab}\overline{T}_{ci}=(-1)^{([a]+[b])([c]+[i])}\overline{T}_{ci} X_{ab},\; &b>c,\\
		X_{ac}\overline{T}_{bi}-(-1)&^{([b]+[i])([a]+[c])}\overline{T}_{bi} X_{ac}=(-1)^{[i]([a]+[c])+[a][b]}(q-q^{-1})\overline{T}_{ci}X_{ab},\; &b>c.    
		\end{align*}
		\item The elements $X_{ab}\in \CX^{k|l}_{\;r|s}, a\in \bI_{k|l}, b\in\bI_{r|s}$ satisfy the following relations
		\begin{equation}\label{eqinvrel}
			\begin{aligned}
			(X_{ab})^{2}=&0,&\quad &[a]+[b]=\bar{1},\\
			X_{ac}X_{bc}=&(-1)^{([a]+[c])([b]+[c])}q_{c}X_{bc}X_{ac},&\quad &a>b,\\
			X_{ab}X_{ac}=&(-1)^{([a]+[b])([a]+[c])}q_{a}^{-1}X_{ac}X_{ab},&\quad&b>c,\\
			X_{ac}X_{bd}=&(-1)^{([a]+[c])([b]+[d])}X_{bd}X_{ac},&\quad  &a>b, c>d,\\
			X_{ac}X_{bd}=&(-1)^{([a]+[c])([b]+[d])}X_{bd}X_{ac}\\
			&   +  (-1)^{[a]([b]+[d])+[b][d]}(q-q^{-1})X_{bc}X_{ad},&\quad  &a>b,c<d.
			\end{aligned}	
		\end{equation}
	\end{enumerate}
\end{lem}

\begin{proof}
   The proof for part (1) is rather involved as there are many cases to consider, even though all are quite similar.  We illustrate the proof by using the third relation in part (1) as an example.
   Using relations \eqref{eqRelModAlg}, we obtain
   \[ 
      \begin{aligned}
      T_{ai}X_{bc}&=\sum_{j\in \bI_{m|n}} (-1)^{[b]([c]+[j])}T_{ai}T_{bj}\overline{T}_{cj}=S_{i>j}+S_{i<j}+S_{i=j}, \ \  \text{with} \\
      S_{j<i}=&\sum_{j<i} (-1)^{[b]([c]+[j])} \biggl( (-1)^{([a]+[i])([b]+[j])}T_{bj}T_{ai}\overline{T}_{cj}\\
     &+(-1)^{[a]([b]+[j])+[b][j]}(q-q^{-1})T_{bi}T_{aj}\overline{T}_{cj}
    \biggr), \\
    S_{j>i}=&\sum_{j>i} (-1)^{[b]([c]+[j])+([a]+[i])([b]+[j])}T_{bj}T_{ai}\overline{T}_{cj},\\
    S_{j=i}=&(-1)^{[b]([c]+[i])+([a]+[i])([b]+[i])}q_iT_{bi}T_{ai}\overline{T}_{ci}.
   \end{aligned}
   \]
   On the other hand, we have 
   \[
   \begin{aligned}
   X_{bc}T_{ai}&=\sum_{j\in\bI_{m|n}} (-1)^{[b]([c]+[j])} T_{bj}\overline{T}_{cj}T_{ai}=S^{\prime}_{j=i}+S^{\prime}_{j\neq i} \ \  \text{with}\\
         S^{\prime}_{j=i}=& (-1)^{[b]([c]+[i])} \biggl( (-1)^{([a]+[i])([c]+[i])}q_i^{-1} T_{bi}T_{ai}\overline{T}_{ci}\\
   &- \sum_{k>i}(-1)^{[c]([a]+[i])+[a][k]}(q-q^{-1}) T_{bi}T_{ak}\overline{T}_{ck} \!\biggr), \\
   S^{\prime}_{j\neq i}=& \sum_{j\neq i}(-1)^{[b]([c]+[j])+ ([a]+[i])([c]+[j])} T_{bj}T_{ai}\overline{T}_{cj}. 
   \end{aligned}
   \]
 Straightforward calculation shows that   
   \[ 
   T_{ai}X_{bc}-(-1)^{([a]+[i])([b]+[c])}X_{bc}T_{ai}=(-1)^{[c]([a]+[b])+[a][b]}(q-q^{-1})T_{bi}X_{ac},\; a>b.  
   \]

   Now we turn to part (2), which is an easy consequence of part (1). For instance, if $[a]+[b]=\bar{1}$, then
   \[ 
   \begin{aligned}
      (X_{ab})^{2}&=\sum_{i\in\bI_{m|n}}
      (-1)^{[a]([b]+[i])}X_{ab}T_{ai}\overbar{T}_{bi}\\
      &=\sum_{i\in \bI_{m|n}}(-1)^{[a]([b]+[i])+([a]+[b])([a]+[i])}q_a^{-1}T_{ai}X_{ab}\overbar{T}_{bi}\\
      &= \sum_{i\in \bI_{m|n}}(-1)^{[a]([b]+[i])+([a]+[b])}q_a^{-1}q_b T_{ai}\overbar{T}_{ai}X_{ab}\\
      &=-q_a^{-1}q_b(X_{ab})^{2}.
   \end{aligned}
    \]
  Hence $(X_{ab})^2=0$ in this case. The remaining relations can be proved in a similar way. 
\end{proof}

The following is one of the main results of this paper. 
\begin{thm}\label{thmFFT} {\rm (FFT for $\Uqg$)} 
	 The invariant subalgebra $\CX^{k|l}_{\;r|s}$ is generated by the elements $X_{ab}$ with $a\in \bI_{k|l}$ and $b\in  \bI_{r|s}$.
\end{thm}

\begin{rmk}The special case with $n=0$ of the theorem is the FFT of  invariant theory for the quantum general linear group \cite{LZZ}. One can also recover from this theorem the FFT of  invariant theory for  the general linear Lie superalgebra \cite{LZ1,S1,S2}  by specialising $q$ to 1. These points will be discussed in Section \ref{sect:appl}. 
\end{rmk}

\subsection{Proof of FFT}

The proof of \thmref{thmFFT}  will be  divided into two cases based on the following lemma.

\begin{lem}\label{leminv}
  The invariant subalgebra $\CX^{k|l}_{\;r|s}$ admits the multiplicity-free decomposition 
  \begin{equation}
   \CX^{k|l}_{\;r|s}\cong\bigoplus_{\la\in \La_{k|l}\cap \La_{r|s}\cap \La_{m|n}} L_{\la}^{k|l}\ot L_{\la}^{r|s}
  \end{equation}
  as $\CL(\Uq(\gl_{k|l}))\ot \widetilde{\CL}(\Uq(\gl_{r|s}))$-module.
\end{lem}

\begin{proof}
	By \defref{defnA} and \thmref{thmHowe}, we obtain that as a $\CL(\Uq(\gl_{k|l}))\ot \cR(\Uq(\gl_{m|n}))\ot \widetilde{\CL}(\Uq(\gl_{r|s}))\ot \cR(\Uq(\gl_{m|n}))$-module,
	\[ \CP^{k|l}_{\,r|s}\cong \bigoplus_{\substack{\la\in \La_{k|l}\cap\La_{m|n},\\ \mu\in \La_{r|s}\cap\La_{m|n}}} L_{\la}^{k|l}\ot L_{\la}^{m|n}\ot L_{\mu}^{r|s}\ot (L_{\mu}^{m|n})^{\ast}.\]
	It follows that
	\[ 
	\begin{aligned}
	\CX^{k|l}_{\;r|s}&\cong \bigoplus_{\substack{\la\in \La_{k|l}\cap\La_{m|n},\\ \mu\in \La_{r|s}\cap\La_{m|n}}} L_{\la}^{k|l}\ot  L_{\mu}^{r|s} \ot (L_{\la}^{m|n}\ot (L_{\mu}^{m|n})^{\ast})^{\Uqg}   \\
	&\cong \bigoplus_{\la\in \La_{k|l}\cap \La_{r|s}\cap \La_{m|n}} L_{\la}^{k|l}\ot L_{\la}^{r|s},
	\end{aligned}
	 \]
where we have used Schur's Lemma that $(L_{\la}^{m|n}\ot (L_{\mu}^{m|n})^{\ast})^{\Uqg}$  has a unique invariant (up to scalar multiples) if and only if $\la=\mu$.
\end{proof}

Motivated  by the approach to FFT of invariant theory in  \cite{LZZ}, we are in a position to deal with the following two cases respectively:
\begin{itemize}
	\item $m\geq \max\{k,r\}$ and $n\geq \max\{l,s\}$;
	\item $m< \max\{k,r\}$ or $n< \max\{l,s\}$.
\end{itemize}

\subsubsection{The case $m\geq \max\{k,r\}$ and $n\geq \max\{l,s\}$}\label{secfirstcase}
In this case,  one immediately obtains
 \begin{equation}\label{eqLafirst}
   \La_{k|l}\cap \La_{r|s}\cap \La_{m|n}=\La_{k|l}\cap \La_{r|s}.
 \end{equation}
($m\geq \max\{k,r\}$ and $n\geq \max\{l,s\}$ is only a sufficient condition for this equation.)  Viewing $\cM^{k|l}_{r|s}$ as a subalgebra of $\cM_{m|n}$, we have $\Delta(\cM_{r|s}^{k|l})\subset \cM_{m|n}^{\;k|l}\ot \cM_{\;\,r|s}^{m|n}$ under the co-multiplication on $\cM_{m|n}$. Let $\widetilde{\Delta}:=(\id\ot S)\Delta$ be the $\CK$-linear map
\[ \widetilde{\Delta}: \cM_{r|s}^{k|l}\longrightarrow \CP_{\,r|s}^{k|l}=\cM^{\;k|l}_{m|n}\ot_{\fR} \overbar{\cM}^{\;r|s}_{m|n},  \]
where $S$ is the antipode of $\cM_{m|n}$ restricted on $\cM^{m|n}_{\;\,r|s}$.

\begin{lem}\label{leminj}
 $\widetilde{\Delta}$ is injective.	
\end{lem}
\begin{proof}
   We have
   \[ (\epsilon\ot \id)\widetilde{\Delta}(f)=\sum_{(f)}1\ot \epsilon(f_{(1)})S(f_{(2)})=S(f),\quad \forall f\in \cM^{k|l}_{r|s}.  \]
   Thus, $\widetilde{\Delta}(f)=0$ if and only if $f=0$, as $S$ is 	invertible.
\end{proof}

\begin{lem}\label{lemXim}
	Assume that $m\geq \max\{k,r\}$ and $n\geq \max\{l,s\}$, then $\widetilde{\Delta}(\cM_{r|s}^{k|l})=\CX^{k|l}_{\;r|s}$. In particular, 
	$\widetilde{\Delta}(t_{ab})=X_{ab}\in \CX^{k|l}_{\;r|s}, \forall a\in \bI_{k|l}, b\in\bI_{r|s}$.
\end{lem}
\begin{proof}
    By \eqref{eqLafirst}, \lemref{leminv} and \thmref{thmHowe}, we have
    \[ \dim_{\CK} (\CX^{k|l}_{\;r|s})_{N}=\sum_{\la\in \La_{k|l}\cap\La_{r|s}, |\la|=N} \dim_{\CK}L^{k|l}_{\la}\times \dim_{\CK}L^{r|s}_{\la}=\dim_{\CK}(\cM_{r|s}^{k|l})_N,  \]
	for any homogeneous components of degree $N$. We only need to show that $\widetilde{\Delta}(\cM_{r|s}^{k|l})\subseteq \CX^{k|l}_{\;r|s}$, since the identity follows from the injectivity of $\widetilde{\Delta}$. Now for any $x\in \Uqg$ and $f\in \cM^{k|l}_{r|s}$, we have 
	\begin{equation*}\label{eqxinv}
	\begin{aligned}
	x\widetilde{\Delta}(f)&=\sum_{(x),(f)} (-1)^{[x_{(2)}][f_{(1)}]} \cR_{x_{(1)}}(f_{(1)})\ot \cR_{x_{(2)}} (S(f_{(2)})),\\
	&=\sum_{(x),(f)} (-1)^{[x_{(2)}]([f_{(1)}]+[f_{(2)}])}X_1\ot X_{2}, 
	\end{aligned}
	\end{equation*}
	with 
	\[
	\begin{aligned}
	  X_1&=(-1)^{([f_{(1)}]+[f_{(2)}]+[x_{(1)}])[x_{(1)}]}f_{(1)}  \langle f_{(2)},x_{(1)}\rangle,\\
	  X_{2}&= (-1)^{([f_{(3)}]+[f_{(4)}]+[x_{(2)}])[x_{(2)}]+[f_{(3)}][f_{(4)}]} S(f_{(4)})\langle S(f_{(3)}), x_{(2)}\rangle. 
	\end{aligned}	
	\]
   Note that we have used the identity $\Delta(S(f))=\sum_{(f)}(-1)^{[f_{(1)}][f_{(2)}]} S(f_{(2)})\ot S(f_{(1)})$. Thus, we obtain
   \[x\widetilde{\Delta}(f)= \sum_{(f)}(-1)^{[f_{(1)}][x]}\epsilon(x)f_{(1)}\ot S(f_{(2)})=\epsilon(x)\widetilde{\Delta}(f) , \]
   completing our proof. The second claim is easy to see by the antipode $S$ in \eqref{eqantipode}.
\end{proof}

To complete the proof of \thmref{thmFFT}, we need the following key lemma.

\begin{lem}\label{lemKeyFFT}
	 Let $\mu$ be  the multiplication in $\cM^{k|l}_{r|s}$. Suppose  that $m\geq \max\{k,r\}$ and $n\geq \max\{l,s\}$,  
	\begin{enumerate}
		\item There exists the $\Z_{+}\times \Z_{+}$-graded vector space bijection:
		      \[
		      \begin{aligned}
		        \varpi:& \quad\cM^{k|l}_{r|s}\ot \cM^{k|l}_{r|s}\longrightarrow \cM^{k|l}_{r|s}\ot \cM^{k|l}_{r|s},\\
		               & f\ot g \mapsto \sum_{(f),(g)} (-1)^{[f_{(2)}][g_{(1)}]} f_{(1)}\ot g_{(1)}\langle f_{(2)}\ot g_{(2)}, \fR^{-1}\rangle.
		      \end{aligned}
		      \]
		      In particular, for any degree $N$ ($N\in\Z_{+}$) graded component of $\cM^{k|l}_{r|s}$,
		      \[(\cM^{k|l}_{r|s})_N=\mu\circ \varpi ( (\cM^{k|l}_{r|s})_{N-1}\ot (\cM^{k|l}_{r|s})_1 ).  \]
		 \item Let $f,g\in \cM^{k|l}_{r|s}$, we have the multiplication formula in $\CX^{k|l}_{\;r|s}$ given by
		       \[ \widetilde{\Delta}(f)\widetilde{\Delta}(g)=\widetilde{\Delta}\circ \mu\circ \varpi(f\ot g).  \]
	\end{enumerate}
	
\end{lem}
\begin{proof}
  For part (1), we only need to show that $\varpi(f\ot g)$ is contained in $\cM^{k|l}_{r|s}\ot \cM^{k|l}_{r|s}$, since $\fR$ is invertible. Viewing $\cM^{k|l}_{r|s}$ as a subalgebra of $\cM_{m|n}$ obtained by the truncation \eqref{eqTrunM} and recalling the labelling set $\hat{\bI}_{k|l}$ therein,  we have for any $a,c \in\hat{\bI}_{k|l}$ and $b,d\in \hat{\bI}_{r|s}$,
  \[ \varpi(t_{ab}\ot t_{cd})=\sum_{a^{\prime},b^{\prime}}\sgn\  t_{aa^{\prime}}\ot t_{cc^{\prime}}\langle t_{a^{\prime}b}\ot t_{c^{\prime}d}, \fR^{-1}\rangle,      \]
  with $\sgn=(-1)^{([a]+[a^{\prime}])([a^{\prime}]+[b])+ ([c]+[c^{\prime}])([c^{\prime}]+[d])+([c]+[c^{\prime}])([a^{\prime}]+[b]) }$. Using the matrix form of $R^{-1}$ in \eqref{eqRinvmat}, we obtain that $\langle t_{a^{\prime}b}\ot t_{c^{\prime}d}, \fR^{-1}\rangle= 0$ unless $a^{\prime}, c^{\prime}\in \hat{\bI}_{r|s}$. In general, for any ordered monomials $f=t_{a_1b_1}\cdots t_{a_Mb_M}$ and $g=t_{c_1d_1}\cdots t_{c_Nd_N}$ with $b_i, d_i\in \hat{\bI}_{r|s}$, we have $\langle f\ot g, \fR^{-1}\rangle =0$ unless $a_i,c_i\in \hat{\bI}_{r|s}$ by the defining property of $\fR$ \eqref{eqDeR}.
  
  To prove part (2), we need the following technical lemma.
  \begin{lem}\label{lemKeylemlem} Maintaining above notation, we have
  	\[
  	\begin{aligned}
  	\sum_{(f),(g)}(-1)^{[g_{(1)}][f_{(2)}]} f_{(2)}g_{(2)}\langle f_{(1)}\ot g_{(1)}, \fR\rangle &=\sum_{(f),(g)}(-1)^{([f_{(1)}]+[f_{(2)}])[g_{(1)}]} g_{(1)}f_{(1)}\langle f_{(2)}\ot g_{(2)},\fR\rangle; \\ 
  \sum_{(f),(g)}(-1)^{[f_{(2)}][g_{(1)}]} f_{(1)}g_{(1)}\langle f_{(2)}\ot g_{(2)},\fR^{-1} \rangle&=\sum_{(f),(g)} (-1)^{([g_{(1)}]+[g_{(2)}])[f_{(2)}]}g_{(2)} f_{(2)} \langle f_{(1)}\ot g_{(1)}, \fR^{-1}\rangle. 
  	\end{aligned}
  	\]
%  	\begin{enumerate}[itemsep=1.5mm]
%  		\item  $\sum_{(f),(g)}(-1)^{[g_{(1)}][f_{(2)}]} f_{(2)}g_{(2)}\langle f_{(1)}\ot g_{(1)}, \fR\rangle= \sum_{(f),(g)}(-1)^{([f_{(1)}]+[f_{(2)}])[g_{(1)}]} g_{(1)}f_{(1)}\langle f_{(2)}\ot g_{(2)},\fR\rangle;  $ 
%  	    \item  $\sum_{(f),(g)} g_{(2)} f_{(2)} \langle f_{(1)}\ot g_{(1)}, \fR^{-1}\rangle= \sum_{(f),(g)}(-1)^{[f_{(2)}][g_{(2)}]} f_{(1)}g_{(1)}\langle f_{(2)}\ot g_{(2)},\fR^{-1} \rangle. $
%  	\end{enumerate}
  \end{lem}
  We postpone the proof of  \lemref{lemKeylemlem} to the end. Let $\fR=\sum_h \alpha_h\ot \beta_h$ and hence $ \fR^{-1}=\sum_h S(\alpha_h)\ot \beta_h$. Now for any $f,g\in \cM^{k|l}_{r|s}$,
  \[
  \begin{aligned}
   \widetilde{\Delta}(f)\widetilde{\Delta}(g)&=\sum_{(f),(g),h} (-1)^{([\alpha_h]+[f_{(2)}])[g_{(1)}]+[\alpha_h][\beta_h]} f_{(1)} \cR_{\beta_h}(g_{(1)}) \ot \cR_{\alpha_h}(S(f_{(2)}))S(g_{(2)})\\
   &=\sum_{(f),(g),h} (-1)^{([\alpha_h]+[f_{(2)}]+[f_{(3)}])([g_{(1)}]+[g_{(2)}])+[\alpha_h][\beta_h]} X_1\ot X_2
  \end{aligned}
  \]
  with 
  \[
  \begin{aligned}
  X_1&=(-1)^{([\beta_h]+[g_{(1)}]+[g_{(2)}])[\beta_h]} f_{(1)}g_{(1)}\langle g_{(2)},\beta_h\rangle,\\
  X_2&=(-1)^{([\alpha_h]+[f_{(2)}]+[f_{(3)}])[\alpha_h]+[f_{(2)}][f_{(3)}]+[f_{(3)}][g_{(3)}]}S(g_{(3)}f_{(3)}) \langle S(f_{(2)}),\alpha_h\rangle.
  \end{aligned}
  \]
The sum  is over $(f),(g),h$ such that $[\alpha_h]=[f_{(2)}]$, $[\beta_h]=[g_{(2)}]$ and $[f_{(2)}]+[g_{(2)}]=\bar{0}$ since
  \[ \langle S(f_{(2)}),\alpha_h\rangle \langle g_{(2)},\beta_h\rangle=(-1)^{[f_{(2)}][g_{(2)}]}\langle f_{(2)}\ot g_{(2)},\fR^{-1}\rangle. \]
  Using the second identity in \lemref{lemKeylemlem}, we obtain 
  \[
  \begin{aligned}
    \widetilde{\Delta}(f)\widetilde{\Delta}(g)&=\sum_{(f),(g)} (-1)^{[f_{(3)}]([g_{(1)}]+[g_{(2)}])+ [g_{(1)}][g_{(2)}]+[f_{(3)}][g_{(3)}]}f_{(1)}g_{(1)} S(g_{(3)}f_{(3)})\langle f_{(2)}\ot g_{(2)},\fR^{-1}\rangle\\  
    &= \sum_{(f),(g)} (-1)^{[f_{(3)}]([g_{(1)}]+[g_{(2)}])+[g_{(1)}][g_{(2)}]}  f_{(1)}g_{(1)} S(f_{(2)}g_{(2)})\langle f_{(3)}\ot g_{(3)},\fR^{-1}\rangle  \\
    &= \sum_{(f),(g)} (-1)^{[f_{(2)}][g_{(1)}]} \widetilde{\Delta}(f_{(1)} g_{(1)})\langle f_{(2)}\ot g_{(2)}, \fR^{-1}\rangle  \\
    &=\widetilde{\Delta}\circ \mu\circ \varpi(f\ot g).   
  \end{aligned}
  \] 
  This completes the proof of Lemma \ref{lemKeyFFT}.
\end{proof}

\begin{rmk}
	Part (2) of \lemref{lemKeyFFT} implies that  the $\CK$-linear map $\widetilde{\Delta}$ is \emph{not a superalgebra homomorphism}.
\end{rmk}

\begin{proof}[Proof of \lemref{lemKeylemlem}]
	For our purpose, we only prove the second identity, while the first one can be proved in the same way. Let $\fR^{-1}=\sum_h S(\alpha_h)\ot \beta_h$, then from \eqref{eqRmat} we have $\Delta(x)\fR^{-1}=\fR^{-1}\Delta^{\prime}(x)$, which implies that for all $x\in \Uqg$
	\[ 
	\begin{aligned}
		&\sum_{(x), h} (-1)^{[\alpha_h][x_{(2)}]}x_{(1)}S(\alpha_h)\ot x_{(2)}\beta_h\ot x_{(3)}\\
		=&\sum_{(x),h}(-1)^{[x_{(2)}]([\beta_h]+[x_{(1)}]) } S(\alpha_h)x_{(2)}\ot \beta_h x_{(1)}\ot x_{(3)}.
	\end{aligned}
     \] 
	Now applying $f\ot g\ot 1$ to both sides of the above equation, we obtain the left hand side
	\[ 
      \LHS=\sum_{(x),(f),(g),h} \sgn\  \langle f_{(1)}, x_{(1)}\rangle \langle f_{(2)}, S(\alpha_h)\rangle \langle	g_{(1)}, x_{(2)}\rangle \langle g_{(2)}, \beta_h \rangle\  x_{(3)},
	  \]
	with $\sgn=(-1)^{[\alpha_h][x_{(2)}]+[f_{(2)}][x_{(1)}]+[g_{(2)}][x_{(2)}]+([g_{(1)}]+[g_{(2)}])([x_{(1)}]+[\alpha_h])}$. The sum is over $(x),(f),(g)$, $h$ such that 
	$[x_{(1)}]=[f_{(1)}], [\alpha_h]=[f_{(2)}], [x_{(2)}]=[g_{(1)}],[\beta_h]=[g_{(2)}]$ and $[f_{(2)}]+[g_{(2)}]=\bar{0}$, yielding
	\[\LHS=\sum_{(x),(f),(g)} (-1)^{[g_{(1)}][f_{(2)}]}\langle f_{(1)}g_{(1)},x_{(1)}\rangle x_{(2)} \langle f_{(2)}\ot g_{(2)},\fR^{-1}\rangle.  \]
	The right hand side can be simplified in a  similar way, and hence we obtain the equality
	\[ 
	\begin{aligned}
	&\sum_{(f),(g),(x)} (-1)^{[f_{(2)}][g_{(1)}]} \langle f_{(1)}g_{(1)},x_{(1)}\rangle x_{(2)} \langle f_{(2)}\ot g_{(2)},\fR^{-1}\rangle \\
	=&\sum_{(f),(g),(x)} (-1)^{([g_{(1)}]+[g_{(2)}])[f_{(2)}]}\langle g_{(2)}f_{(2)},x_{(1)}\rangle x_{(2)} \langle f_{(1)}\ot g_{(1)},\fR^{-1}\rangle. 
	\end{aligned}
	 \]
	 We view $f_{(1)}g_{(1)}$ and $g_{(2)}f_{(2)}$ as linear functions on $\Uqg$. Applying the co-unit $\epsilon$ to both sides on $x_{(2)}$ and using $\sum_{(x)}x_{(1)}\epsilon(x_{(2)})=x$, we obtain our assertion in the lemma.
\end{proof}

Now we prove   \thmref{thmFFT} in the case that $m\geq \max\{k,r\}$ and $n\geq \max\{l,s\}$.

\begin{proof}[Proof of \thmref{thmFFT}]
	We use induction on the degree $N$ of $(\cM^{k|l}_{r|s})_N$.
	By \lemref{lemXim}, $\widetilde{\Delta}(\cM_{r|s}^{k|l})=\CX^{k|l}_{\;r|s}$ and we have seen $\widetilde{\Delta}(t_{ab})=X_{ab}$.  As $(\cM^{k|l}_{r|s})_N=\mu\circ \varpi ( (\cM^{k|l}_{r|s})_{N-1}\ot (\cM^{k|l}_{r|s})_1 )$ by part (1) of \lemref{lemKeyFFT}, the theorem follows directly from part (2) of \lemref{lemKeyFFT} and  inductive step. 
\end{proof}

\subsubsection{The case $m< \max\{k,r\}$ or $n< \max\{l,s\}$}
Let $u=\min\{k,r,m\}$ and $v=\min\{l,s,n\}$. Note that $\cM^{\;u|v}_{m|n}$ (resp. $\overbar{\cM}^{u|v}_{m|n}$) can be embedded into $\cM^{\;k|l}_{m|n}$ (resp. $\overbar{\cM}^{\;r|s}_{m|n}$) by using the truncation procedure as in \eqref{eqTrunM}. Explicitly,  we define the following  commutative subalgebras  respectively by 
\[
\begin{aligned}
  &\Upsilon_{u|v}^{k|l}:=\langle\, K_a\mid a\in  \{1,2,\dots, k-u\}\cup  \{k+v+1,\dots, k+l\} \, \rangle \subseteq \Uq(\gl_{k|l}),\\
  &\Upsilon_{u|v}^{r|s}:=\langle \, K_a\mid a\in \{1,2,\dots, r-u\} \cup \{r+v+1,\dots, r+s\} \, \rangle \subseteq \Uq(\gl_{r|s}).
\end{aligned}
\] 
Then we obtain  
\begin{equation}\label{eqtrunMuv}
   \cM^{\;u|v}_{m|n}=(\cM^{\;k|l}_{m|n})^{\CL(\Upsilon_{u|v}^{k|l})},\quad
   \overbar{\cM}^{\;u|v}_{m|n}=(\overbar{\cM}^{\;r|s}_{m|n})^{\widetilde{\CL}(\Upsilon_{u|v}^{r|s})}.
\end{equation}
Using  the fact that $\CL(\Uq(\gl_{k|l}))\ot \widetilde{\CL}(\Uq(\gl_{r|s}))$ graded-commutes with the $\Uqg$-action on $\CP^{k|l}_{r|s}$, we define
\begin{equation}\label{eqtrunX}
  \CX_{u|v}:=(\CX^{k|l}_{\;r|s})^{\CL(\Upsilon^{k|l}_{u|v})\ot \widetilde{\CL}(\Upsilon^{r|s}_{u|v})}=((\cM^{\;k|l}_{m|n}\ot_{\fR}\overbar{\cM}^{\;r|s}_{m|n} )^{\CL(\Upsilon^{k|l}_{u|v})\ot \widetilde{\CL}(\Upsilon^{r|s}_{u|v})})^{\Uqg}.
\end{equation}

\begin{lem}\label{lemcase2}
 There are following assertions:
 \begin{enumerate}
 	\item As a $\CL(\Uq(\gl_{u|v}))\ot \widetilde{\CL}(\Uq(\gl_{u|v}))$-module, $\CX_{u|v}$ has a multiplicity-free decomposition
 	\begin{equation*}
 	  \CX_{u|v}=\bigoplus_{\la\in \La_{u|v}} L_{\la}^{u|v}\ot L_{\la}^{u|v}.
 	\end{equation*}
 	\item $\CX_{u|v}$ is generated by $X_{ab}$ with $ a,b\in \bI_{u|v}$.
 	\item As a $\CL(\Uq(\gl_{k|l}))\ot \widetilde{\CL}(\Uq(\gl_{r|s}))$-module, $\CX^{k|l}_{\;r|s}$ is generated by $\CX_{u|v}$, i.e.,
 	\[\CX^{k|l}_{\;r|s}= \left(\CL(\Uq(\gl_{k|l}))\ot \widetilde{\CL}(\Uq(\gl_{r|s}))\right)  \CX_{u|v}.\]
 \end{enumerate}
\end{lem}

\begin{proof}
	By \eqref{eqtrunMuv} and \eqref{eqtrunX}, we have $\CX_{u|v}=(\cM^{\;u|v}_{m|n}\ot_{\fR} \overbar{\cM}^{\;u|v}_{m|n}  )^{\Uqg}$. Now $\cM^{\;u|v}_{m|n}$ and $ \overbar{\cM}^{\;u|v}_{m|n}$   can  be viewed respectively as subalgebras of $\cM_{m|n}$ and $\overbar{\cM}_{m|n}$  by truncation since $u\leq m$ and $v\leq n$. This reduces to the first case discussed in  \secref{secfirstcase}, leading to part (1) and part (2)  directly. For part (3), it is well known that the  
	$\CL(\Uq(\gl_{k|l}))\ot \widetilde{\CL}(\Uq(\gl_{r|s}))$ highest weight vectors of weight $\la$ are precisely the $\CL(\Uq(\gl_{u|v}))\ot \widetilde{\CL}(\Uq(\gl_{u|v}))$ highest weight vectors of the same weight, as there is a natural embedding $\Uq(\gl_{k|l})\ot \Uq(\gl_{r|s})\supseteq \Uq(\gl_{u|v})\ot \Uq(\gl_{u|v}).$ Therefore,  $\CX^{k|l}_{\;r|s}$ is generated by $\CX_{u|v}$ as $\CL(\Uq(\gl_{k|l}))\ot \widetilde{\CL}(\Uq(\gl_{r|s}))$-module. 
\end{proof}

Now we can finish the proof of \thmref{thmFFT}.
\begin{proof}[Proof of \thmref{thmFFT}]
 Note that $\CX_{u|v}$ is $\Z_{+}$-graded by setting $\deg X_{ab}=1$ for all $a,b\in \bI_{u|v}$. By \lemref{lemcase2}, we use induction on the degree $N$ homogeneous component of $\CX_{u|v}$ to show that 
 \[ \left(\CL(\Uq(\gl_{k|l}))\ot \widetilde{\CL}(\Uq(\gl_{r|s}))\right) (\CX_{u|v})_{N} \]
  is generated  by elements $X_{ab}$ with $a\in \bI_{k|l}, b\in \bI_{r|s}$. It is obvious for the case $N=1$. In general, let  $XY\in \CX_{u|v}$ with $\deg(XY)=N>1$. Then it is clear that for any homogeneous $x\in \Uq(\gl_{k|l})$ and $y\in \Uq(\gl_{r|s})$
  \[
   (\CL_{x}\ot\widetilde{\CL}_{y})(XY)=\sum_{(x),(y)}(-1)^{[x_{(2)}][y_{(1)}]+([x_{(2)}]+[y_{(2)}])[X]} (\CL_{x_{(1)}}\ot\widetilde{\CL}_{y_{(1)}})(X)  (\CL_{x_{(2)}}\ot\widetilde{\CL}_{y_{(2)}})(Y).
  \]
  Since both  the $\Z_{+}$-gradations of $X$ and $Y$ are less than $N$, our claim follows by induction.
\end{proof}

\subsection{Reformulation of FFT}\label{secrefFFT}
As in classical case, we shall reformulate the polynomial FFT for $\Uqg$ in terms of superalgebra homomorphism. 

In view of the  linear order $\succ$ defined in \eqref{eqorder}, we define monomials $X^{\bm}=\prod_{(a,b)}^{\succ}X_{ab}^{m_{ab}}, \bm\in \fM^{k|l}_{r|s}$, where the factors are arranged  decreasingly in the order $\succ$. Recall that $\CX^{k|l}_{\;r|s}$ is an invariant superalgebra of $\Uqg$-invariants.

\begin{lem}\label{lemXbasis}
	Assume that $m\geq \min\{k,r\}$ and $n\geq \min\{l,s\}$. The set of monomials 
	$ \{ X^{\bm}\mid \bm\in \fM^{k|l}_{r|s} \} $ 
	constitutes a $\CK$-basis for the invariant superalgebra $\CX^{k|l}_{\;r|s}$. 
\end{lem}
\begin{proof}
	Note that  $\La_{k|l}\cap \La_{r|s}\subseteq \La_{m|n}$ if and only if $m\geq \min\{k,r\}$  and $n\geq \min\{l,s\}$. Under our assumption, we obtain from \lemref{leminv} the multiplicity-free decomposition 
	\begin{equation}\label{eqdecX}
	\CX^{k|l}_{\;r|s}\cong\bigoplus_{\la\in \La_{k|l}\cap \La_{r|s}} L_{\la}^{k|l}\ot L_{\la}^{r|s}
	\end{equation}
	as $\CL(\Uq(\gl_{k|l}))\ot \widetilde{\CL}(\Uq(\gl_{r|s}))$-module. By  relations \eqref{eqinvrel}, the monomials $ X^{\bm}, \bm\in \fM^{k|l}_{r|s}$ span the invariant algebra $\CX^{k|l}_{\;r|s}$. Thus, we only need to prove the linear independence of these monomials.  
	
Let $(\CX^{k|l}_{\;r|s})_{N}$ be the homogeneous subspace of degree $N$ in $\CX^{k|l}_{\;r|s}$. Then we obtain from \eqref{eqdecX} 
\[
\dim_{\CK}(\CX^{k|l}_{\;r|s})_{N}=\!\!\!\!\!\!\!\!\sum_{\la\in \La_{k|l}\cap \La_{r|s},\; |\la|= N}\!\!\!\!\!\!\!\!\dim_{\CK}(L_{\la}^{k|l}\otimes L_{\la}^{r|s})=\!\!\!\!\!\!\!\!\sum_{\la\in \La_{k|l}\cap \La_{r|s},\; |\la|= N}\!\!\!\!\!\!\!\!
\dim_{\CK} L_{\la}^{k|l}\dim_{\CK} L_{\la}^{r|s}.
\]
Combing this and  \lemref{lemIden}, we have
\[  \dim_{\CK}(\CX^{k|l}_{\;r|s})_{N}=\# \{ t^{\bm}\mid \bm\in \fM^{k|l}_{r|s}, \ |\bm|=N \}=\# \{ X^{\bm}\mid \bm\in \fM^{k|l}_{r|s}, \ |\bm|=N \}. \]
This finishes the proof of  linear independence, since  $(\CX^{k|l}_{\;r|s})_{N}$ is spanned by the monomials  $X^{\bm}, |\bm|=N $. 
\end{proof}

We introduce the following auxiliary superalgebra, which will play a crucial role in our reformulation of FFT.
\begin{defn}\label{defntilM}
	 Let $k,l,r,s$ be non-negative integers.
   We denote by $\widetilde{\cM}^{k|l}_{r|s}$ the quadratic superalgebra over $\CK$ generated by $\widetilde{t}_{ab}$ with $a\in \bI_{k|l}, b\in \bI_{r|s}$, subject to the following relations:
   	\begin{equation}\label{eqnewalgrel}
   	\begin{aligned}
   	(\widetilde{t}_{ab})^{2}=&0,&\quad &[a]+[b]=\bar{1},\\
   	\widetilde{t}_{ac}\widetilde{t}_{bc}=&(-1)^{([a]+[c])([b]+[c])}q_{c}\widetilde{t}_{bc}\widetilde{t}_{ac},&\quad &a>b,\\
   	\widetilde{t}_{ab}\widetilde{t}_{ac}=&(-1)^{([a]+[b])([a]+[c])}q_{a}^{-1}\widetilde{t}_{ac}\widetilde{t}_{ab},&\quad&b>c,\\
   	\widetilde{t}_{ac}\widetilde{t}_{bd}=&(-1)^{([a]+[c])([b]+[d])}\widetilde{t}_{bd}\widetilde{t}_{ac},&\quad  &a>b, c>d,\\
   	\widetilde{t}_{ac}\widetilde{t}_{bd}=&(-1)^{([a]+[c])([b]+[d])}\widetilde{t}_{bd}\widetilde{t}_{ac}\\
   	&   +  (-1)^{[a]([b]+[d])+[b][d]}(q-q^{-1})\widetilde{t}_{bc}\widetilde{t}_{ad},&\quad  &a>b,c<d.
   	\end{aligned}	
   	\end{equation}
   	The $\Z_{2}$-grading is given by $[\widetilde{t}_{ab}]=[a]+[b]$. We shall write $\widetilde{\cM}_{k|l}:=\widetilde{\cM}^{k|l}_{k|l}$ for convenience.
\end{defn}

\begin{lem}\label{lemsuralgmor}
  Let  $k,l,r,s$ be fixed non-negative integers.  For any $m,n\in \Z_{+}$, the $\CK$-linear map from $\widetilde{\cM}^{k|l}_{r|s}$ to the invariant superalgebra $\CX^{k|l}_{\;r|s}$ of $\Uqg$-invariants 
   \begin{equation}\label{eqalghomo}
   \Psi^{k|l}_{r|s}: \widetilde{\cM}^{k|l}_{r|s}\longrightarrow \CX^{k|l}_{\;r|s}, \quad \widetilde{t}_{ab} \longmapsto X_{ab},\quad \forall a\in\bI_{k|l},b\in \bI_{r|s}.
   \end{equation}
   is a surjective superalgebra homomorphism.
\end{lem}
\begin{proof}
	It is clear from \eqref{eqinvrel} and  \eqref{eqnewalgrel} that the map is a well defined superalgebra homomorphism. The surjectivity is immediate by \thmref{thmFFT}.
\end{proof}

\begin{lem}\label{lemisoX}
  Suppose that $m\geq \min\{k,r\}$ and $n\geq \min\{l,s\}$. Then	$\Psi^{k|l}_{r|s}$ is a superalgebra isomorphism.
\end{lem}
\begin{proof}
	By \lemref{lemsuralgmor}, it suffices to prove $\Ker\, \Psi^{k|l}_{r|s}=0$, which is equivalent to show that elements in  $\CX^{k|l}_{\;r|s}$ have no nontrivial relations except \eqref{eqinvrel} if $m\geq \min\{k,r\}$ and $n\geq \min\{l,s\}$.

	Assume that $\CX^{k|l}_{\;r|s}$ has a  nontrivial relation except \eqref{eqinvrel}. We denote it  by 
	\[ f(X_{a_1b_1},\dots, X_{a_pb_p})=0, \quad p\geq 1, \]	
   where $f(X_{a_1b_1},\dots, X_{a_pb_p})$ is a non-zero polynomial in $\CX^{k|l}_{\;r|s}$. Using \lemref{lemXbasis}, we obtain the expression
  $ f(X_{a_1b_1},\dots, X_{a_pb_p})=\sum_{\bm} c_{\bm}X^{\bm}=0$
   for finitely many $c_{\bm}\in \CK$ with $\bm\in\fM^{k|l}_{r|s}$.  
  This forces all $c_{\bm}=0$ since $X^{\bm}$'s are $\CK$-linearly independent, and hence $f(X_{a_1b_1},\dots, X_{a_pb_p})$ is a zero polynomial, leading to a contradiction. This completes our proof.
\end{proof}

\lemref{lemisoX} can be viewed as another  formulation of $\widetilde{\cM}_{r|s}^{k|l}$, that is, $\widetilde{\cM}_{r|s}^{k|l}\cong \CX^{k|l}_{\;r|s}$ as superalgebras for sufficient large $m$ and $n$. Therefore, $\widetilde{\cM}_{r|s}^{k|l}$ acquires  the $\Uq(\gl_{k|l}))\ot \Uq(\gl_{r|s})$-module structure from its counterpart on $\CX^{k|l}_{\;r|s}$  through the  isomorphism $\Psi^{k|l}_{r|s}$ . 

%This immediately gives the following Howe duality for $\widetilde{\cM}_{r|s}^{k|l}$.

\begin{prop}\label{propHowedec}There are following properties of $\widetilde{\cM}_{r|s}^{k|l}$:
\begin{enumerate}
\item  (Howe duality) The superalgebra $\widetilde{\cM}_{r|s}^{k|l}$ admits the   multiplicity-free decomposition 
   \begin{equation*}\label{eqHowedec}
   \widetilde{\cM}_{r|s}^{k|l}\cong \bigoplus_{\la\in \La_{k|l}\cap \La_{r|s}} L^{k|l}_{\la}\ot L^{r|s}_{\la}
   \end{equation*}
   as $\Uq(\gl_{k|l}))\ot \Uq(\gl_{r|s})$-module;
\item (PBW basis) The set 
   $ \CT:=\{ \widetilde{t}^{\,\bm}=\prod_{(a,b)}^{\succ}\widetilde{t}_{ab}^{\,m_{ab}}\mid \bm\in \fM^{k|l}_{r|s} \} $ 
   constitutes a $\CK$-basis for $\widetilde{\cM}_{r|s}^{k|l}$. 
\end{enumerate}
\end{prop}
\begin{proof}
 By \lemref{lemisoX}, $\widetilde{\cM}_{r|s}^{k|l}\cong \CX^{k|l}_{\;r|s}$ as superalgebras for $m\geq \min\{k,r\}$ and $n\geq \min\{l,s\}$. Then part (1)  follows from \eqref{eqdecX} and part (2) from \lemref{lemXbasis} and  \lemref{lemisoX}.
\end{proof}

\begin{rmk}\label{rmk:BPW-multipl}
For any $\bm=(m_{ab}),\bn=(n_{ab})\in\fM^{k|l}_{r|s}$, we set $\bm+\bn=(m_{ab}+n_{ab})$.  By relations \eqref{eqnewalgrel},
\[ \widetilde{t}^{\,\bm}\widetilde{t}^{\,{\bf n}}=(-1)^{\de_1} q^{\de_2} 
\,\widetilde{t}^{\,\bm+\bn}+ (q-q^{-1})X,  \]
where $X\in \C[q, q^{-1}]\CT$, and  
$\de_1,\de_2\in \Z$ (which depend on $\bm$ and ${\bf n}$). Here $\widetilde{t}^{\,\bm+\bn}=0 $ if  $m_{ab}+n_{ab}=2$ for some $[a]+[b]=\bar{1}$. 
\end{rmk}

We reformulate the FFT of invariant theory for $\Uqg$ as follows.
\begin{thm}\label{FFTref}
	 {\rm (FFT  for $\Uqg$)}
	The superalgebra homomorphism $\Psi^{k|l}_{r|s}: \widetilde{\cM}^{k|l}_{r|s} \rightarrow \CX^{k|l}_{\;r|s}$ is surjective. Moreover, $\Psi^{k|l}_{r|s}$ is a $\Uq(\gl_{k|l})\ot \Uq(\gl_{s|r})$-module homomorphism.
\end{thm}

For notational convenience, we introduce the shorthand $\la=(1^{m_1}2^{m_2}\cdots)$ for the partition $\la$ with $m_1$ copies of 1, $m_2$ copies of 2 and so on. For any two partitions $\la$ and $\mu$,  $\la\subseteq\mu$  means that the  Young diagram of $\la$ can be embedded into that of $\mu$. 

\begin{coro}\label{coroXdec}
   As a $\Uq(\gl_{k|l})\ot \Uq(\gl_{r|s})$-module,  $\Ker\, \Psi^{k|l}_{r|s}$ admits the multiplicity-free decomposition
 	\[ \Ker\, \Psi^{k|l}_{r|s}\cong \bigoplus_{ \la_c\subseteq
 		\la\in \La_{k|l}\cap \La_{r|s}} L_{\la}^{k|l}\ot L_{\la}^{r|s}, \quad \text{with $\la_c=((n+1)^{m+1})$}.\]
 \end{coro}
\begin{proof}
   This is a consequence of \thmref{FFTref}, \propref{propHowedec} and \lemref{leminv}.
 \end{proof}

We close this section by giving a new formulation of $\widetilde{\cM}^{k|l}_{r|s}$ via braiding operator on $V^{k|l}\ot (V^{r|s})^{\ast}$, which shows that $\widetilde{\cM}^{k|l}_{r|s}$ is in fact not artificial.
\begin{prop}\label{propBraidedSym3}
	As a superalgebra, $\widetilde{\cM}^{k|l}_{r|s} \cong S_{q}(V^{k|l}\ot (V^{r|s})^{\ast} )$. Thus we have the following multiplicity-free decomposition
	\[ S_{q}(V^{k|l}\ot (V^{r|s})^{\ast} )\cong \bigoplus_{\la\in \La_{k|l}\cap \La_{r|s}} L_{\la}^{k|l}\ot L_{\la}^{r|s} \]
	as $\Uq(\gl_{k|l})\ot\Uq(\gl_{r|s})$-module. In particular, $V^{k|l}\ot (V^{r|s})^{\ast}$ is a flat $\Uq(\gl_{k|l})\ot\Uq(\gl_{r|s})$-module (in the sense of Section \ref{secbra}).
\end{prop}
\begin{proof}
	The proof is similar to that of \propref{propBraidedSym}. 
%	Recall  that the action of  $\widecheck{R}_{(V^{r|s})^{\ast},(V^{r|s})^{\ast}}$  on $(V^{r|s})^{\ast}\ot (V^{r|s})^{\ast}$ is given in \eqref{eqRmatact}. 
	Let $\widetilde{U}=V^{k|l}\ot (V^{r|s})^{\ast}$ and $P_{23}$ be graded the permutation of two middle factors in $\widetilde{U}\ot\widetilde{U}$. Then we have
	\[ \La_q^{2}(\widetilde{U})=P_{23}\bigg( \big(S^{2}_{q}(V^{k|l})\ot \La^{2}_{q}((V^{r|s})^{\ast}) \big) \bigoplus \big(\La^{2}_{q}(V^{k|l})\ot S^{2}_{q}((V^{r|s})^{\ast})\big)   \bigg).  \]
    Using bases given in \eqref{eqbaS}, \eqref{eqbaLa} and \propref{propBraidedSym2}, we obtain the quadratic relations for $S_q(\widetilde{U})=T(\widetilde{U})/\La_q^{2}(\widetilde{U})$ as follows:
    	\begin{equation*}\label{eqRelS3}
    	\begin{aligned}
    	(x_{ia})^2=&0, &\quad &[i]+[a]=\bar{1},\\
    	x_{ja}x_{ia}=&(-1)^{([i]+[a])([j]+[a])}q_ax_{ia}x_{ja},&\quad &j>i,\\
    	x_{ib}x_{ia}=&(-1)^{([i]+[a])([i]+[b])}q_i^{-1}x_{ia}x_{ib},&\quad &b>a, \\
    	x_{jb}x_{ia}=&(-1)^{([i]+[a])([j]+[b])}x_{ia}x_{jb},&\quad &j>i,\,b>a,\\
    	x_{ja}x_{ib}=&(-1)^{([i]+[b])([j]+[a])}x_{ib}x_{ja}\\
    	&+(-1)^{[i]([j]+[a])+[j][a]}(q-q^{-1})x_{jb}x_{ia},&\quad &j>i,\, a<b. 
    	\end{aligned}
    	\end{equation*}	
 Here $x_{ia}=v_i\ot v_{a}^{\ast}$. The isomorphism between  $\widetilde{\cM}^{k|l}_{r|s}$ and $S_q(\widetilde{U})$ is specified by $\widetilde{t}_{ia}\mapsto (-1)^{[i][a]} x_{ia}$. The multiplicity-free decomposition is also clear from this isomorphism and \propref{propHowedec}.
\end{proof}

Now the superalgebra homomorphisms given  in  \propref{propBraidedSym}, \propref{propBraidedSym2}, \propref{propBraidedSym3} and \thmref{FFTref}   fit into the following commutative diagram, and hence we obtain the surjective superalgebra homomorphism in the bottom row, which reduces  in the limit $q\to 1$ to  the classical case. 

\begin{equation*}\label{eqcommdiag1}
\xymatrixcolsep{5pc}
\xymatrix{
	\widetilde{\cM}^{k|l}_{r|s} \ar[d]^-{\cong} \ar@{->>}[r]^-{\Psi^{k|l}_{\,r|s}} & \CX^{k|l}_{r|s}:= (\cM^{\;k|l}_{m|n}\ot_{\fR} \overbar{\cM}^{\;r|s}_{m|n})^{\Uqg} \ar[d]^-{\cong}   \\
S_q(V^{k|l}\ot (V^{r|s})^{\ast}) \ar@{->>}[r] & \bigg(S_q(V^{k|l}\ot V^{m|n} )\ot_{\fR} S_q((V^{r|s})^{\ast}\ot (V^{m|n})^{\ast})\bigg)^{\Uqg}
}
\end{equation*}
%
%\[ S_q(V^{k|l}\ot (V^{r|s})^{\ast})\longrightarrow \bigg(S_q(V^{k|l}\ot V^{m|n} )\ot_{\fR} S_q((V^{r|s})^{\ast}\ot (V^{m|n})^{\ast})\bigg)^{\Uqg}  \]

\section{The SFT of invariant theory for $\Uqg$}\label{secSFT}

In this section, we shall describe the kernel of the superalgebra epimorphism $ \Psi^{k|l}_{r|s}: \widetilde{\cM}^{k|l}_{r|s}\rightarrow \CX^{k|l}_{\;r|s}$ as a two-sided ideal of $\widetilde{\cM}^{k|l}_{r|s}$. The images of nonzero elements in kernel will give rise to new relations among invariants in $\CX^{k|l}_{\;r|s}$, apart from the quadratic relations \eqref{eqinvrel}. 

The main idea is to identify $\widetilde{\cM}^{k|l}_{r|s}$ as a subalgebra  of $\widetilde{\cM}_{K|L}$ with $K=\max\{k,r \}$ and $L=\max \{l,s\}$. This way we obtain $\Ker\, \Psi^{k|l}_{r|s}$ by restricting $\Ker\, \Psi_{K|L}$ to $\widetilde{\cM}^{k|l}_{r|s}$; see commutative diagram \eqref{eqcommdiag} below.
We start by developing some general results on the algebraic structure of $\cM_{K|L}$ (resp. $\cM^{k|l}_{r|s}$), taking advantage of some nice properties of matrix elements. Then we translate these results to $\widetilde{\cM}_{K|L}$ (resp. $\widetilde{\cM}^{k|l}_{r|s}$). 

\subsection{Algebraic structure of $\cM_{K|L}$}\label{secFunalg}
We shall recall a well-known fact.  Suppose that $K,L\in \Z_{+}$ and $V$ is a left $\Uq(\gl_{K|L})$-module, then it has a canonical structure of a right $ \CK[\GL_q(K|L)]$-comodule with  structure map
\[ \de:\ V\rightarrow V\ot \CK[\GL_q(K|L)]. \] 
(Recall that $\CK[\GL_q(K|L)]=\cM_{K|L}\overbar{\cM}_{K|L}$, which is isomorphic to the coordinate superalgebra of $\Uq(\fgl_{K|L})$ as mentioned in \rmkref{rmkcoord}.)
Using Sweedler's notation, we write $ \delta(v)=\sum_{(v)}v_{(0)}\ot v_{(1)}$. Then we obtain
\[ \sum_{(v)}v_{(0)}\langle v_{(1)},x\rangle=(-1)^{[x][v]}xv,\quad \forall x\in \Uq(\gl_{K|L}), \,v\in V.  \]
Therefore, the comodule structure induces the original module structure in the canonical way and vice versa.

 Let $\la\in \La_{K|L}$ be a $(K,L)$-hook partition  and $(L_{\la}, \pi^{\la})$ be an irreducible  representation of $\Uq(\fgl_{K|L})$ with the highest weight $\la^{\natural}$. Define the elements $t_{ab}^{\la} \in \Uq(\fgl_{K|L})^{\circ}$ by
\[ \langle t_{ab}^{\la},x \rangle =\pi^{\la} ( x )_{ab}, \quad \forall x\in\Uq(\gl_{K|L}),\  a,b=1,2,\dots \dim_{\CK}L_{\la}.  \] 
We write  $T_{\la}$ for the subspace spanned by these elements $t_{ab}^{\la}$.  Then it follows form \propref{proptendec}  that $T_{\la}\subset \cM_{K|L}$,  and $L_{\la}$ naturally affords the right and left $\CK[\GL_q(K|L)]$-comodule structure, which are given respectively by  
\[
  \begin{aligned}
  &\de^R_{\la}: L_{\la} \rightarrow L_{\la}\ot T_{\la},\quad &v_b^{\la}\mapsto \sum_{a=1}^{\dim_{\CK} L_{\la}} v_a^{\la}\ot t_{ab}^{\la}, \\
 &\de^L_{\la}: L_{\la} \rightarrow T_{\la}\ot L_{\la},\quad &v_a^{\la}\mapsto \sum_{b=1}^{\dim_{\CK} L_{\la}} t_{ab}^{\la}\ot v_b^{\la},
  \end{aligned}
   \]
where $v_a^{\la}\ (a=1,2,\dots,\dim_{\CK}L_{\la})$ is a basis for $L_{\la}$. Note that $T_{\la}$ is independent of the choices of basis for $L_{\la}$. 

\begin{lem}\label{lemcomiso}
	As a two-sided $\CK[\GL_q(K|L)]$-comodule, $T_{\la}\cong L_{\la}\ot L_{\la}$ is irreducible and hence
	\[\cM_{K|L}= \bigoplus_{\la\in \La_{K|L}} T_{\la}.  \]
\end{lem}
\begin{proof}
	It is an immediate consequence of \propref{proptendec} and \thmref{thmPeterWeyl}.
\end{proof}

For any two left irreducible $\Uq(\gl_{K|L})$-modules $L_{\la}$ and $L_{\mu}$, we may form the tensor product $L_{\la}\ot L_{\mu}$, which encodes the right  $\CK[\GL_q(K|L)]$-comodule structure with
\[\delta^R_{\la\ot \mu}: L_{\la}\ot L_{\mu} \rightarrow  L_{\la}\ot L_{\mu} T_{\la}T_{\mu},\quad v_{b}^{\la}\ot v_{d}^{\mu}\mapsto \sum_{a,c} v_{a}^{\la}\ot v_{c}^{\mu}(-1)^{[v_c^{\mu}][t_{ab}^{\la}]}t_{ab}^{\la}t_{cd}^{\mu}. \]
Therefore, the product $T_{\la}T_{\mu}$ is nothing but the subspace of $\cM_{K|L}$ spanned by the matrix elements of the left $\Uq(\gl_{K|L})$-module $L_{\la}\ot L_{\mu}$. 

\begin{prop}\label{propProdDec}
	As a two-sided $\CK[\GL_q(K|L)]$-comodule, $T_{\la}T_{\mu}$ admits the following multiplicity-free decomposition
	\begin{equation}\label{eqProdDec}
	T_{\la}T_{\mu}=\bigoplus_{\gamma\in N_{\la,\mu}}T_{\gamma},
	\end{equation}
	where $N_{\la,\mu}=\{\gamma \mid N_{\la,\mu}^{\gamma}> 0\}$ such that $N_{\la,\mu}^{\gamma}$ is the  Littlewood-Richardson coefficient appearing in the decomposition $L_{\la}\ot L_{\mu}\cong\bigoplus_{\gamma \in \La_{K|L}} N_{\la,\mu}^{\gamma} L_{\gamma}$ as left $\Uq(\fgl_{K|L})$-module.   
\end{prop}
\begin{proof}
	This is immediate from the tensor product decomposition of $L_{\la}\ot L_{\mu}$, since each  $T_{\gamma}$ is the irreducible two-sided $\cM_{K|L}$-comodule spanned by the matrix coefficients of $L_{\gamma}$,  and $T_{\gamma_1}\bigcap T_{\gamma_2}=\varnothing$ for different $\gamma_1,\gamma_2\in N_{\la,\mu}$ and $T_{\gamma_1}=T_{\gamma_2}$ if $L_{\gamma_1}\cong L_{\gamma_2}$.
\end{proof}

\begin{rmk}
	The Littlewood-Richardson coefficient $N_{\la,\mu}^{\gamma}$ can be obtained by using super duality due to \cite{CWZ}.
\end{rmk}

\begin{coro}\label{coroTprod}
	Let $\la$ be a $(K,L)$-hook partition. Then we have
	\[ T_{\la}T_{(1)} =\bigoplus_{\gamma\in N_{\la,(1)}} T_{\gamma}, \]
	where  $N_{\la,(1)}$ is the set of all $(K,L)$-hook Young diagrams obtained by adding one box to $\la$.	
\end{coro}

We  denote by $\langle T_{\la}\rangle_{K|L}$ the two-side ideal in $\cM_{K|L}$ generated by $T_{\la}$. For any two-sided ideal $I\subseteq \cM_{K|L}$, we call $I$ the \emph{$\CK[\GL_q(K|L)]$-comodule ideal} if it affords two-sided $\CK[\GL_q(K|L)]$-comodule structure. Then $\langle T_{\la}\rangle_{K|L}$ is the minimal two-sided $\CK[\GL_q(K|L)]$-comodule ideal containing $T_{\la}$, and every two-sided $\CK[\GL_q(K|L)]$-comodule ideal is the sum of ideals of the form $\langle T_{\la}\rangle_{K|L}$.

The following theorem is a quantum super analogue of \cite[Theorem 4.1]{DEP}.

\begin{thm}\label{thmidealdec}
  Let $\la$ be a $(K,L)$-hook partition.  Then $\langle T_{\la}\rangle_{K|L}$ admits the multiplicity-free decomposition
  \begin{equation}\label{eqTideadec}
   \langle T_{\la}\rangle_{K|L}=\bigoplus_{ \la \subseteq \gamma\in  \La_{K|L}} T_{\gamma}
  \end{equation}
  as two-sided $\CK[\GL_q(K|L)]$-comodule.	
\end{thm}
\begin{proof}
	By \corref{coroTprod}, we have $\langle T_{\la}\rangle_{K|L}\subseteq\bigoplus_{ \la \subseteq \gamma\in  \La_{K|L}} T_{\gamma}$. Now assume that $\La_{K|L}\ni \gamma\supseteq \la $ with $|\gamma|=|\la|+1$, then it follows from \corref{coroTprod} again that $T_{\gamma}\subset T_{\la}T_{(1)}$ and hence $T_{\gamma}\subset \langle T_{\la}\rangle_{K|L}$. Using induction on the size of $\gamma$, we obtain $\bigoplus_{ \la \subseteq \gamma\in  \La_{K|L}} T_{\gamma} \subseteq \langle T_{\la}\rangle_{K|L}$.
\end{proof}

We write $I_{\la}=\langle T_{\la}\rangle_{K|L}$ for short. It is an immediate consequence of  \thmref{thmidealdec} that 
\begin{coro} \label{coroidealdec}
	We have two  facts:
	\begin{enumerate}
		\item Let $\La\subseteq \La_{K|L}$ be a subset  of $(K,L)$-hook partitions. The two-sided $\CK[\GL_q(K|L)]$-comodule $\sum_{\gamma\in \La} T_{\gamma}$  is an ideal in  $\cM_{K|L}$ if and  only if $\gamma_1\in \La$ and $\gamma_2\supseteq \gamma_1$ imply $\gamma_2\in \La$.
		\item $I_{\la}\subseteq I_{\mu}$ if and only if $\la\supseteq \mu$.
		
	\end{enumerate}
\end{coro}

Following \cite{DEP}, we will say that a set $\La\subseteq \La_{K|L}$ is a  diagrammatic ideal if $\gamma_1\in \La$ and $\gamma_2\supseteq \gamma_1$ imply $\gamma_2\in \La$. Then \thmref{thmidealdec} and \corref{coroidealdec} yield

\begin{prop}
  There is a bijective correspondence between diagrammatic ideals and two-sided  $\CK[\GL_q(K|L)]$-comodule ideas of $\cM_{K|L}$, which is given by 
  \[  
  \begin{aligned}
  &\La \rightarrow I(\La)=\sum_{\gamma\in \La} T_{\gamma}, \quad  \text{for all diagrammatic ideals $\La$},\\
  &I\rightarrow \{\gamma\mid T_{\gamma}\subseteq I \},\quad  \text{for all  two-sided  $\CK[\GL_q(K|L)]$-comodule ideas $I$}.  
  \end{aligned}  
  \] 
\end{prop}

\subsection{Algebraic structure of $\cM^{k|l}_{r|s}$.} \label{secalgstr2}
Suppose that $k,l,r,s$ are non-negative integers.  Let $K=\max\{k,r\}, L=\max\{l,s\}$ and write $\cM_{K|L}:=\cM^{K|L}_{K|L}$ for the  bi-superalgebra generated by matrix elements $t_{ab}$ with $a\in \bI_{K|L},b\in\bI_{K|L}$ subject to the relations \eqref{eqRel1}. Write
\begin{equation}\label{eqhatI}
\hat{\bI}_{i|j}=\{a| K-i+1\leq a\leq K+j \}\subseteq \bI_{K|L}
\end{equation}
for any $i\leq K$ and $j\leq L$.
We identify $\cM^{k|l}_{r|s}$ with the subalgebra of $\cM_{K|L}$ generated by elements $t_{ab}$ with $a\in \hat{\bI}_{k|l} ,b\in \hat{\bI}_{r|s}$, and denote this embedding by $\iota^{\,k|l}_{\,r|s}$. Conversely, there is a $\CK$-algebra retraction $\pi^{k|l}_{\,r|s}: \cM_{K|L}\rightarrow  \cM^{k|l}_{r|s}$ such that $\pi^{k|l}_{\,r|s}(t_{ab})=t_{ab}$ for all $a\in \hat{\bI}_{k|l} ,b\in \hat{\bI}_{r|s}$,  and $\pi^{k|l}_{\,r|s}(t_{ab})=0$ otherwise. Clearly, $\pi^{k|l}_{\,r|s}\, \iota^{\,k|l}_{\,r|s}=\id_{\cM^{k|l}_{r|s}}$.

We have $\pi^{k|l}_{\,r|s}(T_{\la})\subset \cM^{k|l}_{r|s}$ and  denote the image by  $^{\pi}T_{\la}$ for simplicity, i.e., $^{\pi}T_{\la}=T_{\la}\cap \cM^{k|l}_{r|s}$. Note that $^{\pi}T_{\la}\neq 0$, unless it is contained in some ideal generated by some $t_{ab}$ with $a\in \bI_{K|L}\backslash \hat{\bI}_{k|l}$ or $b\in \bI_{K|L}\backslash \hat{\bI}_{r|s}$. This implies that  $\pi^{k|l}_{\,r|s}(\cM_{K|L})$  amounts to the truncation applied to $ \cM_{K|L}$ as in \thmref{thmHowe}. Therefore, 
\begin{equation}\label{eqTnozero1}
^{\pi}T_{\la}\neq 0 \ \text{if and only if} \ \la\in \La_{k|l}\cap\La_{r|s}.
\end{equation}
It follows that as $ \CL(\Uq(\gl_{k|l})) \ot \cR(\Uq(\gl_{r|s}))$-module
$ ^{\pi}T_{\la}\cong  L_{\la}^{k|l}\ot L_{\la}^{r|s}$ for any $\la\in \La_{k|l}\cap\La_{r|s},$
and in this case  $T_{\la}=\CL(\Uq(\gl_{K|L})) \ot \cR(\Uq(\gl_{K|L})) \, ^{\pi}T_{\la}$.

\begin{prop}\label{lemmult}
	Let $\la,\mu\in \La_{k|l}\cap\La_{r|s}\subseteq \La_{K|L}$. We have the multiplication formula in $\cM^{k|l}_{r|s}$
	\[ ^{\pi}T_{\la}\, ^{\pi}T_{\mu}=\bigoplus_{\gamma \in N_{\la,\mu}\cap \La_{k|l}\cap\La_{r|s}}\!\!\!\!\!   ^{\pi}T_{\gamma},   \]
	where  $N_{\la,\mu}$ is defined as in \propref{propProdDec}. 
\end{prop}
\begin{proof}
	Since $\pi^{k|l}_{\,r|s}$ is an algebra homomorphism, we apply it to both sides of \eqref{eqProdDec} and hence
	\[^{\pi}T_{\la}\, ^{\pi}T_{\mu}=\bigoplus_{\gamma\in N_{\la,\mu}}\,  ^{\pi}T_{\gamma}=\bigoplus_{\gamma \in N_{\la,\mu}\cap \La_{k|l}\cap\La_{r|s}}\!\!\!\!\!   ^{\pi}T_{\gamma},   \]
	where the last equation is the result of \eqref{eqTnozero1}.
\end{proof}

We denote by $ \langle ^{\pi}T_{\la}\rangle^{k|l}_{r|s}$ the two-sided ideal in $\cM^{k|l}_{r|s}$ generated by $^{\pi}T_{\la}$.
\begin{prop}\label{propideadec}
	Let $\la\in \La_{k|l}\cap\La_{r|s}\subseteq \La_{K|L}$.  Then  $\langle ^{\pi}T_{\la}\rangle^{k|l}_{r|s}$  as $\CL(\Uq(\gl_{k|l}))\ot \cR(\Uq(\gl_{r|s}))$-module  admits the following multiplicity-free decomposition
	\[  \langle ^{\pi}T_{\la}\rangle^{k|l}_{r|s}=\bigoplus_{ \la \subseteq \gamma \in \La_{k|l}\cap \La_{r|s}}\!\!\!\! ^{\pi}T_{\gamma}. \] 
\end{prop}
\begin{proof}
	Applying $\pi^{k|l}_{r|s}$ to both sides of \eqref{eqTideadec}, we obtain
	\[ 	\langle ^{\pi}T_{\la}\rangle^{k|l}_{r|s} = \bigoplus_{\la_{c} \subseteq \la\in\La_{K|L}}\, ^{\pi}T_{\la}=\bigoplus_{ \la_c \subseteq \la \in \La_{k|l}\cap \La_{r|s}}\!\!\!\! ^{\pi}T_{\la}, \]
	where the last equation follows from \eqref{eqTnozero1}.
\end{proof}

\begin{rmk}\label{rmkTclassical}
	It is worth noting that \propref{propProdDec}, \thmref{thmidealdec} hold for $\cM_{K|L}|_{q=1}$,  so do \propref{lemmult} and \propref{propideadec} for $\cM^{k|l}_{r|s}$.  These are actually super analogues of \cite[Theorem 4.1]{DEP}, which was originally proved characteristic-freely by  straightening formula. 
\end{rmk}

\subsection{The second fundamental theorem}
Now we turn to the  homomorphism $ \Psi^{k|l}_{r|s}$  defined in \eqref{eqalghomo}. We shall characterise the kernel of $ \Psi^{k|l}_{r|s}$ as a two-sided ideal of $\widetilde{\cM}^{k|l}_{r|s}$.

%To simplify the situation, we need  some algebra homomorphisms. 

\subsubsection{Algebraic structure of $\widetilde{\cM}^{k|l}_{r|s}$}
For our purpose, we will embed $\widetilde{\cM}^{k|l}_{r|s}$ into a larger superalgebra $\widetilde{\cM}_{K|L}$ and then explore its algebraic properties, which are similar to those of $\cM^{k|l}_{r|s}$  in \secref{secalgstr2}.

Fix non-negative integers $k,l,r,s$ and let $K=\max\{k,r\}$, $L=\max\{l,s\}$. Let $\widetilde{\cM}_{K|L}$ be the superalgebra  generated by  $\widetilde{t}_{ab}$ with $a,b\in\bI_{K|L}$ subject to relations \eqref{eqnewalgrel}.
We can  define the following embedding and retraction, which are similar to $\iota^{\,k|l}_{\,r|s}$ and $\pi^{k|l}_{\,r|s}$ in \secref{secalgstr2},
\begin{equation}\label{eqemretra}
   \widetilde{\iota}^{\,k|l}_{\,r|s}: \widetilde{\cM}^{k|l}_{r|s}\rightarrow \widetilde{\cM}_{K|L},\quad       
   \widetilde{\pi}^{k|l}_{r|s}: \widetilde{\cM}_{K|L} \rightarrow \widetilde{\cM}^{k|l}_{r|s},
\end{equation} 
which satisfy $\widetilde{\pi}^{k|l}_{r|s}\, \widetilde{\iota}^{\,k|l}_{\,r|s}=\id_{\widetilde{\cM}^{k|l}_{r|s}}$. Notice that $\widetilde{\cM}_{K|L}$ encodes the   $\Uq(\gl_{K|L})\ot \Uq(\gl_{K|L})$-module structure given in \secref{secrefFFT}, which naturally induces the  $\Uq(\gl_{k|l}) \ot \Uq(\gl_{r|s})$-module structure on the subalgebra $\widetilde{\cM}^{k|l}_{r|s}$ of $\widetilde{\cM}_{K|L}$.
%and $\psi^{k|l}_{r|s}\,\tau^{k|l}_{r|s}=\id_{\CX^{k|l}_{\;r|s}}$.
% Recall that 
%\[ \CX^{k|l}_{\;r|s}=(\cM^{\;k|l}_{m|n}\ot_{\fR} \overbar{\cM}^{\;r|s}_{m|n})^{\Uqg}, \quad \CX_{K|L}:=(\cM^{K|L}_{m|n}\ot_{\fR} \overbar{\cM}^{K|L}_{m|n})^{\Uqg}. \]
%Identifying $\cM^{\;k|l}_{m|n}$ (resp. $\overbar{\cM}^{\;r|s}_{m|n}$)  as a subalgebra of $\cM^{K|L}_{m|n}$ (resp. $\overbar{\cM}^{K|L}_{m|n}$), we immediately obtain the  embedding $\iota^{k|l}_{r|s}: \CX^{k|l}_{\;r|s}\rightarrow \CX_{K|L}$  and the retraction $\pi^{k|l}_{r|s}: \CX_{K|L}\rightarrow \CX^{k|l}_{\;r|s}$.
%The algebra homomorphisms fit into the following commutative diagram:
%\begin{equation}\label{eqcommdiag}
%\begin{aligned}
%\xymatrixcolsep{5pc}
%\xymatrix{
%	\widetilde{\cM}^{k|l}_{r|s} \ar@{->>}[d]^{\Psi^{k|l}_{r|s}}  \ar@{^{(}->}[r]^{\widetilde{\iota}^{\,k|l}_{\,r|s}} & \widetilde{\cM}_{K|L} \ar@{->>}[d]^{\Psi_{K|L}}  \ar@{->>}[r]^{\widetilde{\pi}^{k|l}_{\,r|s}}  &  \widetilde{\cM}^{k|l}_{r|s}  \ar@{->>}[d]^{\Psi^{k|l}_{r|s}}  \\
%	\CX^{k|l}_{\;r|s} \ar@{^{(}->}[r]^{ \tau^{k|l}_{r|s}} & \CX_{K|L} \ar@{->>}[r]^{  \psi^{k|l}_{r|s} } & \CX^{k|l}_{\;r|s} 
%}
%\end{aligned}
%\end{equation}

Recall that $\widetilde{\cM}_{K|L}$ admits the multiplicity-free decomposition $\widetilde{\cM}_{K|L}\cong\bigoplus_{\la\in \La_{K|L}} L_{\la}^{K|L}\ot L_{\la}^{K|L}$ as $\Uq(\gl_{K|L})\ot \Uq(\gl_{K|L})$-module by setting $k=r=K$ and $l=s=L$ in \propref{propHowedec}. 

\begin{defn}
For each $\la\in \La_{K|L}$, we denote by  $\widetilde{T}_{\la}$ the subspace of  $\widetilde{\cM}_{K|L}$ which is isomorphic to $L_{\la}^{K|L}\ot L_{\la}^{K|L}$ as $\Uq(\gl_{K|L})\ot \Uq(\gl_{K|L})$-module. Thus,  $\widetilde{\cM}_{K|L}=\bigoplus_{\la\in \La_{K|L}} \widetilde{T}_{\la}$.	
\end{defn}
By  \eqref{eqemretra}, we have $\widetilde{\pi}^{k|l}_{\,r|s}(\widetilde{T}_{\la})\subset \widetilde{\cM}^{k|l}_{r|s}$ and  denote the image by  $^{\pi}\widetilde{T}_{\la}$ for simplicity. Here $^{\pi}\widetilde{T}_{\la}=\widetilde{T}_{\la}\cap \widetilde{\cM}^{k|l}_{r|s}$. It is readily verified that  $\widetilde{\pi}^{k|l}_{\,r|s}(\widetilde{\cM}_{K|L})$  amounts to the truncation applied to $\widetilde{\cM}_{K|L}$ as in \thmref{thmHowe}; therefore we obtain
\begin{equation}\label{eqTnozero2}
^{\pi}\widetilde{T}_{\la}\neq 0 \ \text{if and only if} \ \la\in \La_{k|l}\cap\La_{r|s}.
\end{equation}
In this case, $ ^{\pi}\widetilde{T}_{\la}\cong  L_{\la}^{k|l}\ot L_{\la}^{r|s} \cong\,   ^{\pi}T_{\la}$ as $ \Uq(\gl_{k|l}) \ot \Uq(\gl_{r|s})$-module.

Observe that  $\widetilde{\cM}^{k|l}_{r|s}$ and $\cM^{k|l}_{r|s}$ are isomorphic as $\Uq(\gl_{k|l}) \ot \Uq(\gl_{r|s})$-modules by \thmref{thmHowe}  and \propref{propHowedec}.
We expect that the subspaces $^{\pi} \widetilde{T}_{\la}$ have the same properties as in  \propref{lemmult} and \propref{propideadec}, though $\widetilde{\cM}^{k|l}_{r|s}$ and $\cM^{k|l}_{r|s}$ have different defining relations in general. This is given in the following lemma.

\begin{lem}\label{lemmultfor}
	Let $\la,\mu\in \La_{k|l}\cap\La_{r|s}\subseteq \La_{K|L}$. We have the multiplication formula in $\widetilde{\cM}^{k|l}_{r|s}$
	\[ ^{\pi}\widetilde{T}_{\la}\, ^{\pi}\widetilde{T}_{\mu}=\bigoplus_{\gamma \in N_{\la,\mu}\cap \La_{k|l}\cap\La_{r|s}}\!\!\!\!\!   ^{\pi}\widetilde{T}_{\gamma},   \]
	where  $N_{\la,\mu}$ is defined as in \propref{propProdDec}. 
\end{lem}
\begin{proof}
	It suffices to prove that  
	$ \widetilde{T}_{\la} \widetilde{T}_{\mu}= \bigoplus_{\gamma\in N_{\la,\mu}} \widetilde{T}_{\gamma} $ in $\widetilde{\cM}_{K|L}$, from which we can obtain the desired decomposition in $\widetilde{\cM}^{k|l}_{r|s}$ by applying $\widetilde{\pi}^{k|l}_{r|s}$ to both sides of equation.
	
    As a $\Uq(\gl_{K|L})\ot \Uq(\gl_{K|L})$-module, $\widetilde{T}_{\la} \, \widetilde{T}_{\mu}$ is a quotient of $L_{\la}^{K|L}\ot L_{\mu}^{K|L}\ot L_{\la}^{K|L}\ot L_{\mu}^{K|L}$, which admits the decomposition
    \[ L_{\la}^{K|L}\ot  L_{\mu}^{K|L}\ot L_{\la}^{K|L}\ot L_{\mu}^{K|L}\cong \bigoplus_{\nu, \gamma\in N_{\la,\mu} } N_{\la,\mu}^{\nu} N_{\la,\mu}^{\gamma} L_{\nu}^{K|L}\ot L_{\gamma}^{K|L}, \]
    where $N_{\la,\mu}^{\nu}, N_{\la,\mu}^{\gamma}>0$ are Littlewood-Richardson coefficients. 
	On the other hand, $\widetilde{T}_{\la}\, \widetilde{T}_{\mu}$ is a submodule of $\widetilde{\cM}_{K|L}\cong \bigoplus_{\la\in \La_{K|L}} L_{\la}^{K|L}\ot L_{\la}^{K|L}$. It follows that $\widetilde{T}_{\la} \, \widetilde{T}_{\mu}$ decomposes as
    $\bigoplus_{\gamma \in N_{\la,\mu}} (L_{\gamma}^{K|L}\ot L_{\gamma}^{K|L})^{\oplus m_{\gamma}}$, where $m_{\gamma}$ ($0\leq m_{\gamma} \leq 1$) denotes the multiplicity. Thus $\widetilde{T}_{\la} \, \widetilde{T}_{\mu}=\bigoplus_{\gamma \in N_{\la,\mu}} (\widetilde{T}_{\gamma})^{\oplus m_{\gamma}}$, 
and $\dim_{\CK}(\widetilde{T}_{\la} \, \widetilde{T}_{\mu})=\sum_{\gamma \in N_{\la,\mu}} m_{\gamma}\dim_{\CK} \widetilde{T}_{\gamma}$. It remains to show that  $m_{\gamma}=1$ for all $\gamma\in N_{\la,\mu}$.

Set $A=\C[q, q^{-1}]$. Let 
$\widetilde{\cM}_{K|L, A}$ be the superalgebra over $A$ generated by the  $\widetilde{t}_{ab}$'s with the relations \eqref{eqnewalgrel}. Similarly we let 
$\cM_{K|L, A}$ be the superalgebra over $A$  generated by the $t_{ab}$'s.
These superalgebras have PBW bases
(cf. \propref{propMbasis} and \propref{propHowedec}). Their 
specialisations at $q=1$ are isomorphic to the super polynomial algebra generated by $K^2+L^2$ even variables and $2K L$ Grassmannian variables.  These are 
isomorphisms of $\U(\gl_{K|L})\ot \U(\gl_{K|L})$-module superalgebras. 
Let $\widetilde{T}_{\la, A}=\widetilde{T}_{\la}\cap\widetilde{\cM}_{K|L, A}$
and $T_{\la, A}=T_{\la}\cap\cM_{K|L, A}$.
Denote by $\widetilde{T}_{\la}|_{q=1}$ the specialisation of $\widetilde{T}_{\la, A}$ at $q=1$,  and by $T_{\la}|_{q=1}$ that of $T_{\la, A}$. It then follows that  $\widetilde{T}_{\la}|_{q=1}\cong T_{\la}|_{q=1}$ as  $\U(\gl_{K|L})\ot \U(\gl_{K|L})$-modules.  

For any $\la$ and $\mu$, it is clear from \rmkref{rmk:BPW-multipl} that the specialisation of $\widetilde{T}_{\la, A}\widetilde{T}_{\mu, A}$ at $q=1$ is equal to $\widetilde{T}_{\la}|_{q=1} \, \widetilde{T}_{\mu}|_{q=1}$. 
By \propref{propProdDec} and \rmkref{rmkTclassical},  we have 
\[ \widetilde{T}_{\la}|_{q=1} \, \widetilde{T}_{\mu}|_{q=1}= \bigoplus_{\gamma\in N_{\la,\mu}} \widetilde{T}_{\gamma}|_{q=1}\cong \bigoplus_{\gamma \in N_{\la,\mu}} L_{\gamma}^{K|L}|_{q=1}\ot L_{\gamma}^{K|L}|_{q=1}. \]
This leads to 	
\[
\dim_{\CK}(\widetilde{T}_{\la} \, \widetilde{T}_{\mu})\ge 
\sum_{\gamma \in N_{\la,\mu}} (\dim_{\C} L^{K|L}_{\gamma}|_{q=1})^2
=\sum_{\gamma \in N_{\la,\mu}} (\dim_{\CK} L^{K|L}_{\gamma})^2=
\sum_{\gamma \in N_{\la,\mu}} \dim_{\CK} \widetilde{T}_{\gamma},
\]
forcing $m_\gamma=1$ for all $\gamma \in N_{\la,\mu}$. This completes the proof.
\end{proof}

Let $\langle ^{\pi}\widetilde{T}_{\la}\rangle^{k|l}_{r|s}$ be the two-sided ideal in $\widetilde{\cM}^{k|l}_{r|s}$ generated by $^{\pi}\widetilde{T}_{\la}$.  We immediately obtain 

\begin{prop}\label{propideadec2}
	Let $\la\in \La_{k|l}\cap\La_{r|s}\subseteq \La_{K|L}$.  Then  $\langle ^{\pi}\widetilde{T}_{\la}\rangle^{k|l}_{r|s}$  as $\Uq(\gl_{k|l})\ot \Uq(\gl_{r|s})$-module  admits the following multiplicity-free decomposition
	\[  \langle ^{\pi}\widetilde{T}_{\la}\rangle^{k|l}_{r|s}=\bigoplus_{ \la \subseteq \gamma \in \La_{k|l}\cap \La_{r|s}}\!\!\!\! ^{\pi} \widetilde{T}_{\gamma}. \] 
\end{prop}
\begin{proof}
  We only need to show that $\langle \widetilde{T}_{\la}\rangle_{K|L}=\bigoplus_{ \la \subseteq \gamma\in  \La_{K|L}} \widetilde{T}_{\gamma}$ in $\widetilde{\cM}_{K|L}$, since  our proposition follows by applying the algebra homomorphism $\widetilde{\pi}^{k|l}_{r|s}$ to both sides of the equation. 
  The decomposition of the two-sided ideal  $\langle \widetilde{T}_{\la}\rangle_{K|L}$ can be derived from \lemref{lemmultfor} in a similar the way as in  \thmref{thmidealdec}. 
\end{proof}

\subsubsection{Formulation of the SFT}
Let $\CX_{K|L}:=(\cM^{K|L}_{m|n}\ot_{\fR} \overbar{\cM}^{K|L}_{m|n})^{\Uqg}$ be the $\Uqg$-invariant subalgebra, which is generated by $X_{ab}$ with $a,b\in\bI_{K|L}$ obeying  relations \eqref{eqinvrel}.
Identifying $\CX^{k|l}_{\;r|s}$   as a subalgebra of $\CX_{K|L}$, we immediately obtain the  embedding $\tau^{k|l}_{\,r|s}: \CX^{k|l}_{\;r|s}\rightarrow \CX_{K|L}$  and the retraction $\psi^{k|l}_{\,r|s}: \CX_{K|L}\rightarrow \CX^{k|l}_{\;r|s}$ with $\psi^{k|l}_{\,r|s}\, \tau^{k|l}_{\,r|s}=\id_{\CX^{k|l}_{\;r|s}}$, which are similar to $\iota^{\,k|l}_{\,r|s}$ and  $\pi^{k|l}_{\,r|s}$ defined in \secref{secalgstr2}.
The algebra homomorphisms fit into the following commutative diagram.
\begin{equation}\label{eqcommdiag}
\begin{aligned}
\xymatrixcolsep{5pc}
\xymatrix{
	\widetilde{\cM}^{k|l}_{r|s} \ar@{->>}[d]^{\Psi^{k|l}_{r|s}}  \ar@{^{(}->}[r]^{\widetilde{\iota}^{\,k|l}_{\,r|s}} & \widetilde{\cM}_{K|L} \ar@{->>}[d]^{\Psi_{K|L}}  \ar@{->>}[r]^{\widetilde{\pi}^{k|l}_{\,r|s}}  &  \widetilde{\cM}^{k|l}_{r|s}  \ar@{->>}[d]^{\Psi^{k|l}_{r|s}}  \\
	\CX^{k|l}_{\;r|s} \ar@{^{(}->}[r]^{ \tau^{k|l}_{r|s}} & \CX_{K|L} \ar@{->>}[r]^{  \psi^{k|l}_{r|s} } & \CX^{k|l}_{\;r|s} 
}
\end{aligned}
\end{equation}

\begin{lem}\label{lemKerKL}
	 Let $\la_c=((n+1)^{m+1})$, then as $\Uq(\gl_{K|L})\ot\Uq(\gl_{K|L})$-module,
	 \[ \Ker\, \Psi_{K|L}= \bigoplus_{\la_{c} \subseteq \la\in\La_{K|L}} \widetilde{T}_{\la}=\langle \widetilde{T}_{\la_c}\rangle_{K|L}. \]
\end{lem}
\begin{proof}
  The is  a special case of  \corref{coroXdec} and \propref{propideadec2} with $k=r=K$ and $l=s=L$. 
\end{proof}

Now we obtain the  second fundamental theorem of invariant theory for $\Uqg$ as follows.
\begin{thm} \label{thmSFT} 	{\rm (SFT for $\Uqg$)}
	Let $\Psi^{k|l}_{r|s}:  \widetilde{\cM}^{k|l}_{r|s} \rightarrow \CX^{k|l}_{\;r|s}$ be the surjective superalgebra homomorphism.
	\begin{enumerate}
		\item 	As a two-sided ideal in $\widetilde{\cM}^{k|l}_{r|s}$,
		\begin{equation*}\label{eqKerdec}
		\Ker\, \Psi^{k|l}_{r|s}=\bigoplus_{ \la_c \subseteq \la \in \La_{k|l}\cap \La_{r|s}}\!\!\!\! ^{\pi}\widetilde{T}_{\la}= \langle ^{\pi}\widetilde{T}_{\la_c}\rangle^{k|l}_{r|s}, 
		\end{equation*}
		where $\langle^{\pi}\widetilde{T}_{\la_c}\rangle^{k|l}_{r|s}$ is the two-sided ideal in $\widetilde{\cM}^{k|l}_{r|s}$ generated by $^{\pi}\widetilde{T}_{\la_c}$ with $\la_c=((n+1)^{m+1})$.
		\item $\Psi^{k|l}_{r|s}$ is an isomorphism if and only if $m\geq \min\{k,r\}$  and $n\geq \min\{l,s\}$. 
		In this case, $\CX^{k|l}_{\;r|s}$ is a quadratic superalgebra generated by $X_{ab}, a\in \bI_{k|l},b\in \bI_{r|s}$ subject to relations  \eqref{eqinvrel}, and it  has a PBW basis  
		$\{ X^{\bm}=\prod_{(a,b)}^{\succ}X_{ab}^{m_{ab}} \mid \bm\in \fM^{k|l}_{r|s} \}.$
	\end{enumerate}
\end{thm}
\begin{proof}
 Let $K=\max\{k,r\}$ and $L=\max\{l,s\}$, we have
 \[
 \begin{aligned}
  \Ker\, \Psi^{k|l}_{r|s}&=\pi^{k|l}_{\,r|s}(\Ker\, \Psi_{K|L}) &\quad& \text{by commutative diagram \eqref{eqcommdiag},}  \\
  &= \bigoplus_{\la_{c} \subseteq \la\in\La_{K|L}}\, ^{\pi}\widetilde{T}_{\la} &\quad& \text{by \lemref{lemKerKL},}   \\
  &=\bigoplus_{ \la_c \subseteq \la \in \La_{k|l}\cap \La_{r|s}}\!\!\!\! ^{\pi}\widetilde{T}_{\la}  &\quad& \text{by \eqref{eqTnozero2},}\\
  &=\langle ^{\pi}\widetilde{T}_{\la_c}\rangle^{k|l}_{r|s} &\quad& \text{by \propref{propideadec2}.}
 \end{aligned}
 \]

For part (2), it is clear from part (1) that $ \Ker\, \Psi^{k|l}_{r|s}=0$ if and only if  $\La_{k|l}\cap \La_{r|s}\subseteq \La_{m|n}$ if and only if $m\geq \min\{k,r\}$  and $n\geq \min\{l,s\}$, whence the isomorphism follows. In this case, the PBW basis is given by \lemref{lemXbasis}.
\end{proof}

\begin{rmk} 
It will be very interesting to find the elements which generate  $^{\pi}\widetilde{T}_{\la_c}$, thus to make the SFT more explicit. We will do this for the cases  $\Uq(\gl_m)$ and $\U(\gl_{m|n})$ in \secref{secapplication}.
\end{rmk}

\section{Examples}\label{sect:appl}
\label{secapplication}
In this section, we shall  elucidate how our main results  can be applied to  derive  ``polynomial'' versions of invariant theory for $\U(\gl_{m|n})$ and  $\Uq(\gl_m)$. Basically, the invariant theory of $\gl_{m|n}$ \cite{S1,S2} can be obtained in our language of matrix elements by specialising  $q$ to $1$. Also, the FFT of invariant theory for $\Uq(\gl_m)$ recovers \cite[Theorem 6.10]{LZZ}, while the SFT for $\Uq(\gl_m)$ appears to be new.

\subsection{Invariant theory for $\U(\gl_{m|n})$}
Recall that $\U(\gl_{m|n})$ is the universal enveloping algebra of the general lineal Lie superalgebra $\gl_{m|n}$. It has the structure of supercocomutative Hopf superalgebra, with comultiplication $\Delta(x)=x\ot 1+1\ot x$, counit $\epsilon(x)=0$ and antipode $S(x)=-x$, $x\in\gl_{m|n}$.

We shall work with the following   polynomial superalgebra  
\[\CP^{k|l}_{\, r|s}|_{q=1}:=\cM^{\;k|l}_{m|n}|_{q=1} \ot_{\C} \overbar{\cM}^{\;r|s}_{m|n}|_{q=1}\cong S(\C^{k|l}\ot \C^{m|n}\oplus \C^{r|s}\ot(\C^{m|n})^{\ast}),\]
which is generated by $T_{ai}$ and $\overbar{T}_{bj}$ subject to relations  
\[ 
\begin{aligned}
&T_{ai}^2=\overbar{T}_{bj}^2=0,\quad [a]+[i]=[b]+[j]=\overbar{1},\\
&T_{ai}T_{ck}=(-1)^{([a]+[i])([c]+[k])}T_{ck}T_{ai},\\
&\overbar{T}_{bj}\overbar{T}_{dl}=(-1)^{([b]+[j])([d]+[l])}\overbar{T}_{dl}\overbar{T}_{bj},\\   
&T_{ai}\overbar{T}_{bj}=(-1)^{([a]+[i])([b]+[j])}\overbar{T}_{bj}T_{ai},
\end{aligned}
\]
for all $a,c\in \bI_{k|l},b,d\in \bI_{r|s}$ and $i,j,k,l\in \bI_{m|n}$. 
It was noted in \cite{SZ} that all arguments in \secref{secfunalg} remain valid in classical case, by replacing $q$ by $1$. 
Thus,  $\U(\gl_{k|l})$, $\U(\gl_{r|s})$ and $\U(\gl_{m|n})$ act on these generators  $T_{ai}$ and $\overbar{T}_{bj}$ in the same way as  the quantum case, and hence the elements
\[ X_{ab}=\sum_{i\in \bI_{m|n}} (-1)^{[a]([b]+[i])} T_{ai}\overbar{T}_{bi},\quad \forall a\in \bI_{k|l},b\in \bI_{r|s}  \]
belong to  $\U(\gl_{m|n})$-invariant subalgebra $\CX^{k|l}_{\;r|s}|_{q=1}$ of $\CP^{k|l}_{\, r|s}|_{q=1}$ by \lemref{lemXinv}. The following can be verified directly (compare with \lemref{leminvrel}).

\begin{lem}\label{lemXrelsup}
	The invariants $X_{ab}\in\CX^{k|l}_{\;r|s}|_{q=1} $ satisfy
	\[X_{ab}^2=0,\,[a]+[b]=\bar{1},\quad X_{ab}X_{cd}=(-1)^{([a]+[b])([c]+[d])}X_{cd}X_{ab}. \]
	for any $a,c\in \bI_{k|l},b,d\in \bI_{r|s}$ 
\end{lem}

%Since $\widetilde{\cM}^{k|l}_{r|s}|_{q=1}\cong \cM^{k|l}_{r|s}|_{q=1}$ as superalgebras,  we obtain from \thmref{FFTref} the FFT of invariant theory for $\U(\gl_{m|n})$ as follows.

The following FFT of invariant theory for $\U(\gl_{m|n})$ follows from  \thmref{FFTref}.
\begin{thm} {\rm (FFT for $\U(\gl_{m|n})$)}
	The superalgebra homomorphism 
	\[\Psi^{k|l}_{r|s}: \cM^{k|l}_{r|s}|_{q=1}\rightarrow \CX^{k|l}_{\;r|s}|_{q=1},\quad t_{ab}\mapsto X_{ab} \]
	is surjective. Moreover, $\Psi^{k|l}_{r|s}$ is a $\U(\gl_{k|l})\ot \U(\gl_{r|s})$-module homomorphism.
\end{thm}

Note that we have replaced $\widetilde{\cM}^{k|l}_{r|s}|_{q=1}$ by $\cM^{k|l}_{r|s}|_{q=1}$ since they are isomorphic as associative superalgebras.
%Now by \corref{coroXdec},  we obtain 
%\[ \Ker\, \Psi^{k|l}_{r|s} \cong \bigoplus_{\la_c \subseteq\la\in \La_{k|l}\cap\La_{r|s}} L^{k|l}_{\la}|_{q=1}\ot L^{r|s}_{\la}|_{q=1}, \quad \text{with $\la_c=((n+1)^{m+1})$} \]
%as $\U(\gl_{k|l})\ot\U(\gl_{r|s})$-module. 
Assume $K=\max\{k,r\}, L=\max\{l,s\}$. Let $T_{\la_c}|_{q=1}$ be the subspace of $\cM_{K|L}|_{q=1}$ spanned by the matrix elements of the irreducible $\U(\gl_{K|L})$-module $L_{\la_c}^{K|L}|_{q=1}$. By \thmref{thmSFT}, we obtain that as a two-sided ideal of $\cM_{r|s}^{k|l}|_{q=1}$,
\begin{equation}\label{eqKerglmn}
\Ker\, \Psi^{k|l}_{r|s}=\langle ^{\pi}T_{\la_c}|_{q=1}\rangle^{k|l}_{r|s},
\end{equation}
where $\pi^{k|l}_{\,r|s}$ is the algebra retraction $\pi^{k|l}_{\,r|s}: \cM_{K|L}|_{q=1}\rightarrow \cM^{k|l}_{r|s}|_{q=1}$ such that
\begin{equation}\label{eqpicond}
 \pi^{k|l}_{\,r|s} (t_{ab})=t_{ab},\,  \text{if $a\in \hat{\bI}_{k|l}$ and $b\in \hat{\bI}_{r|s}$}; \ 0, \, \text{otherwise.}
\end{equation}
 Here $ \hat{\bI}_{k|l}, \hat{\bI}_{r|s}$ are defined as in  \eqref{eqhatI} and then $\cM^{k|l}_{r|s}|_{q=1}$ can be embedded into $\cM_{K|L}|_{q=1}$.

It remains to figure out the matrix elements of $L_{\la_c}^{K|L}|_{q=1}$, which span the space $T_{\la_c}|_{q=1}$.

We start by briefly recalling the Schur-Weyl duality of $\gl_{K|L}$ due to \cite{BR} (also see \cite{S2}). 
Let $\{v_i\mid i\in \bI_{K|L}\}$  be a basis  for $\C^{K|L}$   and write $v_I:= v_{i_1}\ot v_{i_2}\ot\dots\ot v_{i_N}\in (\C^{K|L})^{\ot N}$ with $I=(i_1,i_2,\dots, i_k)$.
The elementary transposition $(a,a+1)\in \Sym_N$ acts on $v_{I}$ by permuting $v_{i_a}$ and $v_{i_{a+1}}$ together with a sign $(-1)^{[v_{i_a}][v_{i_{a+1}}]}$. 
This in general induces the action 
\begin{equation}\label{eqsymact}
\sigma.v_{I}=c(I,\sigma^{-1})v_{\sigma I},\quad \text{with $c(I,\sigma)=\pm 1$},\ \forall \sigma\in \Sym_{N},
\end{equation}
where $\sigma$ acts on the sequence $I$ naturally. The Schur-Weyl duality gives the multiplicity-free decomposition of the tensor product
\[ (\C^{K|L})^{\ot N}\cong \bigoplus_{\la\in \La_{K|L}, |\la|=N} S^{\la}\ot L_{\la}^{K|L}|_{q=1}  \]
as $\C\Sym_{N}\ot \U(\gl_{K|L})$-module, where $S^{\la}$ stands for the simple $\C\Sym_{N}$-module associated to $\la$.

We recall some combinatorics on tableaux. A  \emph{standard $\la$-tableau} is obtained by filling the Young diagram of $\la$ with integers $1,\dots,N$ such that these entries increase down columns and along rows.  We denote by $R(\ft)$ and $C(\ft)$ respectively the row and column stabilisers in $\Sym_N$. Define the Young symmetriser on $\ft$ by 
$
y_{\ft}:=\sum_{\sigma\in R(\ft), \tau\in C(\ft)} (-1)^{\ell(\tau)}\sigma\tau.$
Let $I$ be a sequence of elements from $\bI_{K|L}$ of length $N$. A $\ft$-semistandard sequence $I$ is obtained by filling the tableau $\ft$ with elements from $I$, replacing each element $a$ by $i_a\in I$ ($1\leq a\leq N$), in such a way that  the elements of $\ft$ do not decrease from left to right and downward, the even elements strictly increase along columns, while the odd elements strictly increase along rows. The following result is well known \cite{BR,S2}.

\begin{prop}\label{propbasis}
	Let $\la\in \La_{K|L}$ with size $|\la|=N$ and $\ft$ be any fixed standard $\la$-tableau, $y_{\ft} (\C^{K|L})^{\ot N}\cong L_{\la}^{K|L}|_{q=1}$ is a simple $\U(\gl_{K|L})$-module with a basis
	$ \{y_{\ft}v_{I}\mid \text{$I$ is a $\ft$-semistandard sequence}\}.$
\end{prop}

Let $\la\in\La_{K|L}$ with size $|\la|=N$. Given any two sequences $I=(i_1,i_2,\dots,i_N)\in \bI_{K|L}^{\times N}, J=(j_1,j_2,\dots,j_N)\in \bI_{K|L}^{\times N}$, we define an element in $\cM_{K|L}|_{q=1}$ by 
\[ T(I;J)=(-1)^{\alpha(I,J)}\prod_{a=1}^{N}t_{i_aj_a},\quad \alpha(I,J)=\sum_{a>b}[i_a]([i_b]+[j_b]).  \]
Then $(\C^{K|L})^{\ot N} $ carries the  right $\GL(K|L):=\cM_{K|L}|_{q=1}\overbar{\cM}_{K|L}|_{q=1}$-comodule structure  with the following structure map
\[  \de^{\ot N}: (\C^{K|L})^{\ot N} \rightarrow (\C^{K|L})^{\ot N}\ot \GL(K|L),\quad v_J\mapsto \sum_{I\in \bI_{K|L}^{\times N}}v_I\ot T(I;J).   \]
As a direct summand of $(\C^{K|L})^{\ot N}$, $L_{\la}^{K|L}|_{q=1}$ is a $\GL(K|L)$-comodule with structure map
$ \delta(y_{\ft}v_{I})=\sum_{J}y_{\ft}v_{J}\ot a_{J}P_{\ft}(I,J), $
where $I,J$ are $\ft$-semistandard sequences and  $a_J$ is a scalar multiple. The $P_{\ft}(I,J)$'s are matrix elements of $L_{\la}^{K|L}|_{q=1}$ which can be written explicitly as follows:
\begin{equation}\label{eqPpoly}
P_{\ft}(I,J)=\sum_{\sigma\in R(\ft), \tau\in C(\ft) } (-1)^{\ell(\tau)} c(I, (\sigma\tau)^{-1}) T(\sigma\tau I,J),
\end{equation}
where $c(I, (\sigma\tau)^{-1})$ is determined by \eqref{eqsymact}.
Particularly,  $T_{\la_c}|_{q=1}$ is spanned by  elements $P_{\ft}(I,J)$ on a fixed tableau $\ft$ of shape $\la_c$.
We now arrive at the  SFT of invariant theory for $\U(\gl_{m|n})$.

\begin{thm}{\rm (SFT for $\U(\gl_{m|n})$}
	Let  $\Psi^{k|l}_{r|s}: \cM^{k|l}_{r|s}|_{q=1}\rightarrow \CX^{k|l}_{\;r|s}|_{q=1} $ be the surjective superalgebra homomorphism.
	\begin{enumerate}
		\item  As a two-sided ideal of $\cM^{k|l}_{r|s}|_{q=1}$, $\Ker\, \Psi^{k|l}_{r|s}$ is generated by  $P_{\ft}(I,J)$,  where $\ft$ is a fixed standard tableau of shape $\la_{c}=((n+1)^{m+1})$ and $I$ and $J$ are respectively  $\ft$-semistandard sequences with elements from  $\hat{\bI}_{k|l}$ and  $\hat{\bI}_{r|s}$. 
		\item $\Psi^{k|l}_{r|s}$ is an isomorphism if and only if $m\geq \min\{k,r\}$  and $n\geq \min\{r,s\}$.
		In this case, $\CX^{k|l}_{\;r|s}|_{q=1}$ is a  polynomial superalgebra generated by $X_{ab}, a\in \bI_{k|l},b\in \bI_{r|s}$ subject to relations as in \lemref{lemXrelsup}, and it has a PBW basis  $\{X^{\bm}=\prod_{(a,b)}^{\succ}X_{ab}^{m_{ab}}\mid \bm\in \fM^{k|l}_{r|s} \}$.
	\end{enumerate} 
\end{thm}
\begin{proof}
Using \thmref{thmSFT} and \eqref{eqKerglmn}, we only need to make part (1) more explicit. By our preceding discussion, the restriction subspace $^{\pi} T_{\la_c}|_{q=1}$  is spanned by matrix elements of the form $\pi^{k|l}_{\,r|s}(P_{\ft}(I,J))$. However, by \eqref{eqpicond}  $P_{\ft}(I,J)$   will be killed by $\pi^{k|l}_{r|s}$ unless $I$ and $J$ are  $\ft$-semistandard sequences with elements from  $\hat{\bI}_{k|l}$ and  $\hat{\bI}_{r|s}$, thus part (1) is clear.
\end{proof}

\subsection{Invariant theory for $\Uq(\gl_m)$}
This is a special case of $\Uq(\gl_{m|n})$ with $n=0$, for which the FFT of invariant theory was obtained in \cite{LZZ}, but the SFT was unknown previously. 

We shall work with the quantum analogue of polynomial algebra
$ \CP_{k,r}:= \cM_{k,m}\ot_{\fR} \overbar{\cM}_{r,m}$,
where $\cM_{k,m}:=\cM^{\;k|0}_{m|0},   \, \overbar{\cM}_{r,m}:=\overbar{\cM}^{\;r|0}_{m|0}$. By \rmkref{rmkflatdef} $\CP_{k,r}$ is the flat deformation of the symmetric algebra $S(\C^k\ot \C^{m}\oplus \C^{r}\ot (\C^{m})^{\ast} )$, and  it is generated by elements $T_{ai}$ and $\overbar{T}_{bj}$ $(1\leq a\leq k, 1\leq b\leq r, 1\leq i,j\leq m)$  satisfying relations
\eqref{eqRelModAlg} with all the parities of indexes being even and relations with condition $[a]+[b]=\bar{1}$ excluded. 

Let $\CX_{k,r}:=(\CP_{k,r})^{\Uq(\gl_m)}$ be the $\Uq(\gl_m)$-invariant subalgebra of $\CP_{k,r}$. Then by \lemref{lemXinv} the elements 
$ X_{ab}=\sum_{i=1}^{m}T_{ai}\overbar{T}_{bi}$ belong to $\CX_{k,r}$, $ \forall 1\leq a\leq k,1\leq b\leq r.$
It follows from  \lemref{leminvrel} that

\begin{lem}\label{lemXrel} {\rm (\cite[Lemma 6.9]{LZZ})}
	The invariants $X_{ab}\in \CX_{k,r}$ satisfy the following relations
	\begin{align*}
	X_{ac}X_{bc}=&qX_{bc}X_{ac}, &\ & a>b,\\
	X_{ab}X_{ac}=&q^{-1}X_{ac}X_{ab}, &\ &  b>c,\\
	X_{ac}X_{bd}=&X_{bd}X_{ac}, &\ &a>b, c>d,\\
	X_{ac}X_{bd}=&X_{bd}X_{ac}+(q-q^{-1})X_{bc}X_{ad}, &\ &a>b,c<d.
	\end{align*}		
\end{lem}

Let $\widetilde{\cM}_{k,r}:=\widetilde{\cM}^{k|0}_{r|0}$ be the algebra as defined in \defref{defntilM} with the first relation excluded and all parities of indexes being even. Then the $n=0$ case of \thmref{FFTref} leads to the following result. 
\begin{thm}\label{thmFFTglm}
	{\rm (FFT for $\Uq(\gl_{m})$)}
	The algebra homomorphism $\Psi_{k,r}: \widetilde{\cM}_{k,r}\rightarrow \CX_{k,r},\   \widetilde{t}_{ab} \longmapsto X_{ab}$ is surjective. Moreover, $\Psi_{k,r}$ is a $\Uq(\gl_k)\ot \Uq(\gl_r)$-module homomorphism.
\end{thm}
The FFT of invariant theory for $\Uq(\gl_m)$ given by  \cite[Theorem 6.10]{LZZ} follows from the surjectivity of $\Psi_{k,r}$.

Assume that $K=\max\{k,r\}$. Let $\widetilde{T}_{(1^{m+1})}\cong L_{(1^{m+1})}^{K}\ot L_{(1^{m+1})}^{K}$ be the subspace of $\widetilde{\cM}_{K,K}$. It follows from \thmref{thmSFT} that 
\begin{equation}\label{eqKerglm}
\Ker\, \Psi_{k,r}=\langle ^{\pi}\widetilde{T}_{(1^{m+1})}\rangle_{k,r},
\end{equation}
as a two-sided ideal of $\widetilde{\cM}_{k,r}$, where $\pi_{k,r}$ is the algebra retraction $\pi_{k,r}: \widetilde{\cM}_{K,K}\rightarrow \widetilde{\cM}_{k,r}$ such that $\pi_{k,r}(\widetilde{t}_{ab})=\widetilde{t}_{ab}$ if $1\leq a\leq k$ and $1\leq b\leq r$, and 0 otherwise. 

For our purpose, we need to find elements that span the subspace $\widetilde{T}_{(1^{m+1})}\subset \widetilde{\cM}_{K,K}$. This can be done by relating $\widetilde{T}_{(1^{m+1})}$ to the subspace $T_{(1^{m+1})}$, which is spanned by  matrix elements of $L_{(1^{m+1})}^{K}$. We define the \emph{quantum minors} $\Delta(\underline{a},\underline{b})\in\cM_{K,K}$ and  $\widetilde{\Delta}(\underline{a},\underline{c})\in \widetilde{\cM}_{K,K}$ respectively  by  
\begin{equation}\label{eqquanmin1}
\begin{aligned}
\Delta(\underline{a},\underline{b}):=\sum_{\sigma\in \Sym_N} (-q^{-1})^{\ell(\sigma)} t_{a_1, b_{\sigma(1)}}\cdots t_{a_N, b_{\sigma(N)}},\\
\widetilde{\Delta}(\underline{a},\underline{c}):=\sum_{\sigma\in \Sym_N} (-q^{-1})^{\ell(\sigma)} \widetilde{t}_{a_1, c_{\sigma(1)}}\cdots \widetilde{t}_{a_N, c_{\sigma(N)}},  
\end{aligned}
\end{equation}
where $\ell$ is the length function of the symmetric group $\Sym_{N}$,  and $\underline{a},\underline{b}\in \Xi_{N}, \underline{c}\in \Xi_{N}^{>}$ are sequences of length $N$ from the following sets 
\begin{equation}
\begin{aligned}
   \Xi_{N}&:=\{(a_1,a_2,\dots,a_N)\mid 1\leq a_1< a_2< \dots< a_N\leq K \},\\
   \Xi_{N}^{>}&:=\{(c_1,c_2,\dots,c_N)\mid K\geq c_1> c_2> \dots> c_N\geq 1 \}.
\end{aligned}
\end{equation} 

\begin{lem}
The subspace $\widetilde{T}_{(1^{m+1})}$ of  $\widetilde{\cM}_{K,K}$ is spanned by quantum minors  $\widetilde{\Delta}(\underline{a},\underline{c})$ with $\underline{a}\in \Xi_{m+1}, \underline{c}\in \Xi^{>}_{m+1}$.	
\end{lem}
\begin{proof}
   It is easily verified  that $\phi: \cM_{K,K}\rightarrow  \widetilde{\cM}_{K,K}$ is an  associative algebra isomorphism given by $ \phi (t_{ab})=\widetilde{t}_{a,K+1-b}$. This in particular  induces $\phi(T_{(1^{m+1})})=\widetilde{T}_{(1^{m+1})}$, since by \thmref{thmHowe} and \thmref{propHowedec}  both $\cM_{K,K}$ and $\widetilde{\cM}_{K,K}$ admit multiplicity-free decomposition and  $T_{(1^{m+1})}$ and $\widetilde{T}_{(1^{m+1})}$ are both isomorphic to $L_{(1^{m+1})}^K\ot L_{(1^{m+1})}^K$. We only need to prove that $T_{(1^{m+1})}$ is spanned by quantum minors $\Delta(\underline{a},\underline{b})$ with $\underline{a},\underline{b}\in \Xi_{m+1}$. Once this is completed, we deduce that $\widetilde{T}_{(1^{m+1})}$ is spanned by $\widetilde{\Delta}(\underline{a},\underline{c})=\phi(\Delta(\underline{a},\underline{b}))$ with $\underline{c}=(K+1-b_1,\dots,K+1-b_{m+1})\in \Xi_{m+1}^{>}$.
   
   We now claim that quantum minors $\Delta(\underline{a},\underline{b})$ are matrix elements of $L_{m+1}^K$ which span $T_{(1^{m+1})}$. Recall that quantum exterior algebra $\La_q(V^{K})$ is generated by the elements $\xi_a, 1\leq a\leq K$ with relations $ \xi_a^2=0, \forall a$ and  $\xi_a\xi_b=-q^{-1}\xi_b\xi_a,  a>b$.
   The quantum skew  Howe duality in \rmkref{rmkdul} gives rise to the multiplicity-free decomposition  $\La_q(V^K)= \bigoplus_{N=0}^{K} \La_q(V^K)_{N}$ (see also \cite[Theorem 6.16]{LZZ}), 
   where  $\La_q(V^K)_{N}\cong L_{(1^N)}^{K}$ is the irreducible $\Uq(\gl_{K})$-module with the highest weight $\sum_{i=1}^{N}\epsilon_i$, with the basis given by  $ \xi_{\underline{a}}:=\xi_{a_1}\xi_{a_2}\dots \xi_{a_N}, \underline{a}\in \Xi_N$. Observe that $\La_q(V^{K})$ is  a $\CK[\GL_q(K)]=\cM_{K,K}\overbar{\cM}_{K,K}$-comodule with the structure map
   \[ \delta: \La_q(V^{K})\rightarrow \CK[\GL_q(K)]\ot \La_q(V^{K}),\quad \xi_a \mapsto \sum_{b=1}^{K}t_{ab}\ot \xi_b.\]
   It is easy to see that $\delta(\xi_{\underline{a}})=\sum_{\underline{b}\in \Xi_N}\Delta(\underline{a},\underline{b})\ot \xi_{\underline{b}}$ for any $\underline{a}\in \Xi_N$; this implies that $\Delta(\underline{a},\underline{b})$ are matrix elements of $L^{K}_{(1^N)}$. Our claim follows by letting $N=m+1$.   
\end{proof}

\begin{thm}\label{SFTglm}
	{\rm (SFT for $\Uq(\gl_{m})$)}
	Let  $\Psi_{k,r}: \widetilde{\cM}_{k,r}\rightarrow \CX_{k,r}$ be the  surjective algebra homomorphism.
   \begin{enumerate}
   	\item   As a two-sided ideal of $\widetilde{\cM}_{k,r}$, $\Ker\, \Psi_{k,r}$ is generated by quantum minors $\widetilde{\Delta}(\underline{a},\underline{c})$ with
   	\[ 
   	\begin{aligned}
   	\underline{a}&=(a_1,a_2,\dots,a_{m+1}), \quad &1\leq a_1<a_2< \dots< a_{m+1}\leq k,\\
   	\underline{c}&=(c_1,c_2,\dots,c_{m+1}), \quad &r\geq c_1>c_2> \dots>c_{m+1}\geq 1.
   	\end{aligned}
   	\] 
   	\item $\Psi_{k,r}$ is an isomorphism if and only if  $m\geq \min\{k,r\}$. In this case, $\CX_{k,r}$ is a quadratic  algebra generated by $X_{ab}, 1\leq a\leq k, 1\leq b\leq r$ subject to relations  in \lemref{lemXrel}, and it has a PBW basis $\{\prod_{(a,b)}^{\succ}X_{a b}^{m_{ab}}\mid \bm \in \fM^{k|0}_{r|0} \}$.
   \end{enumerate}	
\end{thm}
\begin{proof}
  Using \thmref{thmSFT} and \eqref{eqKerglm}, we only need to show part (1). Now the restriction subspace $^{\pi}T_{(1^{m+1})}$ is spanned by the elements $\pi_{k,r}(\widetilde{\Delta}(\underline{a}, \underline{c}))$ with $\underline{a}\in \Xi_{m+1}, \underline{c}\in \Xi^{>}_{m+1}$. By the definition of $\pi_{k,r}$, part (1) is clear since $\widetilde{\Delta}(\underline{a}, \underline{c})$ will be killed by  $\pi_{k,r}$ unless $\underline{a}, \underline{c}$ satisfy the condition given in the theorem. 
\end{proof}

\iffalse
\subsection{Some remarks}
  One may expect to get a quantum analogue of Sergeev polynomials, which serve as the matrix elements of simple $\Uq(\gl_{K|L})$-module $L_{\la}^{K|L}$ for any $K,L\in \Z_{+}$. This makes it possible to write down the generators of $\Ker \, \Psi^{k|l}_{r|s}$ for the quantum super case. However, it is rather involved to get such quantum analogue. By the Schur-Weyl duality in  \propref{thmSchur}, the tensor power  $(V^{K|L})^{\ot N}$ admits multiplicity-free decomposition with the direct summand $L_{\la}^{K|L}$ ($|\la|=N$) isomorphic to $y_{\ft}(q) (V^{K|L})^{\ot N}$, where $y_{\ft}(q)$ is the quantum analogue of Young symmetriser. We may get  a basis  for $L_{\la}^{K|L}$ as in \propref{propbasis}, but we do not even have a clear formula for the basis elements $y_{\ft}(q)v_{I}$. Another difficulty is $\cM_{K|L}$ in the quantum super case is highly non-commutative, so it does not behave nicely as that in super case. 
\fi

\appendix 
\section{Highest weight representations of $\Uqg$}\label{appA}
\subsection{Basics on  $\gl(m|n)$}\label{secBasicsuper} We collect some well known results on the general linear Lie superalgebra $\gl_{m|n}$ over $\C$. 
This Lie superalgebra 
can be realised  as a set of $(m+n)\times (m+n)$ complex super matrices  equipped with the Lie super bracket as given in Section \ref{sect:Uq} (see \cite{K}). Let $\fh$ be the Cartan subalgebra of $\gl_{m|n}$ consisting of all diagonal matrices and $\{\epsilon_i\}_{1\leq i\leq m+n}$ be the basis for $\fh^{\ast}$ dual to the basis $\{e_{ii}\}_{1\leq i\leq m+n}$ of $\fh$. We denote by $\CE_{m|n}$  the $m+n$-dimensional vector space over $\R$ with a basis $\{\epsilon_i\}_{1\leq i\leq m+n}$. Alternatively, we may use $\delta_{\mu}:=\epsilon_{i+\mu}$ for $1\leq \mu \leq n$.  We endow the basis elements with a total order $\prec$ , which is called \emph{admissible order} if $\epsilon_i\prec \epsilon_{i+1}$ and $\delta_{\mu}\prec \delta_{\mu+1}$ for all $i,\mu$. As an example, we will consider the following two orders
 \begin{equation}\label{eqord}
 \begin{aligned}
 \text{natural order $\lhd$:}\quad &\epsilon_1\lhd\epsilon_2\lhd\cdots\lhd\epsilon_m\lhd \delta_1\lhd\delta_2\lhd\cdots\lhd\delta_n,\\
 \text{filpped order $\LHD$:} \quad &\delta_1\LHD \delta_2\LHD\cdots\LHD \delta_n\LHD \epsilon_1\LHD\epsilon_2\LHD\cdots\LHD \epsilon_m.
 \end{aligned}
 \end{equation}
Fix an admissible order and let $\CE_1\prec \CE_2\prec \cdots\prec \CE_{m+n}$ be the ordered basis of $\CE_{m|n}$. We  define a symmetric non-degenerate bilinear form on $\CE_{m|n}$ by
\begin{equation}\label{eqbilinear}
 (\epsilon_i,\epsilon_j)=(-1)^{\theta}\delta_{ij},\quad (\delta_\mu,\delta_\nu)=-(-1)^{\theta}\delta_{\mu\nu},\quad (\epsilon_i,\delta_\mu)=(\delta_\mu,\epsilon_i)=0,
\end{equation}
where $\theta\in\{0,1\}$ such that $(\CE_{1},\CE_{1})=1$.

The set $\Phi^{\prec}$ of roots of $\gl_{m|n}$ can be  realised as a subset of $\CE_{m|n}$ with an admissible order $\prec$. Explicitly, each choice of a Borel subalgebra of $\gl_{m|n}$ corresponds to a choice of positive roots  $(\Phi^{\prec})^{+}$, and hence a fundamental system $\Pi^{\prec}=\{\alpha_1,\alpha_2,\dots, \alpha_{m+n-1}\}$ of simple roots, where $\alpha_i=\CE_i-\CE_{i+1}$ for $1\leq i< m+n-1$. The Weyl group conjugacy classes of  Borel subalgebras correspond bijectively to the admissible ordered bases of $\CE_{m|n}$.

Let $\fb^{\prec}$ be the Borel subalgebra of $\gl_{m|n}$ corresponding to the  admissible order $\prec$. In the matrix form, we have $\fb^{\prec}=\Span_{\C}\{e_{ij}\mid \CE_i\preceq \CE_j\}$.  Given an $(m,n)$-hook partition $\la$, we write $M(\la)^{\prec}= \U(\gl_{m|n})\ot_{\U(\fb^{\prec})} \C^{\prec}_{\la}$ for the Verma module of the $\fb^{\prec}$-highest weight $\la^{\prec}$, where $\C^{\prec}_{\la}$ is the one-dimensional $\fb^{\prec}$-module of weight $\la^{\prec}$.  The Verma module has a unique quotient $L^{\prec}_{\la}$ with the highest weight $\la^{\prec}$, which appears in the tensor product decomposition of $(V^{m|n})^{\ot N}$ for some integer $N\in \Z_{+}$. In particular, under the natural order $\lhd$,  the highest weight of $L_{\la}^{\lhd}$ is 
\[\la^{\lhd}=\la^{\natural}=\la_1\epsilon_1+\dots \la_{m}\epsilon_m+ \langle \la_1^{\prime}-m\rangle \delta_1+\dots+\langle \la_n^{\prime}-m\rangle\delta_n  \]
as defined in \eqref{eqlanat}. 

We use odd reflections to translate between various highest weights  of irreducible modules arising from different admissible orders. Two ordered bases for  $\CE_{m|n}$ are called \emph{adjacent} if they are identical except for a switch of a neighbouring pair $\CE_{i}$ and $\CE_{i+1}$ with $\CE_{i}\in \{\epsilon_1,\epsilon_2,\dots,\epsilon_m\}$ and $\CE_{i+1}\in \{\delta_1,\delta_2,\dots,\delta_n\}$.
For example, $\{\epsilon_1\prec \epsilon_2\prec \delta_1\prec \delta_2\}$ and $\{\epsilon_1\prec^{\prime} \delta_1\prec^{\prime}\epsilon_2 \prec^{\prime} \delta_2\}$ are adjacent ordered bases for $\CE_{2|2}$. For two adjacent ordered bases which differ at $i$ and $i+1$, the corresponding positive  root systems are related by \[(\Phi^{\prec^{\prime}})^+=\{(\Phi^{\prec})^{+}\backslash \{\alpha_i\}\} \cup \{-\alpha_i\} \quad \text{with}\quad \alpha_i= \CE_{i}-\CE_{i+1}.\] It was discovered by Serganova that
\[ 
\la^{\prec^{\prime}}=
  \begin{cases}
    \la^{\prec}, & (\la,\alpha_i)=0,\\
    \la^{\prec}-\alpha_i, & (\la, \alpha_i)\neq 0.
  \end{cases}
\]
In particular, we can apply a sequence of odd reflections to the natural order $\lhd$ and then obtain the flipped order $\LHD$, yielding
\begin{prop}\label{propLHDwt}
	The $\fb^{\LHD}$-highest weight $\la^{\LHD}$ for $L_{\la}^{\LHD}$ is given by
	\[ \la^{\LHD}=\la_1^{\prime}\delta_1+\dots +\la_n^{\prime}\delta_n+ \langle\la_1-n\rangle \epsilon_1+\dots+ \langle\la_m-n\rangle \epsilon_m.
    \]
\end{prop}

We use the pair $(\gl_{m|n},\Pi^{\prec})$ to denote the general Lie superalgebra $\gl_{m|n}$ furnished with fundamental system $\Pi^{\prec}$ arising from admissible order $\prec$.  Then it is worthy to note that $(\gl_{m|n},\Pi^{\LHD})=(\gl_{n|m},\Pi^{\lhd})$, and \propref{propLHDwt} implies that the irreducible $(\gl_{m|n},\Pi^{\lhd})$-module $L^{\lhd}_{\la}$, viewed as $(\gl_{n|m},\Pi^{\lhd})$-module, has the highest weight $(\la^{\prime})^{\natural}=\la^{\LHD}$ as given in \propref{propLHDwt}.

For our purpose,  we shall determine the lowest weight of $L_{\la}^{\lhd}$. We consider the following fundamental system $\Pi^{\rm L}$
\[\Pi^{\rm L}:=\{\delta_n-\delta_{n-1}, \delta_{n-1}-\delta_{n-2},\dots,\delta_1-\epsilon_m, \epsilon_{m}-\epsilon_{m-1},\dots,\epsilon_2-\epsilon_1\},  \]
which is conjugate to $\Pi^{\LHD}$ under Weyl group of $\gl_{m|n}$. Precisely, let $w_{k}$ be the longest element in the symmetric group $\Sym_k$ such that $w_k(i)=k+1-i$, $1\leq i\leq k$ for any $k\in \Z_{+}$. Then $\Pi^{\rm L}=w_mw_n\Pi^{\LHD}$, where $w_m$ acts on all $\epsilon_i$ and $w_n$ acts on all $\delta_\mu$ by permuting subscripts. The corresponding Borel subalgebra $\fb^{\rm L}$, which is conjugate to $\fb^{\LHD}$, consists of all lower triangular $(m+n)\times (m+n)$-super matrices. Therefore, the $\fb^{\lhd}$-lowest weight of $L_{\la}^{\lhd}$ is indeed the $\fb^{\rm L}$-highest weight, where the latter can be obtained by the group action of $w_mw_n$ on the $\fb^{\LHD}$-highest weight.  By \propref{propLHDwt}, we immediately have
\begin{prop}\cite{VZ}
The $\fb^{\lhd}$-lowest weight, denoted by $\overbar{\la}^{\natural}$,  of $L_{\la}^{\lhd}$ is given by
\[
\begin{aligned}
  \overbar{\la}^{\natural}&= \langle\la_m-n\rangle \epsilon_1+\dots+ \langle\la_1-n\rangle  \epsilon_m+\la_n^{\prime}\delta_1+\dots +\la_1^{\prime}\delta_n\\
   &=(\langle\la_m-n\rangle,\langle\la_{m-1}-n\rangle,\dots,\langle\la_1-n\rangle; \la_n^{\prime},\la_{n-1}^{\prime},\dots,\la_1^{\prime}).
\end{aligned}
 \]
\end{prop}

\subsection{The quantum supergroup $\U_q(\gl_{m|n},\Pi^{\prec})$}
Fix an admissible order $\prec$, and assume the fundamental system  $\Pi^{\prec}=\{\alpha_1,\alpha_2, \dots, \alpha_{m+n-1}\}$ for $\gl_{m|n}$, where $\alpha_i=\CE_i-\CE_{i+1}$ for $1\leq i< m+n-1$. Let  $\Theta\subseteq\{1, 2, \dots, m+n-1\}$ be the labelling set of the odd simple roots, i.e.,  $\{\alpha_s\mid s\in\Theta\}$ is the subset of $\Pi^{\prec}$ consisting of the odd simple roots. We stick to the nondegenerate bilinear  form in \eqref{eqbilinear}, and  define  the
Cartan matrix of $\gl_{m|n}$ associated to $\Pi^{\prec}$  by
\[
A=(a_{ij}) \quad\text{with}\quad
a_{ij}=\begin{cases}\dfrac{2(\alpha_i,\alpha_j)}{(\alpha_i,\alpha_i)},&\mbox{if}~(\alpha_i,\alpha_i)\neq0,\\
(\alpha_i,\alpha_j),&\mbox{if}~(\alpha_i,\alpha_i)=0.
\end{cases}
\]

\begin{defn}\label{defi:quantised}
	 The quantum general linear supergroup $\U_q(\gl_{m|n},\Pi^{\prec})$  with the fundamental system $\Pi^{\prec}$ is the unital associative $\CK$-superalgebra on generators $e_i,f_i$ $(i=1, 2, \dots, m+n-1)$ and $K_a^{\pm1}$ $(a=1, 2, \dots, m+n)$,  where $e_s,f_s, s\in\Theta$ are odd and the rest are even, subject to the relations
	 \begin{enumerate}[itemsep=1.5mm]
	  \item [(R1)] $K_a K_a^{-1}=K_a^{-1}K_a=1,\quad K_a K_b=K_b K_a;$
	  \item [(R2)] $K_ae_iK_a^{-1}=q^{(\CE_a, \CE_i-\CE_{i+1})}e_j,\quad   K_a f_iK_a^{-1}=q^{-(\CE_a,\CE_i-\CE_{i+1})}f_i;$
	  \item [(R3)] 	$e_if_j-(-1)^{[e_i][f_j]}f_je_i=\de_{ij}\dfrac{K_i K_{i+1}^{-1}-K_{i+1} K_i^{-1}}{q_i-q_i^{-1}}$, with $q_i=q^{(\CE_i,\CE_i)}$;
	  \item [(R4)] Serre relations:
	      \[\begin{aligned}
	      &(e_s)^2=(f_s)^2=0, \quad &\text{if $a_{ss}=0$},\\
	      &\Ad_{e_i}^{1-a_{ij}}(e_j)=\Ad_{f_i}^{1-a_{ij}}(f_j)=0,\quad &\text{if $a_{ii}\neq0,i\neq j$;}
	      \end{aligned}
	      \] 
	   \item [(R5)] High order Serre relations:
	     \[\Ad_{e_s}\Ad_{e_{s-1}}\Ad_{e_s}(e_{s+1})=0,\quad
	     \Ad_{f_s}\Ad_{f_{s-1}}\Ad_{f_s}(f_{s+1})=0, \quad\forall s\in \Theta,\]
	   where $\Ad_{e_i}(x)$ and $\Ad_{f_i}(x)$ are defined respectively by
	   \[
	   \begin{aligned}
	   \Ad_{e_i}(x)=e_ix-(-1)^{[e_i][x]}K_i K_{i+1}^{-1} x K_{i+1}K_i ^{-1}e_i, \\
	   \Ad_{f_i}(x)=f_ix-(-1)^{[f_i][x]} K_{i+1}K_i ^{-1} x K_i K_{i+1}^{-1} f_i. 
	   \end{aligned}
	   \]
	 \end{enumerate}
\end{defn}	

We refer to \cite{XZ,Y} for more details on  $\U_q(\gl_{m|n},\Pi^{\prec})$. In particular, \defref{defnUglmn} is the special case with fundamental system $\Pi^{\lhd}$ and $\Theta=\{m\}, e_i=E_{i,i+1},f_i=F_{i+1,i}$. Clearly, $\Uq(\gl_{m|n},\Pi^{\LHD})=\Uq(\gl_{n|m},\Pi^{\lhd})$.

When $q$ is an indeterminate, the finite dimensional irreducible representations of  $\Uqg$ are similar to those of  $\gl(m|n)$ \cite{Z93}. Here we only mention the results in  \secref{secBasicsuper} using quantum language.  We write $\Uqg=\Uq(\gl_{m|n},\Pi^{\lhd})$ for short.

\begin{prop}\label{proplowwht}
	Let $\la\in\La_{m|n}$ be an $(m,n)$-hook partition and $L_{\la}^{m|n}$ be the irreducible $\Uqg$-module with the highest weight $\la^{\natural}\in \La_{m|n}^{\natural}$ for any $m,n\in \Z_{+}$.
	\begin{enumerate}
		\item The $\Uq(\gl_{m|n})$-module  $L_{\la}^{m|n}$ has the highest weight $(\la^{\prime})^{\natural}$ when it is viewed as $\Uq(\gl_{n|m})$-module, that is, $L_{\la}^{m|n}\cong L_{\la^{\prime}}^{n|m}$ as $\Uq(\gl_{n|m})$-modules.
		\item The lowest weight $\overbar{\la}^{\natural}$ of  $\Uqg$-module $L_{\la}^{m|n}$ is given by the formula
		\[ \overbar{\la}^{\natural}=(\langle\la_m-n\rangle,\langle\la_{m-1}-n\rangle,\dots,\langle\la_1-n\rangle; \la_n^{\prime},\la_{n-1}^{\prime},\dots,\la_1^{\prime}).  \]
	\end{enumerate}
\end{prop}

\medskip
%\noindent {\bf Acknowledgements:}
\section*{Acknowledgement}
I wish to thank Professors Gus Lehrer and Ruibin Zhang for advices and help during the course of this work. This work was supported at different stages by student stipends from the China Scholarship Council and the Australian Research Council.

%\section{Proofs of two isomorphisms}\label{appB}

\end{document}